\documentclass[10pt,twoside,a4paper,english]{article}

\title{On the rigid-lid approximation for two shallow layers\\ of immiscible fluids with small density contrast}
\newcommand{\shorttitle}{On the rigid-lid approximation for two immiscible fluids with small density contrast}
\author{Vincent Duch\^ene}
\date{\today}

\usepackage{babel} 
\usepackage[utf8]{inputenc} 
\usepackage[T1]{fontenc} 
\usepackage[ec]{aeguill} 
\usepackage{amsmath, amsthm} 
\usepackage{amsfonts,amssymb}
\usepackage{float} 
\usepackage[pdftex]{graphicx}
\usepackage{fancyhdr} 
\usepackage{url} 
\usepackage{hyperref} 
\usepackage[hang,centerlast]{subfigure} 

\numberwithin{equation}{section}

\newcommand{\RR}{\mathbb{R}}
\newcommand{\NN}{\mathbb{N}}
\renewcommand{\t}{\tilde}
\renewcommand{\u}{\underline}
\renewcommand{\P}{\mathcal{P}}
\newcommand{\e}{\mathbf{e}}
\newcommand{\x}{\mathbf{x}}
\newcommand{\R}{\mathcal{R}}
\renewcommand{\O}{\mathcal{O}}
\renewcommand{\r}{\varrho}

\newcommand{\nn}{\nonumber}
\newcommand{\id}[1]{\left\vert_{\scriptstyle #1}\right.}
\newcommand{\Vr}{V_{\rm rem}}
\newcommand{\Vs}{V^{s}_{\rm cor}}
\newcommand{\Vf}{V^{f}_{\rm cor}}
\newcommand{\Vapp}{V_{\rm app}}
\newcommand{\Uapp}{U_{\rm app}}
\newcommand{\Vrl}{V_{\rm RL}}
\newcommand{\W}{V}
\DeclareMathOperator{\Id}{Id}

\DeclareMathOperator*{\esssup}{ess\,sup}

\newtheorem{Theorem}{Theorem}[section]
\newtheorem{Proposition}[Theorem]{Proposition}
\newtheorem{Definition}[Theorem]{Definition}
\newtheorem{Lemma}[Theorem]{Lemma}
\newtheorem{Corollary}[Theorem]{Corollary}
\newtheorem{Remark}[Theorem]{Remark}

\fancypagestyle{thestyle}{%
	\fancyfoot{}
	\fancyhead[CE]{\shorttitle} 
	\fancyhead[CO]{\scriptsize Vincent Duch\^ene} 
	\fancyhead[RE]{} 
	\fancyhead[LO]{\scriptsize\today} 
	\fancyhead[LE,RO]{\scriptsize\thepage}
	}
\fancypagestyle{plain}{\fancyhead{} } 

\begin{document}

\maketitle

\begin{abstract}
The rigid-lid approximation is a commonly used simplification in the study of density-stratified fluids in oceanography. Roughly speaking, one assumes that the displacements of the surface are negligible compared with interface displacements. In this paper, we offer a rigorous justification of this approximation in the case of two shallow layers of immiscible fluids with constant and quasi-equal mass density. More precisely, we control the difference between the solutions of the Cauchy problem predicted by the shallow-water (Saint-Venant) system in the rigid-lid and free-surface configuration. We show that in the limit of small density contrast, the flow may be accurately described as the superposition of a baroclinic (or slow) mode, which is well predicted by the rigid-lid approximation; and a barotropic (or fast) mode, whose initial smallness persists for large time. We also describe explicitly the first-order behavior of the deformation of the surface, and discuss the case of non-small initial barotropic mode.
\end{abstract}

\section{Introduction} 

\subsection{Motivation} The mass density of water in the ocean is not constant, due to variations of temperature and salinity. As a matter of fact, one typically observes a sharp separation between a layer of warm, relatively fresh water above a layer of cold, more salted water. The interface between these two layers may experience great deformations that are mostly invisible at the surface, but account for important oceanographic features, such as internal solitary waves or the dead-water phenomenon (see, {\em e.g.},~\cite{Gill82,Jackson04,HelfrichMelville06} and references therein). The study of these internal waves has attracted a considerable amount of attention in the past decades, and lead to a vast collection of various models. In order to simplify the setting, two approximations are commonly used in the literature, namely the rigid-lid and Boussinesq approximations. Roughly speaking, the rigid-lid approximation consists in neglecting the surface displacements compared to interface displacements, while the Boussinesq approximation relies on the assumption that the density differences between the two layers is small. Acknowledgedly, these two assumptions are related: a fixed amount of energy generates a much smaller displacement on the air/water interface than on the fresh/salted water interface, because the ratio of mass densities across the interface is negligible in the former case when compared to the latter.

The ambition of this article is to offer a rigorous justification of the above presumption. We restrict ourselves to one of the simplest possible setting, that is two infinite, two-dimensional layers of immiscible fluids with constant density, above a flat bottom. Moreover, we consider sufficiently shallow layers so that the hydrostatic approximation is valid; thus we study the so-called Saint-Venant~\cite{Saint-Venant71}, or shallow-water equations. 
Even in that much simplified setting, we will come across serious difficulties, which come from the fact that the typical surface wave speed, as predicted by the linearized system, is much greater than the typical interface wave speed, in particular in the limit of vanishing density contrast. 
Thus within the terms neglected in the rigid-lid approximation are contributions whose velocity blows up in the limit we consider. As a matter of fact, even the well-posedness of the Cauchy problem for the Saint-Venant system in the free-surface configuration on a relevant time scale ({\em i.e.} non-vanishing with the density contrast) is challenging. 

To our knowledge, very few works are concerned with the validity of the aforementioned approximations, despite the early concerns expressed by Long~\cite{Long65} and Benjamin~\cite{Benjamin66}. Grimshaw, Pelinovsky, Poloukhina~\cite{GrimshawPelinovskyPoloukhina02}, Craig, Guyenne, Kalisch~\cite{CraigGuyenneKalisch05}, Craig, Guyenne, Sulem~\cite{CraigGuyenneSulem10} and the author~\cite{Duchene11a} derived and compared asymptotic models in both the rigid-lid and free-surface settings. However, they do not directly compare solutions of the two models with corresponding initial data, but rather parameters of their models, or explicit solutions (solitary waves). Moreover, and maybe more importantly, their analysis is restricted to weakly nonlinear waves, so that the deformation of both the surface and interface is assumed to be small. Recently, Leonardi~\cite{Leonardi11} studied in much details the validity of the rigid-lid approximation in a linearized setting, and without explicitly looking at the limit of small density differences. Conversely, our study accounts for fully nonlinear waves, and directly compares the solutions predicted by the rigid-lid and free-surface systems, in the limit of vanishing density contrast.

\subsection{Presentation of the models, and main result}\label{S.models}
In this section, we present the two models we study, namely the shallow-water (or Saint-Venant) systems in the free-surface and rigid-lid configuration; see Figure~\ref{F.Sketch}. We briefly describe some early properties of these models, and state our main result in Theorem~\ref{T.mr}. There follows an outline of the present paper, and some notations used therein.

\begin{figure}[htb] 
 \subfigure[Free-surface situation]{
\hfill\includegraphics[width=0.44\textwidth]{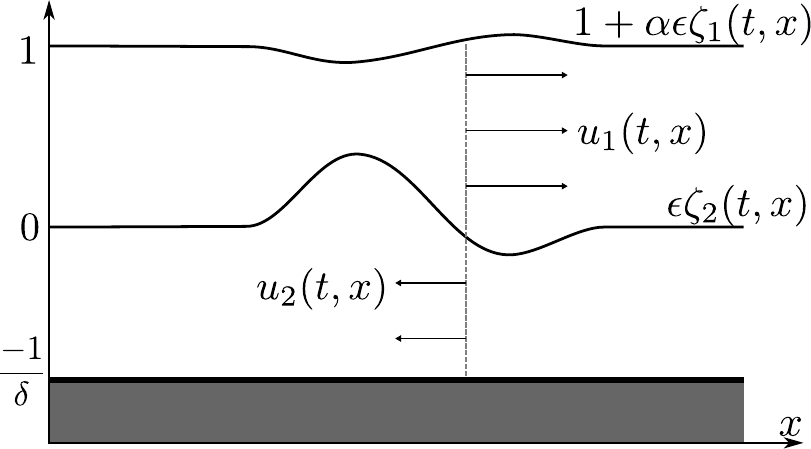}
\label{F.SketchFS}
}\hfill
 \subfigure[Rigid-lid situation]{
\includegraphics[width=0.44\textwidth]{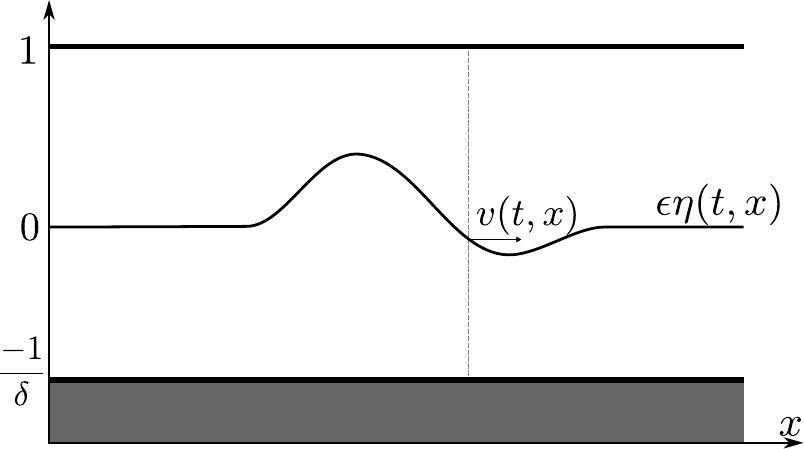}
\label{F.SketchRL}
}\hfill
\caption{Sketch of the domain in the two different situations}
\label{F.Sketch}
\end{figure}

\paragraph{The free-surface system.} Let us first introduce the shallow-water model with free surface, that we simply refer to as {\em free-surface system}.
\begin{equation}\label{FS}\left\{ \begin{array}{l}
\alpha\partial_t \zeta_1 + \partial_x(h_1u_1) +\partial_x(h_2u_2) = 0, \\
\partial_t \zeta_2 +\partial_x(h_2u_2) = 0, \\
\partial_t u_1 + \alpha \frac{\delta+\gamma}{1-\gamma}\partial_x\zeta_1 +\frac{\epsilon}{2}\partial_x\left(|u_1|^2\right) = 0, \\
\partial_t u_2 + (\delta+\gamma)\partial_x \zeta_2 + \gamma \alpha\frac{\delta+\gamma}{1-\gamma}\partial_x \zeta_1 +\frac{\epsilon}{2}\partial_x\left( |u_2|^2 \right) =0,
\end{array} \right.\end{equation}
where we denote $h_1=1+\epsilon \alpha\zeta_1-\epsilon\zeta_2$, and $h_2=\frac1\delta+\epsilon\zeta_2$.

This system has been obtained\footnote{The models presented in these works are not limited to flat bottom or horizontal dimension $d=1$. They present different constants in the velocity equations. This is due to a different choice of scaling in the non-dimensionalizing step. We chose our scaling in order to set the typical velocity of the internal wave (obtained by solving explicitly the linear system, {\em i.e.} setting $\alpha=\epsilon=0$) as $c_0=\pm 1$, consistently with the rigid-lid system~\eqref{RL}.} in~\cite{ChoiCamassa96,CraigGuyenneKalisch05}, and justified in~\cite{Duchene10} as an asymptotic model (in the shallow-water regime) for a system of two layers of immiscible, homogeneous, ideal, incompressible fluid under the only influence of gravity (the so-called full Euler system). It describes the evolution of the deformation of the surface, $\zeta_1$, the interface, $\zeta_2$, and the horizontal velocity of the fluid in the upper (resp. lower) layer, $u_1$ (resp. $u_2$).\footnote{The Saint-Venant model is usually derived using the so-called hydrostatic approximation. Equivalently, one may assume that the horizontal scale is large compared with the vertical scale, so that the horizontal velocity field is accurately described as constant throughout the depth of each layer of fluid.} More precisely, the two layers are assumed to be connected, infinite in the horizontal dimension $x\in\RR$, delimited below by a flat bottom, and by the graph of the functions $\zeta_1(t,x)$, $\zeta_2(t,x)$ (see Figure~\ref{F.SketchFS}). 

The parameters $\alpha,\delta,\gamma,\epsilon$ are dimensionless parameters that describe characteristics of the flow. More precisely:
\begin{itemize}
\item[$\delta$] represents the ratio of the upper-layer to the lower-layer depth;
\item[$\gamma$] represents the ratio of the mass density between the two fluids;
\item[$\epsilon$] represents the maximal deformation of the interface, divided by the upper-layer depth;
\item[$\alpha$] represents the ratio of the maximal deformation of the surface to the one of the interface.
\end{itemize}
In particular, $h_1$ denotes the depth of the upper layer, and $h_2$ the depth of the lower layer.
\begin{Remark}
Another dimensionless parameter plays an important role, but is not visible here, although it is essential for the construction and relevance of the shallow-water models. If we denote by $\mu$ the ratio of the depth of the two layers to a characteristic horizontal length, then one assumes $\mu\ll 1$, and all terms of size $\O(\mu^2)$ are neglected in~\eqref{FS}.
\end{Remark}
An additional dimensionless parameter is ubiquitous in the present work, and obtained as a combination of the aforementioned parameters. It turns out to be convenient to express the assumption that the density contrast between the two fluids is small with 
\[ \r \ \ll \ 1 \qquad ;\qquad \r \ \equiv \ \sqrt{\frac{1-\gamma}{\gamma+\delta}} \ . \]

We conclude the presentation of the free-surface system by mentioning that system~\eqref{FS} is obviously a system of four conservation laws, but also induces at least two other conserved quantities. Indeed, as noticed in~\cite{BarrosGavrilyukTeshukov07}, after manipulating the equations, one may obtain:
\begin{itemize}
\item Conservation of horizontal momentum:
\[ \partial_t (\gamma h_1u_1+ h_2 u_2)+\partial_x p +\partial_x (\gamma h_1|u_1|^2 + h_2|u_2|^2) \ = \ 0,\]
where $p$ is the ``pressure'': $p=\frac12\left(\gamma\frac{\delta+\gamma}{1-\gamma}(h_1+h_2)^2+(\gamma+\delta)h_2^2\right)$.
\item Conservation of energy:
\[ \partial_t E +\partial_x \left( \frac12 (\gamma h_1|u_1|^2 u_1 + h_2|u_2|^2 u_2)
+\gamma h_1^2u_1 + h_2^2u_2 + \gamma h_1h_2(u_1 + u_2) \right) \ = \ 0,\]
where we denote $E\equiv \frac12\gamma h_1|u_1|^2+ \frac12h_2 |u_2|^2+ p $.
\end{itemize}

\paragraph{The rigid-lid system.}
The model corresponding to~\eqref{FS} in the rigid-lid configuration, that we refer to as {\em rigid-lid system}, is
\begin{equation}\label{RL}
\left\{ \begin{array}{l}
\displaystyle\partial_{ t}{\eta} \ + \ \partial_x \Big(\frac{h_1h_2}{h_1+\gamma h_2}v\Big) \ =\ 0, \\ \\
\displaystyle\partial_{ t}v \ + \ (\gamma+\delta)\partial_x{\eta} \ + \ \frac{\epsilon}{2} \partial_x\Big(\frac{h_1^2-\gamma h_2^2}{(h_1+\gamma h_2)^2}|v|^2\Big) \ = \ 0 \ .
\end{array} 
\right. 
\end{equation}
Here, $\eta$ represents the deformation of the interface, and $v$ the shear velocity, namely $v=u_2-\gamma u_1$; see below and Figure~\ref{F.SketchRL}. Again, $h_1,h_2$ denote the depth of the upper (resp. lower) layers, thus $h_1=1-\epsilon\eta$ and $h_2=1/\delta+\epsilon\eta$. Parameters $\gamma,\delta,\epsilon$ are defined as previously.

System~\eqref{RL} has been justified as an asymptotic model in the shallow-water regime in~\cite{BonaLannesSaut08},\footnote{The justification provided in~\cite{BonaLannesSaut08} ---as well as in~\cite{Duchene10} in the free-surface configuration--- is in the sense of consistency: sufficiently smooth solutions of the full Euler system satisfy the equations of~\eqref{RL} up to small, {\em i.e.} $\O(\mu^2)$, remainder terms. The rigorous, full justification follows from the well-posedness of both the full Euler system and the shallow-water model, as well as a stability result which allows to compare the solutions of both systems with corresponding initial data on the relevant time-scale. In the rigid-lid situation, Lannes~\cite{Lannes13} recently solved the difficult problem of the well-posedness of the full Euler system, consequently completing the full justification of~\eqref{RL}; see~\cite[Theorem~7]{Lannes13}. No such result is available in the bi-fluidic free-surface configuration.} starting from the full Euler system in the rigid-lid configuration. 
Let us show how to {\em formally} recover~\eqref{RL} from~\eqref{FS}. Set $\zeta_1\equiv 0$ (or, equivalently, $\alpha=0$) in~\eqref{FS}. It follows in particular from the first equation that
\begin{equation} \partial_x (h_1 u_1)+\partial_x(h_2 u_2) \ = \ 0 . \label{momentum}\end{equation}
Since $h_1 u_1$ and $h_2 u_2$ are scalar functions vanishing at infinity, we deduce the identity $h_1 u_1=-h_2 u_2$. Thus, when we define $v\equiv u_2-\gamma u_1$, one obtains
\begin{equation} u_1\equiv \frac{-h_2 v}{h_1+\gamma h_2} \quad \text{ and }\quad u_2\equiv \frac{h_1 v}{h_1+\gamma h_2}.\label{VtoU-intro}\end{equation}
It is now clear that the second equation, and a linear combination of the last two equations of~\eqref{FS} yield~\eqref{RL} (with $\eta\equiv\zeta_2$). We aim at giving a rigorous confirmation of the above calculations.

\paragraph{Main result.} We state here the main result of the present work.
\begin{Theorem} \label{T.mr}
Let $s\geq s_0+1$, $s_0>1/2$, and $\delta_{\min},\delta_{\max},\gamma_{\min}>0$. Consider $(\alpha,\delta,\epsilon,\gamma)\in\P$, with
\[
\P \ \equiv \ \big\{ (\alpha,\delta,\epsilon,\gamma),\ 0\ \leq \ \alpha\ \leq\ 1, \quad \delta_{\min}\ \leq\ \delta\ \leq \ \delta_{\max}, \quad 0 \ < \ \epsilon \ \leq\ 1,\quad \gamma_{\min}\ \leq \ \gamma\ <\ 1 \ \big\}.\]
Let $\zeta_1^0,\zeta_2^0,u_1^0,u_2^0\in H^{s+1}(\RR)$ satisfy the following hypotheses:
\begin{equation}\label{condE0}
\big\vert \zeta_2^0 \big\vert_{H^{s+1}}\ + \ \big\vert u_2^0 - \gamma u_1^0 \big\vert_{H^{s+1}} \ \leq \ M \ \quad \text{ and } \quad 
 \frac{\alpha}{\r} \big\vert \zeta_1^0 \big\vert_{H^{s+1}} \ + \ \big\vert \gamma h_1 u_1^0\ + \ h_2 u_2^0 \big\vert_{H^{s+1}} \ \leq \ M\ \r \quad 
 \end{equation}
 as well as (denoting $h_1^0\equiv 1+\epsilon \alpha \zeta_1^0-\epsilon \zeta_2^0$ and $h_2^0\equiv \delta^{-1}+\epsilon \zeta_2^0$)
 \begin{multline}\label{condH0}
 \forall x\in\RR , \quad \min\Big\{h_1^0(x) \ ; \ h_2^0(x) -\epsilon^2 \frac{ |u_2^0(x)-u_1^0(x)|^2}{\gamma+\delta} \ ; \\ (h_1^0(x)+\gamma h_2^0(x))^3-\epsilon^2\frac{\gamma(1+\delta^{-1})^2|u_2^0(x)-\gamma u_1^0(x)|^2}{\gamma+\delta} \Big\} \ \geq\  h_{0}>0 ,
 \end{multline}
 where $0<h_0,M<\infty$ are fixed.

Then there exist $T^{-1},C$, positive, depending only and non-decreasingly on $M,h_0^{-1},\delta_{\min}^{-1},\delta_{\max}$, $\gamma_{\min}^{-1}$ and $\frac1{s_0-\frac12}$, such that
the following holds.
\begin{enumerate}
\item There exists a unique solution, $(\eta,v)\in C([0,T/(\epsilon M)];H^{s+1}(\RR)^2)\cap C^1([0,T/(\epsilon M)];H^{s}(\RR)^2)$ to~\eqref{RL}, with initial data $(\eta\id{t=0}=\zeta_2^0$, $v\id{t=0}=u_2^0-\gamma u_1^0)$.
\item There exists a unique solution, $(\zeta_1,\zeta_2,u_1,u_2)\in C([0,T_{\max});H^{s+1}(\RR)^4)\cap C^1([0,T_{\max});H^{s}(\RR)^4)$ to~\eqref{FS}, with initial data $(\zeta_1^0,\zeta_2^0,u_1^0,u_2^0)$, and $T_{\max}\geq T/\max\{\epsilon M,\r\}$.
\item One has, for any $0\leq t\leq T/\max\{\epsilon M,\r\}$,
\[ \frac{\alpha}{\r} \big\Vert \zeta_1 \big\Vert_{L^\infty([0,t];H^s)} \ + \ \big\Vert \gamma h_1 u_1 +  h_2 u_2 \big\Vert_{L^\infty([0,t];H^s)} \ \leq \ C \ M\ \r,\]
and
\[ \big\Vert \eta-\zeta_2 \big\Vert_{L^\infty([0,t];H^s)} +\big\Vert v-(u_2-\gamma u_1) \big\Vert_{L^\infty([0,t];H^s)} \ \leq \ C\ M\ \r.\]
\end{enumerate}
\end{Theorem}
\begin{Remark}\label{R.epsilon}
The restriction on the maximal time of existence for the solution of the free-surface system, $T_{\max}\geq T/\max\{\epsilon M,\r\}$, as opposed to the classical $T_{\max}\geq T/(\epsilon M)$, is purely technical, and does not reveal any limitation that would appear in the weakly non-linear case, $\epsilon M=\O(\r)$. On the contrary, we know that in the latter case (see Proposition~\ref{P.WPFS} and Remark~\ref{R.SmallInitialData}), the system~\eqref{FS} is well-posed over time $T_{\max}\gtrsim (\epsilon M)^{-1}$, without the additional condition in~\eqref{condE0}. Moreover, it would not be difficult to obtain an asymptotic description of the solution similar to the one obtained by the author in~\cite{Duchene11a} (without the dispersion terms), namely that the flow may be accurately approximated as a superposition of four independent waves; each driven by an inviscid Burgers' equation. The solution of the rigid-lid system~\eqref{RL} complies to similar description (with only two counter-propagating waves), thus the two solutions are easily compared. We present in Section~\ref{S.discussion} a similar decomposition of the flow allowing stronger nonlinearities; see in particular Theorem~\ref{T.mr2} and Proposition~\ref{P.ConsIP}. 

In order to acknowledge the fact that we are interested in strong nonlinearities, and to ease the reading, {\em we set $\epsilon\equiv 1$ in the following}.
\end{Remark}
\begin{Remark}\label{R.alpha}
The factor $\frac\alpha\r$ in front of $\zeta_1$ is natural in our context. Indeed, one easily deduces from the aforementioned conservation of energy for~\eqref{FS} that
\[ \int_\RR E(x)-E(\infty)\ dx \ \approx\ \frac{\gamma}{\r^2} \big\vert \alpha \zeta_1 \big\vert_{L^2}^2 \ + \ \big\vert \zeta_2 \big\vert_{L^2}^2\ + \ \gamma\big\vert u_1 \big\vert_{L^2}^2\ + \ \big\vert u_2 \big\vert_{L^2}^2 \quad \text{ is constant in time,}\]
so that without any further assumption than a finite initial energy, we know that $\gamma^{1/2}\frac{\alpha}{\r} \big\vert \zeta_1 \big\vert_{L^2} $ remains bounded as long as the solution is well-defined. For simplicity's sake, {\em we set $\alpha\equiv \r$ in the following}.
\end{Remark}

Let us emphasize again the consequences of the assumptions made on the preceding remarks. The set of parameters we consider throughout the rest of the paper is
\[
\P \ \equiv \ \left\{(\alpha,\delta,\epsilon,\gamma),\ \ \alpha\ = \ \r \ \equiv \ \sqrt{\frac{1-\gamma}{\gamma+\delta}} , \quad \delta_{\min}\ \leq\ \delta\ \leq \ \delta_{\max}, \quad \epsilon \ = \ 1, \quad 0\ < \ \gamma\ <\ 1 \ \right\}.\]
with fixed $0< \delta_{\min}\leq \delta_{\max}<\infty$. The interesting limit is therefore $\r\to0$ or, equivalently, $\gamma\to 1$. Except for Section~\ref{S.preliminaries} and Appendix~\ref{S.app}, we additionally impose $0<\gamma_{\min}\leq \gamma$, with $\gamma_{\min}$ fixed. The assumptions $\epsilon= 1$ and $\alpha =  \r$ do not lack in generality, as one can recover the general case, and in particular the set of parameters in the statement of Theorem~\ref{T.mr}, after applying straightforward scaling factors on the unknowns.
\begin{Remark}\label{R.r-not-small} Notice that we do not impose any smallness on the parameter $\r$. Of course, for non-small $\r$, our result does not improve already existing results in the literature, namely the well-posedness of the Cauchy problem in Sobolev spaces for the free-surface and rigid-lid systems (see Section~\ref{S.preliminaries}). In that case, one does not expect the free-surface solution to be accurately described by the rigid-lid solution. In other words, {\em the rigid-lid approximation is not valid if $\r$ is not small}; see, for example, the discussion and numerical simulations in~\cite{Duchene11a}. When $\r$ is small, the essential assumption is the second inequality in~\eqref{condE0}, which can be viewed as an assumption of {\em well-prepared initial data}: it ensures that the time-derivative of the flow is initially bounded, uniformly for $\r$ small. Such assumptions are standard in the analysis of singularly perturbed systems; see {\em e.g.}~\cite{KlainermanMajda81,BrowningKreiss82}.
\end{Remark}
\begin{Remark}\label{R.d=2}
A natural extension of our work would consist in treating the situation of horizontal dimension $d=2$. The free-surface system in that case has the same quasilinear structure as~\eqref{FS}, and a symmetrizer has been exhibited in~\cite{Duchene10}. On the contrary, the rigid-lid system as constructed in~\cite{BonaLannesSaut08} is quite different as it involves a non-local operator constructed from the orthogonal projector onto the gradient
vector fields of $L^2 (\RR^d)^d$. This can be seen from the fact that equation~\eqref{momentum}, imposed by the rigid-lid hypothesis, becomes $\nabla\cdot (h_1 \mathbf u_1+h_2 \mathbf u_2)=0$, which does not enforce $h_1 \mathbf u_1+h_2 \mathbf u_2=\mathbf 0$ when $\mathbf u_1,\mathbf u_2$ map $\RR^2$ to $\RR^2$ (and in particular,~\eqref{VtoU-intro} does not hold in general). Let us note, however, that the well-posedness of the shallow-water system in the rigid-lid configuration when $d=2$ has been established in~\cite{GuyenneLannesSaut10,BreschRenardy11}.  Interestingly, the system considered in~\cite{BreschRenardy11}, which is formulated differently than in~\cite{BonaLannesSaut08,GuyenneLannesSaut10} and admits non-irrotational velocity fields, offers a clear approximate solution (in the sense of consistency) to the Saint-Venant system in the free-surface configuration. 
\end{Remark}
\begin{Remark}\label{R.topography}
The case of a (sufficiently regular) non-flat bottom topography can be treated following the strategy of this work, after straightforward arrangements. Indeed, the hyperbolic structure of systems~\eqref{FS} and~\eqref{RL} is not altered when topography is taken into account, and the only modification is the apparition of a ``source'' term of the form $\mathbf f(U)\partial_x b$ where $\mathbf f$ is a vector-valued function depending only on $U$ the unknown vector-field, and $b$ the bottom topography. Note however that the decomposition between fast and slow mode introduced in Section~\ref{S.discussion} would not be valid, as the persistence of spatial localization ({\em e.g.} Lemma~\ref{L.persistance}) does not hold with the additional source term.
\end{Remark}
\begin{Remark} \label{R.Kelvin-Helmholtz}
Contrarily to the shallow-water systems~\eqref{FS} and~\eqref{RL}, the corresponding full Euler system is ill-posed in Sobolev spaces in absence of surface (or rather interface) tension, due to the so-called Kelvin-Helmholtz instabilities. In~\cite{Lannes13}, Lannes shows that, at least in the rigid-lid configuration, a small amount of interface tension may be sufficient to regularize the high frequency component of the flow, hence ensuring the existence and uniqueness of a solution to the initial-value problem for large time. By selecting the low-frequency component of the flow, the shallow-water assumption tames the Kelvin-Helmholtz instabilities, and allows for our systems to be well-posed even without the corresponding surface tension components. Conditions~\eqref{condH0}, or more precisely the restrictions on the magnitude of the shear velocity that define the domain of hyperbolicity of systems~\eqref{FS} and \eqref{RL}, are reminiscence of these instabilities.
\end{Remark}

\paragraph{Outline of the paper.} Section~\ref{S.preliminaries} is dedicated to some preliminary results on the Cauchy problem for systems~\eqref{FS} and~\eqref{RL}, obtained through classical techniques on quasilinear, hyperbolic systems. Indeed, one easily checks that systems~\eqref{FS} and~\eqref{RL} are Friedrichs-symmetrizable under reasonable assumptions on the data. As a matter of fact, the Cauchy problem for~\eqref{RL} has been studied in details in~\cite{GuyenneLannesSaut10,BreschRenardy11} (with the much more difficult case of horizontal dimension $d=2$), and we recall their result in Proposition~\ref{P.WPRL}.

In the same way, one obtains easily the well-posedness of the Cauchy problem for the free-surface system~\eqref{FS} through standard energy methods; we state the result in Proposition~\ref{P.WPFS}, and postpone its proof to Appendix~\ref{S.app}. However, the resulting time of existence is only of size $T\gtrsim \r$. One purpose of our work to obtain a control of the energy over large time ({\em i.e.} uniform with respect to $\r$ small), as well as describing the asymptotic behavior of the solution when $\r$ vanishes.

Let us mention that Proposition~\ref{P.WPFS} also contains the usual blow-up criterion, so that item $2.$ in Theorem~\ref{T.mr} is a consequence of the control of the solution on the relevant time scale. Thus it suffices to prove item $3.$, and the entire statement follows. Section~\ref{S.proof} is dedicated to the proof of item $3.$

Finally, in Section~\ref{S.discussion}, we discuss several natural developments around Theorem~\ref{T.mr}, namely
\begin{itemize}
\item The construction of a first-order corrector term in order to reach a higher precision. In particular, we describe the asymptotic behavior of the small deformation at the surface.
\item The case of ill-prepared initial data, that is data failing to meet the smallness assumption in~\eqref{condE0}.
\end{itemize}
On both counts, the relevant notion lies in a decomposition between fast mode and slow mode (or barotropic and baroclinic mode), that we precise therein. Finally, Section~\ref{S.numerics} also contains a discussion on the different results of the present work, supported with numerical simulations.

\paragraph{Notations.} 
If not specified, $C_0$ denotes a nonnegative constant whose exact expression is of no importance. In the present work, $C_0$ almost always depend non-decreasingly on $\delta_{\min}^{-1},\delta_{\max},\gamma_{\min}^{-1}$, and often on $\frac{1}{s_0-1/2}$, such dependency being non-necessarily specified. The notation $a\lesssim b$ or $a=\O(b)$ means $a\leq C_0 b$, and $a\approx b$ means $a\lesssim b$ and $b\lesssim a$, while $a\sim b$ means $\frac{a}{b}\to 1$ ($\r\to0$).

 We denote by $C(\lambda_1, \lambda_2,\dots)$ a nonnegative constant depending on the parameters
 $\lambda_1$, $\lambda_2$,\dots, and whose dependence on the $\lambda_j$ is always assumed to be nondecreasing.
 
 The real inner product of any functions $f_1$
 and $f_2$ in the Hilbert space of square-integrable functions, $L^2=L^2(\RR)$, is denoted by
\[
 \big(\ f_1\ ,\ f_2\ \big)\ =\ \int_{\RR}f_1(x)f_2(x)\ dx.
 \]
 The space $L^\infty=L^\infty(\RR)$ consists of all essentially bounded, Lebesgue-measurable functions
 $f$, and
\[
 \big\vert f\big\vert_{L^\infty}\ =\ \esssup_{x\in\RR} \vert f(x)\vert\ <\ \infty\ .
\]

 For any real $s\geq0$, $H^s=H^s(\RR)$ denotes the Sobolev space of all tempered
 distributions, $f$, endowed with the norm $\vert f\vert_{H^s}=\vert \Lambda^s f\vert_{L^2} < \infty$, where $\Lambda$
 is the fractional derivative $\Lambda=(\Id -\partial_x^2)^{1/2}$.
 
For any $U\equiv (\zeta_1,\zeta_2,u_1,u_2)^\top\in H^s(\RR)^4$ and $0<\gamma< 1$, we introduce the following norm:
\[ \big\vert U \big\vert _{X^s}^2 \ = \ \gamma\big\vert \zeta_1 \big\vert_{H^s}^2+\big\vert \zeta_2 \big\vert_{H^s}^2+\gamma \big\vert u_1 \big\vert_{H^s}^2+\big\vert u_2 \big\vert_{H^s}^2\ .\]
Except in Section~\ref{S.preliminaries} and Appendix~\ref{S.app}, we assume that $\gamma$ is uniformly bounded from below, so that $X^s$ is equivalent to the standard $H^s(\RR)^4$-norm.

 For any functions $u=u(t,x)$ and $v(t,x)$ defined on $ [0,T)\times \RR$ with some $T>0$, we denote the inner product, the $L^2$-norm as well as the Sobolev norms with respect to the spatial variable $x$, with $\big(u,v\big)=\big(u(t,\cdot),v(t,\cdot)\big)$, $\big\vert u \big\vert_{L^2}=\big\vert u(t,\cdot)\big\vert_{L^2}$, and $ \vert u \vert_{H^s}=\vert u(t,\cdot)\vert_{H^s}$, respectively.
 
For $T>0$ and $X$ a functional space, we denote $L^\infty([0,T);X)$, the space of functions such that $u(t,\cdot)$ is controlled in $X$, uniformly for $t\in[0,T)$. This space is endowed with the following norm:
 \[\big\Vert u\big\Vert_{L^\infty([0,T);X)} \ = \ \esssup_{t\in[0,T)}\vert u(t,\cdot)\vert_{X} \ < \ \infty.\]
 Finally, $C^k([0,T);X)$ denote the space of $k$-times continuously differentiable functions in $X$.
 
\section{Preliminary results}\label{S.preliminaries}
In this section, we present some results concerning the Cauchy problem related to the free-surface and rigid-lid systems, respectively~\eqref{FS} and~\eqref{RL}, in Sobolev spaces.
\begin{Proposition}[Well-posedness result concerning the rigid-lid system] \label{P.WPRL} ~\\
Let $s\geq s_0+1$, $s_0 > 1/2$, and $U^0 = (\zeta^0,v^0)^\top \in H^s(\RR)^{2}$ be such that there exists $h_0>0$ with
\begin{equation}\label{conditionSW}
h_1 \ \equiv\ 1- \eta \geq h_0>0, \quad h_2 \ \equiv \ \frac1\delta +\eta \geq h_0>0, \quad \gamma+\delta-\gamma\frac{(1+\delta^{-1})^2}{(h_1+\gamma h_2)^3}|v|^2\geq h_0>0.
\end{equation}
There exists $T_{\max} > 0$ and a unique $U_{\rm RL} = (\eta,v)^\top\! \in C([0,T_{\max});H^s(\RR)^{2})\cap C^1([0,T_{\max});H^{s-1}(\RR)^{2})$, maximal solution to~\eqref{RL} (with $\epsilon=1$), with initial data $U_{\rm RL}\id{t=0}=U^0$. 

Moreover, there exists constants $0<C_0,T^{-1}\leq \big\vert U^0\big\vert_{H^s(\RR)^{2}} \ C(\big\vert U^0\big\vert_{H^s(\RR)^{2}},h_{0}^{-1},\delta_{\min}^{-1},\delta_{\max})$ such that one has $T_{\max}\geq T$, and for any $t\in [0,T]$,
\[ \big\vert U_{\rm RL}(t,\cdot )\big\vert_{H^s(\RR)^{2}} +\big\vert \partial_t U_{\rm RL}(t,\cdot )\big\vert_{H^{s-1}(\RR)^{2}} \ \leq \ C_0\ \exp(C_0 \ t) ,\]
and $U(t,\cdot)$ satisfies~\eqref{conditionSW} uniformly for any $t\in[0,T]$ (with $h_{0}/2$ replacing $h_{0}$). 
\end{Proposition}
This result has been precisely expressed in~\cite[Theorem~1]{GuyenneLannesSaut10}, and follows from standard techniques on quasilinear, Friedrichs-symmetrizable systems. More precisely, the existence and uniqueness of a solution follows from energy estimates on the linearized equation, of which the estimate above is a particular case. In order to assert the well-posedness in the sense of Hadamard, one should also state that the flow depends continuously upon the initial data. Such a result holds: one may control the energy of the difference between two solutions corresponding to different initial data, provided these initial data are sufficiently regular. Precise blow-up conditions, specifying the possible scenarios within the ones stated in Proposition~\ref{P.WPFS}, below, are also presented in~\cite[Corollary~1]{GuyenneLannesSaut10}.
\bigskip

Let us now turn to the free-surface system,~\eqref{FS}. We recall that we set $\alpha=\r=\sqrt{\frac{1-\gamma}{\gamma+\delta}}$ and $\epsilon=1$, so that the system may be written as
\[ \partial_t U \ + \ A[U]\partial_x U \ = \ 0,\]
with $U\equiv (\zeta_1,\zeta_2,u_1,u_2)^\top$ and
 \[ A[U] \ \equiv \ \begin{pmatrix}
 u_1 &\frac{u_2-u_1}{\r} & \frac{1+\r\zeta_1-\zeta_2}{\r} & \frac{\delta^{-1}+\zeta_2}{\r} \\
0 & u_2 & 0 & \delta^{-1}+\zeta_2 \\
\frac{1}{\r} & 0 & u_1 & 0 \\
 \frac\gamma{\r} & \delta+\gamma & 0 & u_2
\end{pmatrix} \ = \ A_0+ A_1(U),\]
where $A_0$ is a constant $4$-by-$4$ matrix, and $A_1(\cdot)$ is a linear mapping into $4$-by-$4$ matrices.

As we show in Appendix~\ref{S.app}, the above system admits an explicit symmetrizer, $S[U]$, which is definite positive provided $U\equiv (\zeta_1,\zeta_2,u_1,u_2)^\top$ satisfies some conditions similar to~\eqref{conditionSW}, namely
 \begin{equation}\label{condH}
 \forall x\in\RR , \qquad h_1(x) \ \geq \ h_{0}\ >\ 0\quad ; \quad h_2(x) - \frac{ |u_2(x)-u_1(x)|^2}{\gamma+\delta} \ \geq \ h_{0}\ >\ 0 ,
 \end{equation}
where we recall: $h_1\equiv 1+\r \zeta_1-\zeta_2$ and $h_2\equiv \delta^{-1}-\zeta_2$.

However, one clearly sees that the system exhibits $1/\r$ factors, which pass on the constants in the energy estimates, thus lowering the {\em a priori} time of existence. We state in Proposition~\ref{P.WPFS}, below, the well-posedness of the Cauchy problem as given by standard energy methods on quasilinear, Friedrichs-symmetrizable systems; remark that the time of existence of the solution is restricted to the poor $T_{\max}\gtrsim \r$. This timescale is intuitively seen from a change of variable: define $U(t,\cdot)\equiv \t U(t/\r,\cdot)$, so that $\t U$ satisfies 
\[ \partial_\tau \t U+\r A[\t U]\partial_x \t U \ = \ 0,\]
and one has $\r A[U]\equiv \r A_0+ \r A_1(U)$, with the matrix $\r A_0$ and the linear mapping $\r A_1(\cdot)$ being both uniformly bounded with respect to $\r\ll 1$.
\begin{Proposition}[Naive well-posedness result for the free-surface system] \label{P.WPFS} ~\\
Let $s\geq s_0+1$, $s_0 > 1/2$, and $U^0 \equiv (\zeta_1^0,\zeta_2^0,u_1^0,u_2^0)^\top\in X^{s}$ be such that~\eqref{condH} holds with $h_0>0$.

There exist $T_{\max} > 0$ and $U =(\zeta_1^0,\zeta_2^0,u_1^0,u_2^0)^\top\! \in C([0,T_{\max});H^s(\RR)^4)\cap C^1([0,T_{\max});H^{s-1}(\RR)^4)$, unique maximal solution to~\eqref{FS} (with $\alpha=\r,\epsilon=1$), with initial data $U\id{t=0}=U^0$. 

Moreover, there exists positive constants $0<C_0,T^{-1}\leq \big\vert U^0\big\vert_{X^s}\ C(\big\vert U^0\big\vert_{X^s},h_{0}^{-1},\delta_{\min}^{-1},\delta_{\max})$, such that one has $T_{\max}\geq T\r $, $U(t,\cdot) $ satisfies~\eqref{condH} for any $t\in [0,T\r]$ (with $h_0/2$ replacing $h_0$), and
\[ \forall t\in [0,T\r],\qquad \big\vert U(t,\cdot )\big\vert_{X^s} +\r \big\vert \partial_t U(t,\cdot )\big\vert_{X^{s-1}} \ \leq \ C_0 \ \exp( C_0\ \r^{-1}\ t ) .\]

Finally, if $T_{\max}<\infty$, then at least one of the following holds:
\begin{itemize}
\item $\big\vert U \big\vert_{L^\infty([0,t]\times\RR)^4} $ or $\big\vert \partial_x U\big\vert_{L^\infty([0,t]\times\RR)^4}$ blows up as $t\nearrow T_{\max}$; or
\item at least one of the conditions in~\eqref{condH} ceases to be true at $t=T_{\max}$.
\end{itemize}
\end{Proposition}
The proof of Proposition~\ref{P.WPFS} is postponed to Appendix~\ref{S.app}, so as not to interrupt the flow of the text.
\begin{Remark}
Condition~\eqref{condH} is a sufficient condition for hyperbolicity, in the sense that it ensures that the symmetrizer
we define and use in Appendix~\ref{S.app} is positive definite. We do not claim that this condition defines exactly the domain of hyperbolicity of system~\eqref{FS} (contrarily to~\eqref{conditionSW} for the rigid-lid system~\eqref{RL}); see~\cite{AbgrallKarni09,Castro-DiazFernandez-NietoGonzalez-VidaEtAl11,StewartDellar13} for a more detailed analysis on this point. In particular, one would expect the hyperbolic domain of the free-surface system to asymptotically correspond to~\eqref{conditionSW} in the limit $\r\to0$, which is not the case for~\eqref{condH}, the latter being more stringent.
\end{Remark}
\begin{Remark}\label{R.SmallInitialData}
Notice that a uniform time of existence, $T\gtrsim 1$, is recovered for sufficiently small initial data: $\big\vert U_0\big\vert_{X^s}=\O(\r)$. This result can be viewed through the following change of unknowns: $U \equiv \r \breve U$. The function $\breve U$ satisfies
\[ \partial_t \breve U+ A[\r \breve U]\partial_x \breve U \ = \ 0,\]
and $ A[\r \breve U]\equiv A_0+ \r A_1(\breve U)$. The fact that the constant operator $A_0\partial_x$ is not uniformly bounded with respect to $\r\ll 1$ does not prevent solutions to exist in a time domain independent of $\r$, because it does not contribute to commutator estimates. This simple observation motivates the strategy we use to prove Theorem~\ref{T.mr}, as described in Section~\ref{S.proof}.
\end{Remark}

\section{Proof of the main result}\label{S.proof}
This section is dedicated to the proof of Theorem~\ref{T.mr}. Our first ingredient consists in constructing a system equivalent to~\eqref{FS}, but whose non-linear contribution is uniformly bounded with respect to $\r$. In order to do so, we shall use different variables. Considering the conservation of horizontal momentum displayed in Section~\ref{S.models}, we introduce the horizontal momentum, $m\equiv \gamma h_1 u_1+h_2 u_2$, and the shear velocity $u_s\equiv u_2-\gamma u_1$. One has immediately:
\begin{equation}\label{UtoV} u_s\equiv u_2-\gamma u_1 \quad \text{ and } \quad m\equiv \gamma h_1 u_1+h_2 u_2 \end{equation}
if and only if
\begin{equation}\label{VtoU} u_1=\frac{m-h_2 u_s}{\gamma(h_1+h_2)}\quad \text{ and } \quad u_2=\frac{m+h_1 u_s}{h_1+h_2}.\end{equation}
Straightforward manipulations of the system~\eqref{FS} yield the new system of conservation laws we consider:
\begin{equation}\label{FS2}\left\{ \begin{array}{l}
\partial_t \zeta_1 + \frac1\r\partial_x m +\frac{1-\gamma}{\gamma\r} \partial_x\left( h_1\frac{m-h_2 u_s}{h_1+h_2}\right)= 0, \\
\partial_t \zeta_2 +\partial_x \left(\frac{h_2}{h_1+ h_2}(h_1 u_s+ m)\right) = 0, \\
\partial_t u_s + (\delta+\gamma)\partial_x \zeta_2 +\dfrac{1}{2}\partial_x \left(\frac{\gamma(m+h_1 u_s)^2-(m-h_2 u_s)^2}{\gamma(h_1+h_2)^2}\right) = 0, \\
\partial_t m +\gamma \frac{h_1+h_2}{\r}\partial_x\zeta_1+(\gamma+\delta)h_2\partial_x\zeta_2+\partial_x \left(\frac{h_1(m-h_2 u_s)^2+\gamma h_2(m+h_1 u_s)^2}{\gamma(h_1+h_2)^2}\right) =0.
\end{array} \right.\end{equation}
We still refer to this system as the {\em free-surface system}. 
Systems~\eqref{FS2} and~\eqref{FS} are equivalent in the following sense.
\begin{Proposition}\label{P.FsvsFS2}
Let $s\geq s_0+1$, $s_0>1/2$. Let $V\equiv (\zeta_1,\zeta_2,u_s,m)^\top\in C([0,T];H^s(\RR)^4)$ be a strong solution to~\eqref{FS2}, with $T>0$, given. Assume that for any $t\in [0,T]$, one has
\[ \exists h_0>0 \quad \text{ such that } \quad \min_{x\in\RR,t\in[0,T]}\big\{ h_1(t,x)+h_2(t,x)=1+\delta^{-1}+\r\zeta_1(t,x)\big\}\geq h_0>0.\]
Then $U\equiv (\zeta_1,\zeta_2,u_1,u_2)^\top\in C([0,T];H^s(\RR)^4)$, where $u_1$ and $u_2$ are given by~\eqref{VtoU}, is a strong solution to~\eqref{FS}. 

Conversely, if a given $U\equiv (\zeta_1,\zeta_2,u_1,u_2)^\top\in C([0,T];H^s(\RR)^4)$ is a strong solution to~\eqref{FS}, and the above non-vanishing depth condition holds; then $V\equiv (\zeta_1,\zeta_2,u_s,m)^\top\in C([0,T];H^s(\RR)^4))$, given by~\eqref{UtoV}, is a strong solution to~\eqref{FS2}.
\end{Proposition}
\begin{proof}
The existence and regularity of $U\in C([0,T];H^s(\RR)^4)$ (resp. $V\in C([0,T];H^s(\RR)^4)$) is deduced from the corresponding control of $V$ (resp. $U$), using product estimates in Lemma~\ref{L.Moser}, as well as Corollary~\ref{C.depth}. As usual, one deduces from the system satisfied by, say, $V$ ---namely~\eqref{FS2}--- the corresponding estimate $\partial_t V\in C([0,T];H^{s-1}(\RR)^4)$, and $\partial_t U\in C([0,T];H^{s-1}(\RR)^4)$ follows. The fact that $U$ satisfies~\eqref{FS} if $V$ satisfies~\eqref{FS2}, and conversely, demands somewhat tedious but straightforward computations, that we leave to the reader.
\end{proof}
\begin{Remark}
We do not claim here that the aforementioned solutions are unique. The uniqueness of a solution to~\eqref{FS} is given in Proposition~\ref{P.WPFS} and requires additional conditions on the initial data, namely~\eqref{condH}. We prove later on that these conditions are also sufficient to ensure the uniqueness of a solution to~\eqref{FS2}; see Lemma~\ref{L.hyp}.
\end{Remark}

\noindent {\bf Strategy and discussion.} We see two benefits in considering~\eqref{FS2} in lieu of~\eqref{FS}. First the rigid-lid system, which was encrypted in~\eqref{FS}, is now apparent in~\eqref{FS2}. This will be helpful, although not necessary, for the construction of the approximate solution in the subsequent subsection. More importantly, one sees that the only terms factored by $\r^{-1}$ in~\eqref{FS2} are constant. This second property is crucial for our analysis, and justifies the use of~\eqref{FS2}.
\medskip

Let us briefly sketch the key arguments in the proof of Theorem~\ref{T.mr}, before we continue with the detailed analysis in the following subsections. 
We first introduce some notations, used thereafter. We rewrite the hyperbolic system~\eqref{FS2} as
\begin{equation}\label{FS3}\partial_t V \ +\ \left(\frac1\r L_\r + B[V]\right) \partial_x V \ = \ 0,\end{equation}
with $V\equiv (\zeta_1,\zeta_2,u_s,m)^\top$, and where
\begin{itemize}
\item $\frac1\r L_\r$ represents the linear component of the system; see precise expression below.
\item $B[\cdot]$ contains the nonlinear contribution: it is uniformly bounded with respect to $\r$.
\end{itemize}

In Section~\ref{S.consistency}, we construct an approximate solution, $\Vapp$, satisfying~\eqref{FS2} as well as the initial data, up to a small remainder. Thus defining $W\equiv V-\Vapp $ where $V$ is the exact solution, one has
\begin{equation}\label{eqn:FSB}
\partial_t W \ + \ \frac1\r \left( L_\r + \r B[\Vapp+W]\right) \partial_x W\ = \ \mathcal{R},
\end{equation}
with $W\id{t=0}$ and $\mathcal R$ small (typically of size $\O(\r)$). Our aim is to prove that $W$ remains small for large time ({\em i.e.} bounded from below uniformly with respect to $\r$), and Theorem~\ref{T.mr} quickly follows (see Section~\ref{S.completion}).

When compared with the classical theory of Friedrichs-symmetrizable quasilinear systems, the main issue we face when controlling $W$ in the natural energy space lies in the two following facts:
\\
 (i) one has to control the contribution from the unbounded component $\frac1\r L_\r$ in the energy space, which may generate a destructive $\O(\r^{-1})$ factor; and\\
 (ii) one cannot use the equation in order to deduce a uniform control of $\partial_t W$ from the corresponding control of $\partial_x W$, as once again this would yield a destructive $\O(\r^{-1})$ factor. 

These two difficulties are only apparent, as shows a careful study of the symmetrizer of the system. In Section~\ref{S.Symmetrizer}, we introduce and study the symmetrizer, $T[\cdot]$, as well as ${\Upsilon[\cdot]\equiv  T[\cdot]\big(\frac1\r L_\r + B[\cdot]\big)}$. In particular, one can check that (roughly speaking) $\Upsilon[\cdot]\equiv \frac1\r\Upsilon_0+\O(1)$ with $\Upsilon_0$ a constant matrix, so that differentiation or commutation with the operator $\Upsilon[\cdot]$ is actually bounded; thus issue (i) can be faced.

Issue (ii) asks for a more specific analysis. We introduce $\Pi\equiv \left(\begin{smallmatrix}
0 &&&\\ &1&&\\&&1&\\&&&0
\end{smallmatrix}\right) $ the orthogonal projector onto the kernel of $L_{(0)}$, denoting $L_{(0)}=\lim_{\gamma\to 1}L_\r$; see below. It follows that $\big\vert \Pi \partial_t W\big\vert_{X^{s-1}}\lesssim \big\vert W\big\vert_{X^s}$, uniformly with respect $\r$ small. As for the other component, one shows that $ T[\cdot](\Id-\Pi)=T_0+\O(\r )$ with $T_0$ a constant matrix, so a factor of size $\O(\r)$ is gained after differentiation or commutation with this operator.

The detailed energy estimates are computed in Section~\ref{S.completion}.
\medskip

There is an intuitive explanation for the reason why the above claims hold. By precisely analyzing the $4$-by-$4$ matrix $L_\r$:
\[ L_\r \ \equiv \ \begin{pmatrix}
0&0 & \frac{\gamma-1}{\gamma(\delta+1)}& \frac{\gamma+\delta}{\gamma(\delta+1)} \\
0 & 0 & \frac{\r}{1+\delta}& \frac{\r}{1+\delta} \\
0 &\r (\gamma+\delta) &0 & 0 \\
\gamma(1+\delta^{-1})& \r\frac{\delta+\gamma}{\delta} & 0 &0
\end{pmatrix} ,\]
 one may check that for $\r$ sufficiently small, $L_\r$ has four distinct, real eigenvalues, namely
\[ \lambda^f_{\pm}(\r) \ = \ \pm \sqrt{1+\delta^{-1}}\ + \ \O(\r^2) \quad ; \quad \lambda^s_{\pm}(\r) \ = \ \pm \r \ + \ \O(\r^3).\]
The linear theory thus predicts that the flow can be decomposed as the superposition of four waves, propagating at velocity $c^f_\pm \sim \pm \frac{\sqrt{1+\delta^{-1}}}{\r}$, and $c^s_\pm \sim \pm 1$, which we name {\em fast mode} (resp. {\em slow mode}). Roughly speaking, the slow mode corresponds to the flow predicted by the rigid-lid system, and the terms neglected in the rigid-lid approximation correspond to the fast mode. 

An important feature of the free-surface system, which is revealed by our change of variable, is that the fast and slow modes are supported on (approximately) orthogonal components, which is responsible for the fact that coupling effects between the two modes are small.
More precisely, if we denote $L_{(0)} \equiv \lim_{\gamma\to 1} L_\r \ \equiv \ \left(\begin{smallmatrix}
& &  & 1 \\
 &  & 0 &  \\
 & 0 & &  \\
1+\delta^{-1}& &  &
\end{smallmatrix}\right)$, then one easily checks that the eigenvectors corresponding to the two non-zero eigenvalues of $L_{(0)}$ are orthogonal to the kernel of $L_{(0)}$. 
Therefore, roughly speaking, the slow mode is supported by variables $\zeta_2$ and $u_s$, while the fast mode is supported by variables $\zeta_1$ and $m$. We take advantage of this fact by treating separately the slow mode terms (multiplying by $\Pi$, the orthogonal projector onto the kernel of $L_{(0)}$) and fast mode terms (multiplying by $\Id-\Pi$, the orthogonal projector onto the space spanned by the other eigenvectors of $L_{(0)}$). The former contributions are easily controlled as time differentiation does not induce destructive $\O(\r^{-1})$ factor. As for the latter, the property $ T[\cdot](\Id-\Pi)=T_0+\O(\r )$ reflects the fact that the corresponding eigenvalues are well separated; thus the perturbation by $\r B[\cdot]$ typically yield deviations of size $\O(\r)$, following standard perturbation theory~\cite{Kato95}. Finally, the desired property on $\Upsilon[\cdot]$ is easily checked:
\[ \Upsilon[\cdot]\equiv  T[\cdot]\Pi\left( \frac1\r L_\r + B[\cdot]\right)+ T[\cdot](\Id-\Pi)\left( \frac1\r L_\r + B[\cdot]\right)= \frac1\r T_0(\Id-\Pi)L_\r +\O(1).\]

We let the reader refer to Section~\ref{S.discussion} for a more precise investigation of the decomposition of the flow into fast and slow modes, and numerical illustrations.

\subsection{Construction of the approximate solution}\label{S.consistency}
In this section, we construct an approximate solution to the free-surface system~\eqref{FS2}, using the corresponding solution to the rigid-lid system~\eqref{RL}, as defined below. 
\medskip

Let us recall that from there on, we assume that $\gamma$ is uniformly bounded from below: $\gamma\geq \gamma_{\min}>0$. In particular, the norm $X^s$ is equivalent to the standard $H^s(\RR)^4$-norm, and will be used as such.
\medskip

\begin{Definition}[Rigid lid approximate solution] \label{D.VRL}
For a given initial data $\zeta_2^0,u_s^0$, satisfying~\eqref{conditionSW}, the {\em rigid-lid approximate solution} corresponding to $(\zeta_2^0,u_s^0)^\top$ is denoted $\Vrl\equiv(0,\eta,v,0)^\top$, where $V\equiv (\eta,v)^\top$ is the unique solution to the rigid-lid system~\eqref{RL} with $V\id{t=0}\equiv(\zeta_2^0,u_s^0)^\top$.
\end{Definition}
\begin{Proposition}\label{P.ConsApp}
Let $s\geq s_0,\ s_0>1/2$, and $\zeta_2^0,u_s^0\in H^{s+1}(\RR)$, satisfying~\eqref{conditionSW} with $h_0>0$, and $\big\vert (\zeta_2^0,u_s^0)^\top\big\vert_{H^{s+1}\times H^{s+1}}\leq M$. Then there exists $0<T^{-1},C_1,C_2,C_3\leq M\ C\big(M,h_0^{-1},\delta_{\min}^{-1},\delta_{\max},\gamma_{\min}^{-1}\big)$, with
\begin{itemize}
\item $\Vrl\in C([0,T];X^{s+1}) \cap C^1([0,T];X^{s})$ is well-defined as above, and satisfies
\begin{equation}\label{est:Us} \forall t\in [0,T],\qquad \big\vert \Vrl \big\vert_{X^{s+1}} +\big\vert \partial_t \Vrl \big\vert_{X^{s}} \ \leq\ C_1 .
\end{equation}
\item There exists $\Vr\in C([0,T];X^{s+1})\cap C^1([0,T];X^{s})$, with
\begin{equation}\label{est:Vr} \forall t\in [0,T],\qquad \big\vert \Vr \big\vert_{X^{s+1}} +\big\vert \partial_t \Vr \big\vert_{X^{s}} \ \leq\  C_2\ \r,
\end{equation}
such that $\Vapp\equiv \Vrl+\Vr$ satisfies~\eqref{FS2}, up to a remainder term $R$, with
\begin{equation}\label{est:RR} 
\big\Vert R \big\Vert_{L^\infty([0,T];X^s)}\ \leq \ C_3\ \r\ \big( M+ \r\big). \end{equation}
\end{itemize}
\end{Proposition}
\begin{Remark}
The explicit formula for $\Vr$, which is precisely displayed in the proof, below,
does not play a significant role in this section, except as a technical artifice to obtain the desired estimate. In particular, it does not appear in Theorem~\ref{T.mr}. However, as discussed in Section~\ref{S.discussion}, it corresponds to a first order correction of the approximate solution, and is clearly observable in our numerical simulations.
\end{Remark}
\begin{proof}[Proof of Proposition~\ref{P.ConsApp}] By Proposition~\ref{P.WPRL}, there exists $C_1,T^{-1}\leq M C(M,h_0^{-1},\delta_{\min}^{-1},\delta_{\max})$ such that $\Vrl \in C([0,T];X^{s+1})$ is well-defined by Definition~\ref{D.VRL}, and~\eqref{est:Us} holds.

We now plug $\Vapp\equiv\Vrl+\Vr$ into~\eqref{FS2}, and check that one can explicitly define a function $\Vr\equiv \Vr[\eta,v]$ such that the remainder term, $R$, satisfies the estimate of the Proposition. Anticipating the result, we denote $\Vapp\equiv (\r\breve \zeta_1,\eta,v,\r^2 \breve m)^\top$, and subsequently
\begin{equation}\label{FS2s}\left\{ \begin{array}{l}
\r \partial_t \breve \zeta_1 + \r\partial_x \breve m +\frac{1-\gamma}{\gamma\r} \partial_x\left( h_1\frac{\r^2\breve m-h_2 v}{h_1+h_2}\right)= r_1, \\
\partial_t \eta +\partial_x \left(\frac{h_2}{h_1+ h_2}(h_1 v+ \r^2\breve m)\right) = r_2, \\
\partial_t v + (\delta+\gamma)\partial_x \eta +\frac{1}{2}\partial_x \left(\frac{\gamma(\r^2\breve m+h_1 v)^2-(\r^2\breve m-h_2 v)^2}{\gamma(h_1+h_2)^2}\right) = r_3, \\
\r^2 \partial_t \breve m +\gamma (h_1+h_2)\partial_x\breve \zeta_1+(\gamma+\delta)h_2\partial_x\eta+\partial_x \left(\frac{h_1(\r^2\breve m-h_2 v)^2+\gamma h_2(\r^2\breve m+h_1 v)^2}{\gamma(h_1+h_2)^2}\right) =r_4,
\end{array} \right.\end{equation}
with $h_1\equiv 1+ \r^2\breve \zeta_1-\eta$ and $h_2\equiv \delta^{-1}+\eta$.

Our aim is to prove that one can choose $\breve \zeta_1$ and $\breve m$ such that
\begin{equation}\label{hypVr}
 \big\vert \breve \zeta_1 \big\vert_{H^{s+1}}+\big\vert \breve m \big\vert_{H^{s+1}} \ +\ \big\vert\partial_t \breve \zeta_1 \big\vert_{H^{s}}+\big\vert \partial_t \breve m \big\vert_{H^{s}} \ \leq \ C_2 ,
\end{equation}
and 
\begin{equation}\label{hypR}
\big\vert r_1 \big\vert_{H^s}+\big\vert r_2 \big\vert_{H^s}+\big\vert r_3 \big\vert_{H^s}+\big\vert r_4 \big\vert_{H^s}\ \leq \ C_3 \ \r\ (M +\r).
\end{equation}
In order to ease the reading of the argument, we first assume that~\eqref{hypVr} holds, and see how $\breve \zeta_1,\breve m$ can be naturally chosen so that~\eqref{hypR} is satisfied. Our choice for $\breve \zeta_1,\breve m$ is precisely stated in~\eqref{defbz1} and~\eqref{defbm}, below, and checking that~\eqref{hypVr} is actually satisfied is then a straightforward consequence of~\eqref{est:Us}.
\smallskip

Recall that, by definition, $(\eta,v)^\top$ satisfies~\eqref{RL}. In particular, from the first equation in~\eqref{RL}, one deduces 
\[ r_2= \partial_x \left(\frac{h_1h_2 v}{h_1+ h_2}-\frac{\u h_1h_2 v}{\u h_1+\gamma h_2} \right)+\r^2\partial_x \left(\frac{h_2\breve m}{h_1+ h_2}\right) ,\]
where we denote $\u h_1\equiv 1-\eta$, the depth of the upper layer in the rigid-lid approximation.

Let us recall that $\Vrl$ satisfies~\eqref{est:Us}, and also~\eqref{conditionSW}. Thus one can apply the product estimates in Lemma~\ref{L.Moser} as well as Corollary~\ref{C.depth} (we also recall that by definition, $1-\gamma=\r^2(\gamma+\delta)$), to deduce
\begin{equation}\label{est:r2}
\big\Vert r_2\big\Vert_{L^\infty([0,T/M];H^{s})} \ \leq \ M \r^2 \ C(M,h_0^{-1},C_2,\delta_{\min}^{-1},\delta_{\max}),
\end{equation}
where we used the {\em a priori} estimate~\eqref{hypVr}.

Similarly, one deduces from the second equation in~\eqref{RL} that 
\[r_3 = \dfrac{1}{2}\partial_x \left(\left\{\frac{\gamma h_1^2-h_2^2}{\gamma(h_1+h_2)^2}-\frac{h_1^2-\gamma h_2^2}{(h_1+\gamma h_2)^2}\right\}v^2+\frac{\gamma(\r^2\breve m+h_1 v)^2-\gamma^2 h_1^2v^2+h_2^2v^2-(\r^2\breve m-h_2 v)^2}{\gamma(h_1+h_2)^2}\right),\]
so that one has as above,
\begin{equation}\label{est:r3}
\big\Vert r_3\big\Vert_{L^\infty([0,T/M];H^{s})} \ \leq \ M\r^2 \ C(M,h_0^{-1},C_2,\delta_{\min}^{-1},\delta_{\max},\gamma_{\min}^{-1}).
\end{equation}

Let us now look at the fourth equation in~\eqref{FS2s}. Note that one has
\begin{multline*}
\gamma (h_1+h_2)\partial_x\breve \zeta_1+(\gamma+\delta)h_2\partial_x\eta+\partial_x \left(\frac{h_1h_2(\gamma h_1+ h_2) v^2}{\gamma(h_1+h_2)^2}\right) \\
= \partial_x\left(\gamma\big((1+\delta^{-1})\breve \zeta_1+\frac{\r^2}2\breve\zeta_1^2\big) +(\gamma+\delta)\big(\delta^{-1}\eta+\frac{1}2\eta^2\big)+ \frac{h_1h_2(\gamma h_1+ h_2) v^2}{\gamma(h_1+h_2)^2}\right).
\end{multline*}
It is now clear that one can choose
\begin{equation}\label{defbz1}
\breve \zeta_1\ \equiv \ -\big(\eta+\frac{\delta}2\eta^2\big)-\frac{( 1-\eta)(\delta^{-1}+\eta)v^2}{(1+\delta^{-1})^2},
\end{equation}
 so that the above is of size $\O(\r^2)$. More precisely, and using once again~\eqref{hypVr}, one has
\begin{equation}\label{est:r4}
\big\Vert r_4\big\Vert_{L^\infty([0,T/M];H^{s})} \ \leq \ M\r^2\ C(M,h_0^{-1},C_2,\delta_{\min}^{-1},\delta_{\max},\gamma_{\min}^{-1}).
\end{equation}

We conclude with the first equation in~\eqref{FS2s}. Using that $\r^2=\frac{1-\gamma}{\gamma+\delta}$, one has
\[ r_1 \ = \ \r\left( \partial_t \breve \zeta_1 + \partial_x \breve m +\frac{\gamma+\delta}{\gamma} \partial_x\left( h_1\frac{\r^2\breve m-h_2 v}{h_1+h_2}\right)\right).\]
We now recall that $(\eta,v)^\top$ satisfies~\eqref{RL}, so that one deduces explicitly $\partial_t \breve \zeta_1$ from~\eqref{defbz1}, and
\[ \left\vert \partial_t \breve \zeta_1 -\partial_x \left(\frac{\u h_1h_2 v}{\u h_1+ \gamma h_2}\right) \right\vert_{H^s} \ \leq M^2 \ C(M,h_0^{-1},C_2,\delta_{\min}^{-1},\delta_{\max},\gamma_{\min}^{-1}) .\]
Now, one can check that by choosing
\begin{equation}\label{defbm}
\breve m\ \equiv \ \frac{\delta}{1+\delta}v,
\end{equation}
it follows
\[ \left\vert \frac{\u h_1h_2 v}{\u h_1+ \gamma h_2}+ \breve m -\frac{\gamma+\delta}{\gamma} \frac{h_1h_2 v}{h_1+h_2}\right\vert_{H^s} \ \leq \ (M^2+M\r^2)\ C(M,h_0^{-1},C_2,\delta_{\min}^{-1},\delta_{\max},\gamma_{\min}^{-1}),\]
so that estimates~\eqref{est:Us} and~\eqref{hypVr} yield
\begin{equation}\label{est:r1}
\big\Vert r_1 \big\Vert_{L^\infty([0,T/M];H^{s})} \ \lesssim \ (M^2\r +M\r^2)\ C(M,h_0^{-1},C_2,\delta_{\min}^{-1},\delta_{\max},\gamma_{\min}^{-1}).
\end{equation}

Estimates~\eqref{est:r2},~\eqref{est:r3},~\eqref{est:r4} and~\eqref{est:r1} give the desired estimate:~\eqref{hypR}, or equivalently~\eqref{est:RR}. Moreover, one easily deduces from the estimate concerning $\Vrl$ in~\eqref{est:Us}, the corresponding estimate on $\Vr\equiv(\r\breve\zeta_1,0,0,\r^2\breve m)^\top$:~\eqref{hypVr}, or equivalently~\eqref{est:Vr}. Proposition~\ref{P.ConsApp} is proved.
\end{proof}

\subsection{Properties of the system and its symmetrizer}\label{S.Symmetrizer}
This section is dedicated to preliminary results on the new free-surface system~\eqref{FS2} and its symmetrizer, which allow the energy analysis of the subsequent subsection.

We recall here that~\eqref{FS2} has been constructed from~\eqref{FS} through a change of variables: for any $U\in (\zeta_1,\zeta_2,u_1,u_2)^\top$ solution to~\eqref{FS}, we uniquely associate $V\equiv (\zeta_1,\zeta_2,u_s,m)^\top$ solution to~\eqref{FS2}, through the change of variable~\eqref{UtoV}; see Lemma~\ref{P.FsvsFS2}. In other words, we have an explicit
\[ F:\begin{array}{ccc}
X & \to & X\\
 (\zeta_1,\zeta_2,u_s,m)^\top& \mapsto &(\zeta_1,\zeta_2,u_1,u_2)^\top
\end{array}\]
(in this section, the space $X$ may be $L^\infty(\RR)^4$ or $H^s(\RR)^4$, $s>1/2$) which is one-to-one and onto provided the non-vanishing depth condition is satisfied: 
\begin{equation}\label{condDepth} \exists h_0>0 \quad \text{ such that } \quad \min_{x\in\RR,t\in[0,T]}\big\{ h_1(t,x)+h_2(t,x)=1+\delta^{-1}+\r\zeta_1(t,x)\big\}\geq h_0>0.\end{equation}
It follows that, recalling the notation for~\eqref{FS} as 
\[ \partial_t U+A[U]\partial_x U=0,\]
one may rewrite~\eqref{FS2} (after multiplication with the appropriate operator)
\[{\rm d} F[V]\partial_t V + A[F(V)] {\rm d} F[V]\partial_x V =0,\]
where ${\rm d} F[V]$ is the Jacobian matrix of $F$. In other words, recalling earlier notation in~\eqref{FS3}, one has
\[ \partial_t V+\left(\frac1\r L_\r+B[V]\right)\partial_x V=0 \quad \text{ with } \quad \frac1\r L_\r+B[V] = ({\rm d} F[V])^{-1}A[F(V)]. \]

Thus the symmetrizer of the new system~\eqref{FS2} is readily available from the one of system~\eqref{FS}.
\begin{Lemma}\label{L.hyp}
Let $S[\cdot]$ be a symmetrizer of~\eqref{FS}; {\em e.g.}~\eqref{defS}. Then $T[\cdot]\equiv ({\rm d} F[\cdot])^\top S[F(\cdot)]{\rm d} F[\cdot]$ is a symmetrizer of~\eqref{FS2}. Moreover, $T[V]$ is definite positive if and only if $F(V)$ satisfies~\eqref{condH}.
\end{Lemma}
\begin{proof}
For any $V\in X$, the operator $T[V]$ is obviously symmetric. Moreover, $T[V]$ is definite positive if and only if $S[F(V)]$ is definite positive, since one has
\begin{equation}\label{defpos}
 \forall \x\in \RR^4, \quad T[V]\x \cdot \x =S[F(V)] ({\rm d} F[V] \x ) \cdot ({\rm d} F[V] \x ) ,
 \end{equation}
and ${\rm d} F[V]$ is invertible provided $V$ satisfies~\eqref{condDepth}. Let us note that the hyperbolicity condition~\eqref{condH} is obviously more stringent than~\eqref{condDepth}.

Finally, it is straightforward to check that
 \[
T[V]\left(\frac1\r L_\r + B[V]\right) =({\rm d} F[V])^\top S[F(V)]A[F(V)] {\rm d} F[V]\]
is symmetric, and this concludes the proof.
\end{proof}

We conclude that one can construct an explicit symmetrizer of system~\eqref{FS2}, using $S[\cdot]$ given in~\eqref{defS}. However, this symmetrizer has a quite complicate expression, and we do note display it here. We will only present the necessary properties of the operators at stake, which are easily checked with the use of a computer algebra system, such as~\textsf{Maple}.

\begin{Lemma}\label{C.B}
Let $V,W\in X$ satisfying~\eqref{condDepth}, and $\frac1\r L_\r+B[\cdot] \equiv ({\rm d} F[\cdot])^{-1}A[F(\cdot)]{\rm d} F[\cdot]$ defined above. Then one has
\begin{equation}\label{bound:B}
\big\Vert B[V] \big\Vert_X \ \leq \ C_0\big\vert V\big\vert_X,\quad 
\big\Vert B[V]-B[W] \big\Vert_X \ \leq \ C_0\big\vert V-W\big\vert_X, \end{equation}
with $C_0=C(\big\vert V\big\vert_{X},\big\vert W\big\vert_{X},\delta_{\min}^{-1},\delta_{\max},\gamma_{\min}^{-1})$, and where we denote $\displaystyle\big\Vert A\big\Vert_X\equiv \sup_{V\in X\setminus\{\mathbf{0}\}}\frac{\big\vert A V\big\vert_X}{\big\vert V\big\vert_X}$.
\end{Lemma}
\begin{proof}
Let us recall that $B[\cdot]$ has a complicated expression, but is explicit; it involves only products of the components of $V$, or factors of the form $\frac1{h_1+h_2}$. Thus one can apply Lemma~\ref{L.Moser} and Corollary~\ref{C.depth} (since~\eqref{condDepth} holds), and the result easily follows.
\end{proof}

\begin{Lemma}\label{C.S}
Denote $T[\cdot]\equiv ({\rm d} F[\cdot])^\top S[F(\cdot)]{\rm d} F[\cdot]$ and $\Upsilon[\cdot]\equiv ({\rm d} F[\cdot])^\top \Sigma[F(\cdot)] {\rm d} F[\cdot] $, with $S[\cdot]$ and $\Sigma[\cdot]=S[\cdot]A[\cdot]$ are defined in~\eqref{defS},~\eqref{defSigma}. Let $V\in X$ such that $F(V)$ satisfies~\eqref{condH} with $h_0>0$. Then there exists $C_0=C(\big\vert V\big\vert_{X},h_0^{-1},\delta_{\min}^{-1},\delta_{\max},\gamma_{\min}^{-1})$ such that one has 
\begin{enumerate}
\item $T[V],\Upsilon[V]$ are symmetric. $T[V]$ is positive definite. More precisely, for any $W\in L^2(\RR)^4$, one has
\begin{equation}\label{bound:elliptic}
\frac1{C_0}\big\vert W\big\vert_{L^2}^2\ \leq \ \big( T[V] W \ , \ W\big) \ \leq\ C_0\big\vert W\big\vert_{L^2}^2.\end{equation}
\item $T[V],\Upsilon[V]$ satisfy the following estimates:
\begin{equation}\label{bound:SSigma}
\big\Vert T[V] \big\Vert_X \ \leq \ C_0\quad ; \quad \big\Vert \Upsilon[V]\big\Vert_X\ \leq \ \r^{-1}C_0 .\end{equation}
\item If $V\equiv V(\varkappa)$ and $\partial_\varkappa V\in X$, then
\begin{equation}\label{bound:dSdSigma} \big\Vert \partial_\varkappa (T[V])\big\Vert_X \ \leq \ C_0 \big\vert \partial_\varkappa V\big\vert_X  \quad ; \quad \big\Vert \partial_\varkappa (\Upsilon[V])\big\Vert_X \ \leq \ C_0 \big\vert \partial_\varkappa V\big\vert_{X} 
\end{equation}
and
\begin{equation}\label{bound:PfdS}
\big\Vert \partial_\varkappa (T[V])(\Id-\Pi) \big\Vert_X \ \leq \ \r\  C_0 \big\vert \partial_\varkappa V\big\vert_{X},
\end{equation}
recalling the notation $\Pi\equiv \left(\begin{smallmatrix}
0 &&&\\ &1&&\\&&1&\\&&&0
\end{smallmatrix}\right)$. 
\end{enumerate}
\end{Lemma}
 \begin{proof}
 That $T[V]$ is symmetric, positive definite and $\Upsilon[V]$ is symmetric has been already stated in Lemma~\ref{L.hyp}. Estimate~\eqref{bound:elliptic} follows from Lemma~\ref{L.energyXs} and~\eqref{defpos}, recalling that $\gamma\geq \gamma_{\min}>0$ ensures that the $L^2(\RR)^4$-norm is equivalent to the $X^0$-norm. 
 
Estimates~\eqref{bound:SSigma} are direct consequences of the corresponding estimates on $S[\cdot]$, $A[\cdot]$ as well as $F(\cdot)$, $ {\rm d}F[\cdot]$, which are easily checked. We recall that the necessary product estimates in $X=L^\infty(\RR)^4$ or $X=H^s(\RR)^4$ ($s>1/2$) are given by Lemma~\ref{L.Moser} and Corollary~\ref{C.depth}. The first estimate in~\eqref{bound:dSdSigma} is obtained similarly.
 
 Finally, the second estimate in~\eqref{bound:dSdSigma} as well as~\eqref{bound:PfdS} are less obvious, but can be checked with the help of a computer algebra system (we have to ensure that first order terms in $\r$ are all constant). \end{proof}

\subsection{Completion of the proof}\label{S.completion}
Denote $V$ a strong solution to the free-surface system~\eqref{FS2} satisfying the non-vanishing depth condition,~\eqref{conditionSW}; and $\Vapp$ the approximate solution constructed in Proposition~\ref{P.ConsApp}. 
One easily checks that $W\equiv V-\Vapp$ satisfies the following system:
\begin{equation}\label{eqn:W}
\partial_t W \ + \ \frac1\r \left( L_\r + \r B[V]\right) \partial_x W\ = \ \mathcal{R},
\end{equation}
with $\mathcal{R}\equiv R- ( B[\Vapp+W] - B[\Vapp])\partial_x \Vapp $, where $R$ is estimated in Proposition~\ref{P.ConsApp}. 

The following Lemma presents an {\em a priori} energy estimate on $W$ satisfying the above system, from which our desired result is based on.

\begin{Lemma}\label{L.HsW} Let $s\geq s_0+1,\ s_0>1/2$, and $W$ a strong solution to~\eqref{eqn:W}, with $W\id{t=0}\in X^s$. Assume that there exists $M,T,h_0>0$ such that $F(V)$ satisfies~\eqref{condH} and 
\[ 
 \big\Vert \W \big\Vert_{L^\infty([0,T];X^s)} \ +\ \big\Vert \partial_t \W \big\Vert_{L^\infty([0,T];X^{s-1})}\ \leq \ M.\]
Then one has
\begin{equation}\label{energyestimateHsW}
	\forall t\in [0,T],\qquad
	\big\vert W(t,\cdot) \big\vert_{X^s}\leq C_0 	\big\vert W(0,\cdot) \big\vert_{X^s} e^{ C_0 M t}\ + \ C_0 \int_{0}^{t} e^{ C_0 M ( t-t')}\big\vert \R(t',\cdot)\big\vert_{X^s}\ dt'.
\end{equation}
with $C_0= C(M,h_0^{-1},\delta_{\min}^{-1},\delta_{\max},\gamma_{\min}^{-1})$.
\end{Lemma}
\begin{proof}
We compute the inner product of~\eqref{eqn:W} with $ T[\W] \Lambda^{2s}W$, and obtain
\[
\big(\Lambda^s T[\W] \partial_t W , \Lambda^s W\big) \ + \ \big(\Lambda^s \Upsilon[\W]\partial_x W,\Lambda^s W\big)
\ = \ \big( \Lambda^s T[\W]\R , \Lambda^s W\big) \ ,
\]
where $T[\cdot]$ and $\Upsilon[\cdot]$ have been defined in the previous subsection.

From the symmetry of $T[\cdot]$ and $\Upsilon[\cdot]$, one deduces
\begin{multline}
\label{eqn:energyequalityXsW}
\frac12 \frac{d}{dt}E^s(W) \ = \frac12\big( \big[\partial_t, T[\W]\big]\Lambda^s W,\Lambda^s W\big)+\frac12\big( \big[\partial_x, \Upsilon[\W]\big] \Lambda^s W, \Lambda^s W\big) \\
-\big(\big[ \Lambda^s, T[\W] \big]\partial_t W , \Lambda^s W\big) \ - \ \big(\big[\Lambda^s , \Upsilon[\W]\big] \partial_x W,\Lambda^s W\big)
+
\big(\Lambda^s T[\W] \R,\Lambda^sW\big),
\end{multline}
where we define
\[ E^s(W)\equiv \big( T[\W] \Lambda^s W,\Lambda^s W\big).\]
We estimate below each of the terms in the right-hand side of~\eqref{eqn:energyequalityXsW}. 
\medskip

\noindent {\em Estimate of $\big(\big[\partial_t, T[\W]\big] \Lambda^s W,\Lambda^s W\big)$.} From~\eqref{bound:dSdSigma} in Lemma~\ref{C.S} (with $X=L^\infty(\RR)^4$), one has
\[ \big\vert \big[\partial_t, T[\W]\big] \Lambda^s W \big\vert_{L^2} \ \leq \ \big\vert \partial_t \W\big\vert_{L^\infty} C(\big\vert \W \big\vert_{L^\infty},\delta_{\min}^{-1},\delta_{\max},\gamma_{\min}^{-1} \big) \big\vert \Lambda^s W \big\vert_{L^2}.\]
By hypothesis, $\big\vert \partial_t \W\big\vert_{X^{s-1}}$ is controlled, and continuous Sobolev embedding for $s-1\geq s_0>1/2$ imply an equivalent control on the $L^\infty$-norm. One obtains simply
\[ \big\vert \big[\partial_t, T[\W]\big] \Lambda^s W \big\vert_{L^2} \ \leq \ M\ C(M,h_0^{-1},\delta_{\min}^{-1},\delta_{\max},\gamma_{\min}^{-1})\ \big\vert \Lambda^s W \big\vert_{L^2}.\]
It follows from the above and Cauchy-Schwarz inequality that
\begin{equation}\label{est:dts}
\big\vert \big( \big[\partial_t, T[\W]\big] \Lambda^s W,\Lambda^s W\big)\big\vert \ \leq \ \ C_0\ M\ \big\vert W \big\vert_{X^s}^2 \ ,
\end{equation}
with $C_0= C(M,h_0^{-1},\delta_{\min}^{-1},\delta_{\max},\gamma_{\min}^{-1})$.
\medskip

\noindent {\em Estimate of $\big( \big[\partial_x, \Upsilon[\W]\big] \Lambda^s W, \Lambda^s W\big) $.} As above, Cauchy-Schwarz inequality and Lemma~\ref{C.S} yield
\[
\big( \big[\partial_x, \Upsilon[\W]\big]\Lambda^s W, \Lambda^s W\big) \ \leq \ \big\vert \partial_x \W \big\vert_{L^\infty} C(\big\vert \W\big\vert_{L^\infty},h_0^{-1},\delta_{\min}^{-1},\delta_{\max},\gamma_{\min}^{-1}) \big\vert\Lambda^s W \big\vert_{L^2}^2 \ ,
\]
which is easily estimated thanks to continuous Sobolev embeddings. One obtains
\begin{equation}\label{est:dxs}
\big\vert \big( \big[\partial_x, \Upsilon[\W]\big] \Lambda^s W, \Lambda^s W\big) \big\vert \ \leq \ C_0\ M\ \big\vert W \big\vert_{X^s}^2 \ ,
\end{equation}
with $C_0= C(M,h_0^{-1},\delta_{\min}^{-1},\delta_{\max},\gamma_{\min}^{-1})$.
\medskip

\noindent {\em Estimate of $\big(\Lambda^s T[\W] \R,\Lambda^s W\big) $.} We apply Cauchy-Schwarz inequality and~\eqref{bound:SSigma} in Lemma~\ref{C.S}. One deduces
\begin{equation}\label{est:rems}
\big(\Lambda^s T[\W] \R,\Lambda^s W\big) \ \leq\ C_0 \big\vert W \big\vert_{X^s} \big\vert \R \big\vert_{X^s} \ ,
\end{equation}
with $C_0= C(M,h_0^{-1},\delta_{\min}^{-1},\delta_{\max},\gamma_{\min}^{-1})$.
\medskip

\noindent{\em Estimate of $\big(\big[\Lambda^s , \Upsilon[\W]\big] \partial_x W,\Lambda^s W\big)$.} We make use of Kato-Ponce's commutator estimate recalled in Lemma~\ref{L.KatoPonce}. It follows 
\[
\big\vert \big[\Lambda^s , \Upsilon[\W]\big] \partial_x W\big\vert_{L^2(\RR)^4}\lesssim \big\Vert \partial_x (\Upsilon[\W])\big\Vert_{X^{s-1}} \big\vert \partial_x W\big\vert_{X^{s-1}}.
\]
 From~\eqref{bound:dSdSigma} in Lemma~\ref{C.S}, and since $X^{s-1}$ is a Banach algebra, one has 
\[ \big\Vert \partial_x (\Upsilon[\W])\big\Vert_{X^{s-1}} \ \lesssim \ \big\vert \partial_x \W\big\vert_{X^{s-1}}C(\big\vert \W\big\vert_{X^{s-1}},h_0^{-1},\delta_{\min}^{-1},\delta_{\max},\gamma_{\min}^{-1})\ \lesssim \ M\ C(M,h_0^{-1},\delta_{\min}^{-1},\delta_{\max},\gamma_{\min}^{-1}).\]
It follows 
\begin{equation}\label{est:comxs}
\big\vert \big(\big[\Lambda^s , \Upsilon[\W]\big] \partial_x W,\Lambda^s W\big) \big\vert \ \leq \ \ C_0\ M\ \big\vert W \big\vert_{X^s}^2 \ ,
\end{equation}
with $C_0= C(M,h_0^{-1},\delta_{\min}^{-1},\delta_{\max},\gamma_{\min}^{-1})$.
\medskip

\noindent{\em Estimate of $\big(\big[ \Lambda^s, T[\W] \big]\partial_t W , \Lambda^s W\big) $.} As above, Kato-Ponce's commutator estimate yields
\[
\big\vert \big[ \Lambda^s, T[\W] \big]\partial_t W \big\vert_{L^2(\RR)^4}\lesssim \big\Vert \partial_x( T[\W])\big\Vert_{X^{s-1}} \big\vert \partial_t W\big\vert_{X^{s-1}} \lesssim M \big\vert \partial_t W\big\vert_{X^{s-1}} .
\]
Unfortunately, making use of the identity~\eqref{eqn:W} only yields $\big\vert \partial_t W\big\vert_{X^{s-1}}\lesssim \frac1\r \big\vert W\big\vert_{X^{s}}$, which is not sufficient to conclude. Thus we need now to use precisely the structure of our system, and in particular the estimate~\eqref{bound:PfdS}. Thus we decompose into two components:
\[
\big[ \Lambda^s, T[\W] \big]\partial_t W \ \equiv \ \big[ \Lambda^s, T[\W] \big]\Pi\partial_t W + \big[ \Lambda^s, T[\W] \big](\Id-\Pi)\partial_t W.\]

Let us start with the ``slow'' contribution, $\big[ \Lambda^s, T[\W] \big]\Pi\partial_t W$. One can use equation~\eqref{eqn:W} in order to control $\Pi\partial_t W$, uniformly with respect to $\r$ small. Indeed, one has
\[ \Pi \partial_t W \ = \ -\frac1\r \Pi L_\r\partial_x W - \Pi B[\W]\partial_x W+\Pi \R,\]
so that
\begin{align*}
 \big\vert \Pi\partial_t W\big\vert_{X^{s-1}} \ &\leq \ \big\vert \frac1\r \Pi L_\r \partial_x W\big\vert_{X^{s-1}}\ +\ \big\vert B[\W]\partial_x W \big\vert_{X^{s-1}}\ + \ \big\vert\mathcal{R} \big\vert_{X^{s-1}}, \\
 &\leq \ ( 1+M) C(M,h_0^{-1},\delta_{\min}^{-1},\delta_{\max},\gamma_{\min}^{-1})\big\vert \partial_x W \big\vert_{X^{s-1}} +\big\vert \R \big\vert_{X^{s-1}} ,
\end{align*}
where we used estimate~\eqref{bound:SSigma} in Lemma~\ref{C.S}, and the property $\big\Vert \Pi L_\r\big\Vert =\O(\r)$.
It follows
\[
\big\vert \big[ \Lambda^s, T[\W] \big]\Pi \partial_t W \big\vert_{L^2(\RR)^4}\leq C_0 \ M\ \big(\big\vert W\big\vert_{X^{s}}+\big\vert \R \big\vert_{X^{s-1}} \big),
\]
with $C_0=C(M,h_0^{-1},\delta_{\min}^{-1},\delta_{\max},\gamma_{\min}^{-1})$.

We continue with the ``fast'' contribution, $ \big[ \Lambda^s, T[\W] \big](\Id-\Pi)\partial_t W$. Since $(\Id-\Pi)$ is constant, it commutes with $\Lambda^s$, and Kato-Ponce's commutator estimates (Lemma~\ref{L.KatoPonce}) yield
\[
\big\vert \big[ \Lambda^s, T[\W] \big](\Id-\Pi) \partial_t W \big\vert_{L^2(\RR)^4}\lesssim \big\Vert \partial_x \left(T[\W](\Id-\Pi)\right)\big\Vert_{H^{s-1}} \big\vert \partial_t W\big\vert_{X^{s}}.
\]
Now, one has as above,
\[ \big\vert \partial_t W\big\vert_{X^{s}}\leq C_0 \big(\frac1\r \big\vert W\big\vert_{X^{s}}+\big\vert \R \big\vert_{X^{s-1}} \big),
\]
with $C_0=C(M,h_0^{-1},\delta_{\min}^{-1},\delta_{\max},\gamma_{\min}^{-1})$. Estimate~\eqref{bound:PfdS} in Lemma~\ref{C.S} allows to recover a factor of size $\O(\r)$:
\[ \big\Vert \partial_x \left(T[\W](\Id-\Pi)\right)\big\Vert_{H^{s-1}} \ \leq \ C_0 \ M\ \r ,\]
with $C_0=C(M,h_0^{-1},\delta_{\min}^{-1},\delta_{\max},\gamma_{\min}^{-1})$. Thus we proved
\[
\big\vert \big[ \Lambda^s, T[\W] \big]\Pi \partial_t W \big\vert_{L^2(\RR)^4}\lesssim C_0 \ M\ \big(\big\vert W\big\vert_{X^{s}}+\r \big\vert \R \big\vert_{X^{s-1}} \big),
\]
with $C_0=C(M,h_0^{-1},\delta_{\min}^{-1},\delta_{\max},\gamma_{\min}^{-1})$.

Altogether, one has, applying Cauchy-Schwarz inequality,
\begin{equation}\label{est:comts}
\big\vert \big(\big[ \Lambda^s, T[\W] \big]\partial_t W , \Lambda^s W\big) \big\vert \ \leq \ C_0\ M\ \big(\big\vert W\big\vert_{X^{s}}+\big\vert \R \big\vert_{X^{s-1}} \big)\big\vert W \big\vert_{X^s} \ ,
\end{equation}
with $C_0=C(M,h_0^{-1},\delta_{\min}^{-1},\delta_{\max},\gamma_{\min}^{-1})$.

Plugging~\eqref{est:dts},\eqref{est:dxs},\eqref{est:rems},\eqref{est:comxs},\eqref{est:comts} into~\eqref{eqn:energyequalityXsW} yields
\[\frac12 \frac{d}{dt}E^s(W) \ \leq \ C_0 \big( M \ \big\vert W \big\vert_{X^s}^2 \ + \ \big\vert \R\big\vert_{X^{s}}\big\vert W \big\vert_{X^s}\big).\]
Finally, estimate~\eqref{bound:elliptic} in Lemma~\ref{C.S} yields
\[\frac12 \frac{d}{dt}E^s(W) \ \leq \ C_0' \ M\ E^s(W) \ + \ C_0' \big\vert \R\big\vert_{X^{s}}E^s(W)^{1/2}.\]
with $C_0'=C(M,h_0^{-1},\delta_{\min}^{-1},\delta_{\max},\gamma_{\min}^{-1})$, and Lemma~\ref{L.HsW} follows from Gronwall-Bihari's Lemma.
 \end{proof}

\noindent {\bf Completion of the proof of Theorem~\ref{T.mr}.} Let us now quickly show how Theorem~\ref{T.mr} follows from Lemma~\ref{L.HsW}. For a given initial data as in the Theorem, Proposition~\ref{P.WPFS} yields the existence of $T_{\max}>0$ and a unique solution $U\equiv (\zeta_1,\zeta_2,u_1,u_2)^\top\in C([0,T_{\max});X^{s+1})\cap C^1([0,T_{\max}) X^{s})$ to~\eqref{FS} such that $U(t,\cdot)$ satisfies~\eqref{condH} for $t\in [0,T_{\max})$. It follows from Proposition~\ref{P.FsvsFS2} that the change of variables~\eqref{UtoV} yields $V\equiv (\zeta_1,\zeta_2,u_s,m)^\top\in C([0,T_{\max});X^{s+1})\cap C^1([0,T_{\max}); X^{s})$ solution to~\eqref{FS2}. 

Thanks to Proposition~\ref{P.ConsApp}, and since condition~\eqref{condH0} ensures that $(\zeta_2^0,u_s^0)^\top$ satisfies~\eqref{conditionSW}, one has $\Vapp=\Vrl+\Vr$ is well-defined and controlled for $t\in[0,T/M]$. More precisely, there exists $T^{-1},C_1=C(M,h_0^{-1},\delta_{\min}^{-1},\delta_{\max},\gamma_{\min}^{-1})$ such that
\begin{equation} \label{estV}
\sup_{t\in[0,T/M] } \big\{ \big\vert \Vapp(t,\cdot) \big\vert_{X^{s+1}}+ \big\vert \partial_t \Vapp(t,\cdot) \big\vert_{X^{s}}\big\} \ \leq \ C_1 M.\end{equation}
Denote $W\equiv V-\Vapp$. By construction, one has
\[ \big\vert W\id{t=0} \big\vert_{X^s} + \r \big\vert \partial_t W\id{t=0} \big\vert_{X^{s-1}} \ \leq \ C_2\ \r \ M,\]
with $C_2=C(M,h_0^{-1},\delta_{\min}^{-1},\delta_{\max},\gamma_{\min}^{-1})$. We introduce the time $T^\sharp$ as
\begin{equation} \label{estW}
T^\sharp \equiv \sup\left\{ t\in [0,T_{\max},T/M] , \ \big\Vert W\big\Vert_{L^\infty([0,t];X^s)}+\r\big\Vert \partial_t W\big\Vert_{L^\infty([0,t];X^{s-1})}\leq 2 C_2 \ \r\ M \right\}.\end{equation}
One has $T^\sharp>0$ since $W=V-\Vapp\in C([0,T]; X^{s+1})\cap C^1([0,T]; X^{s})$; our aim is to prove that $T^\sharp$ is uniformly bounded from below as in Theorem~\ref{T.mr}.

Recall that $W$ satisfies~\eqref{eqn:W}; thus we apply Lemma~\ref{L.HsW} with 
\[\R\equiv R- ( B[\Vapp+W] - B[\Vapp])\partial_x \Vapp.\]

Proposition~\ref{P.ConsApp} yields
\[\big\Vert R \big\Vert_{L^\infty([0,T/M];X^s)}\ \lesssim \ M\r\ \big( M + \r\big). \]
Now, using that $(X^s,\big\vert \cdot \big\vert_{X^s})$ is a Banach algebra, and using~\eqref{bound:B} in Lemma~\ref{C.B}, one has 
\[ \big\vert ( B[\Vapp+W] - B[\Vapp])\partial_x \Vapp \big\vert \ \lesssim C(\big\vert \Vapp\big\vert_{X^s},\big\vert W \big\vert_{X^s}) \big\vert W \big\vert_{X^s} \big\vert \partial_x \Vapp \big\vert_{X^s}.\]
It follows from the above estimates that
\begin{equation} \label{estR}
\big\Vert \R \big\Vert_{L^\infty([0,T^\sharp];X^s)}\ \leq \ C_3(M^2\r+M\r^2),\end{equation}
with $C_3=C(M,h_0^{-1}\delta_{\min}^{-1},\delta_{\max},\gamma_{\min}^{-1}) $.
\medskip

Finally, we apply~\eqref{energyestimateHsW} in Lemma~\ref{L.HsW} (making use of~\eqref{estV},\eqref{estW},\eqref{estR}), and deduce
\[
	\forall\ 0\leq t\leq T^\sharp,\qquad
	\big\vert W(t,\cdot) \big\vert_{X^s}\leq C_0 M\r e^{ C_0 M t}+ C_0(M\r+\r^2) ( e^{ C_0 M t} -1) \ ,
\]
with $C_0=C(M,h_0^{-1},\delta_{\min}^{-1},\delta_{\max},\gamma_{\min}^{-1})$. A similar estimate is obtained on $\partial_t W$, using the equation satisfied by $W$, namely~\eqref{eqn:W}:
\[ \big\vert \partial_t W \big\vert_{X^{s-1}} \ \leq \ \frac1\r C(M,h_0^{-1},\delta_{\min}^{-1},\delta_{\max},\gamma_{\min}^{-1}) \big\vert \partial_x W \big\vert_{X^{s-1}} \ + \ \big\vert \R \big\vert_{X^{s-1}}.\]
It follows that there exists $T'>0$, depending non-decreasingly on $M,h_0^{-1},\delta_{\min}^{-1},\delta_{\max},\gamma_{\min}^{-1}$, such that one has \[T^\sharp\geq \min\{ T_{\max},T'/M,T'/\r\}.\]

Triangular inequalities and~\eqref{estV},\eqref{estW} immediately yield
\begin{align}
\label{estslow}\big\Vert {\zeta_2} \big\Vert_{L^\infty([0,T^\sharp];H^s)}+ \big\Vert u_s \big\Vert_{L^\infty([0,T^\sharp];H^s)} \ &\leq \ M\exp(C_0 M t) \ ,\\
\label{estfast} \big\Vert {\zeta_1} \big\Vert_{L^\infty([0,T^\sharp];H^s)}+\big\Vert m \big\Vert_{L^\infty([0,T^\sharp];H^s)}\ &\leq \ M\r \exp(C_0 M t) \ , \\
\label{estdt} \big\Vert | \partial_t {\zeta_1}|+| \partial_t {\zeta_2} |+| \partial_t u_s|+| \partial_t m| \big\Vert_{L^\infty([0,T^\sharp];H^{s-1})} \ &\leq \ M \exp(C_0 M t) \ ,
\end{align}
with $C_0=C(M,h_0^{-1},\delta_{\min}^{-1},\delta_{\max},\gamma_{\min}^{-1})$.

It follows in particular from~\eqref{estdt} that for any $t\in [0,T^\sharp]$, one has
\[ \big\vert h_2(t,\cdot) - h_2(0,\cdot)\big\vert_{H^{s-1}} \ \leq \ \big\vert \int_0^t \partial_t \zeta_2(t',\cdot)\ dt' \big\vert_{H^{s-1}} \ \leq\ C(M,h_0^{-1},\delta_{\min}^{-1},\delta_{\max},\gamma_{\min}^{-1}) \ M \ t,\]
where we recall that $h_2\equiv \delta^{-1}+\zeta_2$. Similar estimates on $h_1\equiv 1+\r \zeta_1-\zeta_2$ and $u_1,u_2$ given by~\eqref{VtoU} show that $U(t,\cdot)$ satisfies condition~\eqref{condH} uniformly for $t\in [0,\min\{T^\sharp,T''/M\})$ (replacing $h_0$ with $h_0/2$), with ${T''}^{-1}=C(M,h_0^{-1},\delta_{\min}^{-1},\delta_{\max},\gamma_{\min}^{-1})$.

From the blow-up conditions stated in Proposition~\ref{P.WPFS} and a classical continuity argument, it is now clear that there exists $T>0$, depending only and non-decreasingly on $M,h_0^{-1},\delta_{\min}^{-1},\delta_{\max},\gamma_{\min}^{-1}$, such that $T_{\max}\geq T/\max\{M,\r\}$. 

The estimates in Theorem~\ref{T.mr} are a straightforward consequence of~\eqref{estW},~\eqref{estslow} and~\eqref{estfast} (using Lemma~\ref{L.Moser} and Corollary~\ref{C.depth}), and the proof of Theorem~\ref{T.mr} is now complete.

\section{Decomposition of the flow}\label{S.discussion}

In this section, we offer partial answers to two of the natural questions arising from Theorem~\ref{T.mr}:
\begin{enumerate}
\item {\em Can we describe more precisely the asymptotic behavior of the solution, and in particular the leading order deformation of the surface?}
\item {\em Can we extend the result to ill-prepared initial data, that is data which fail to meet the smallness assumption in~\eqref{condE0}?}
\end{enumerate}

In both cases, as we shall see, the answer will be given through a decomposition between fast and slow modes. Such decomposition is exact in the linear case ($\epsilon=0$ in~\eqref{FS}) as the the system becomes a linear wave equation; therefore the flow is a superposition of four traveling waves. Diagonalizing $\frac1\r L_\r$ (using the notation introduced in~\eqref{FS3}) shows that when $\r\to 0$, two of these waves (corresponding to the solution of the rigid-lid system, and mainly supported on variables $\zeta_2,u_s$) are traveling with velocity $c^s_\pm \sim \pm 1$, while the two other ones (mainly supported on $\zeta_1,m$) are traveling with velocity $c^f_\pm \sim \pm\sqrt{1+\delta^{-1}}/\r$.

This decomposition is far from being new. In the literature, the two modes are also often referred to as surface/interface modes, or barotropic/baroclinic modes, since the fast mode components share the properties of water-waves for one layer of a fluid of constant mass density~\cite{Gill82}. The decomposition is exact in the linear setting, and has been showed to hold approximately in the weakly nonlinear setting; see~\cite{Duchene11a}, and references therein. In that case, the smallness of $\epsilon$ allows to control the coupling effects between each of the waves (even when additional ---small--- dispersion terms are included), provided the initial data is sufficiently spatially localized.

Our aim in this section is to show that this decomposition is quite robust, and holds even when strong nonlinearities are involved. As already mentioned, such result will rely on a condition of spatial localization of the initial data, that we express through weighted Sobolev spaces.

In Section~\ref{S.cor}, we construct slow and fast mode correctors which allow to obtain a higher-order approximate solutions of the free-surface system, using only the corresponding solution to the rigid-lid system and the initial data. Thus we improve the results stated in Proposition~\ref{P.ConsApp} and Theorem~\ref{T.mr} with Proposition~\ref{P.ConsCorrector} and Theorem~\ref{T.mr2}, respectively. In Section~\ref{S.IP}, we extend the consistency result obtained in Proposition~\ref{P.ConsApp} to ill-prepared initial data, that is data allowing non-small horizontal momentum and deformation of the surface, and thus involving a leading order slow mode. Unfortunately, we cannot carry on the study of Section~\ref{S.completion}, and deduce the stronger result corresponding to Theorem~\ref{T.mr} (although numerical simulations are in full agreement with such result). Finally, subsection~\ref{S.numerics} contains numerical simulations illustrating the aforementioned results, and an accompanying discussion.

\begin{Remark}
Recall we set $\epsilon=1$ and $\alpha=\r$ after Theorem~\ref{T.mr}; see Remarks~\ref{R.epsilon} and~\ref{R.alpha}. The general setting, and therefore statements as in Theorem~\ref{T.mr} are easily recovered. We also implicitly assume that the constant $M$, which evaluates the magnitude of the initial perturbation, is bounded from below. More specifically, for technical reasons, we restrict our study to time interval $t\in [0,T]$ with $T^{-1}$ bounded, rather than $t\in[0,T/M]$ ---although, as discussed in Remark~\ref{R.epsilon}, we do not expect any particular limitation to occur when $M$ is small. 

As for Theorem~\ref{T.mr} (see Remark~\ref{R.r-not-small}), our statements do not impose the parameter $\r$ to be small, but are of little interest otherwise. In particular, our strategy of approximating the flow as the superposition of a fast and a slow mode approximate solution relies heavily on the fact that the fast mode is propagating with velocity $|c|\gtrsim 1/\r$, so that coupling effects are strong only during time interval of size $\O(\r)$ (since the two modes are localized away from each other afterwards).

If both $M$ and $\r$ are not small, then the initial perturbation will give rise to fast and slow modes of comparable magnitude and velocity. The two modes will therefore interact in a non-trivial, nonlinear way, and the full free-surface system is required to accurately describe the flow.
\end{Remark}

\subsection{Improved approximate solution}\label{S.cor}
In this section, we show that one can construct a first-order corrector to the rigid-lid approximate solution displayed in Theorem~\ref{T.mr}, provided the initial data is bounded in weighted Sobolev spaces. A key ingredient is the establishment of a fast mode corrector, which allows to take into account small initial data supported on variables $\zeta_1,m$.

In Proposition~\ref{P.ConsCorrector}, we provide a higher-order approximate solution to~\eqref{FS2} in the sense of consistency, {\em i.e.} similarly to Proposition~\ref{P.ConsApp}. One can then apply the strategy developed in Section~\ref{S.proof}, and one obtains the stronger result expressed in Theorem~\ref{T.mr2}, below.

\begin{Proposition}\label{P.ConsCorrector}
 Let $s\geq s_0,\ s_0>1/2$, and $\zeta_1^0,\zeta_2^0,u_s^0,m^0\in H^{s+1}(\RR)$, satisfying~\eqref{condE0},\eqref{condH0} (after the change of variable~\eqref{VtoU}) with given $0<M,h_0<\infty$. Assume additionally that there exists $\sigma>1/2$ such that
 \[ \big\vert (1+|\cdot|^2)^\sigma \zeta_1^0 \big\vert_{H^{s+1}}+\big\vert (1+|\cdot|^2)^\sigma m^0 \big\vert_{H^{s+1}}+\r \big\vert (1+|\cdot|^2)^\sigma \zeta_2^0 \big\vert_{H^{s+1}}+\r \big\vert (1+|\cdot|^2)^\sigma u_s^0 \big\vert_{H^{s+1}}\ \leq \ M\r \ .\]
Then there exists $0<T^{-1},C_0\leq C(M,h_0^{-1},\frac1{2\sigma-1},\delta_{\min}^{-1},\delta_{\max},\gamma_{\min}^{-1})$ such that 
 \begin{enumerate}
 \item $\Vrl\equiv (0,\eta,v,0)^\top$ is well-defined by Definition~\ref{D.VRL}, and satisfies
\[\forall t\in [0,T],\qquad \big\vert \Vrl \big\vert_{X^{s+1}} +\big\vert \partial_t \Vrl \big\vert_{X^{s}} \ \leq\ C_0\ M .
\]
 \item $\Vs\equiv (\r\breve{\zeta_1},0,0,0)^\top$ is well-defined with
 \[ \breve \zeta_1\ \equiv \ -\big(\eta+\frac{\delta}2\eta^2\big)- \frac{( 1-\eta)(\delta^{-1}+\eta)v^2}{(1+\delta^{-1})^2}. \]
 \item $\Vf$ is well-defined with
 \[ \Vf(t,x)\equiv \begin{pmatrix}u_+(x-c/\r t) +u_-(x+c/\r t)\\ 0\\ 0\\ c(u_+(x-c/\r t)-u_-(x+c/\r t)) \end{pmatrix},\]
 where $c\equiv \sqrt{1+\delta^{-1}}$, and $u_\pm(x) \ = \ \frac12\big(\zeta_1^0-\r \breve{\zeta_1}\id{t=0}\pm c^{-1} m^0\big)$.
 \item There exists $\Vr$, with
\[ \forall t\in [0,T],\qquad \big\vert \Vr(t,\cdot) \big\vert_{X^{s+1}_{\rm ul}} \ \leq\ C_0\ M\ ,
\]
 such that $\Vapp\equiv \Vrl+\Vs+\Vf+\r^2\Vr$ satisfies~\eqref{FS2} up to a remainder term, $R$, with
\[ \int_0^T \big\vert R(t,\cdot)\big\vert_{X^s}\ dt \ \leq \ C_0 \ M\ \r^2\ .\]
 \end{enumerate}
 \end{Proposition}
 \begin{Remark}\label{R.ul}
 We denote $(H^s_{\rm ul},\big\vert \cdot \big\vert_{H^s_{\rm ul}})$ the uniformly local Sobolev space introduced in~\cite{Kato75}:
 \[ \big\vert u \big\vert_{H^s_{\rm ul}} \ \equiv \ \sup_{j\in\NN} \big\vert \chi(\cdot -j) u(\cdot)\big\vert_{H^s},\]
 where $\chi$ is a smooth function satisfying $\chi\equiv 0$ for $|x|\geq 1$, $\chi\equiv 1$ for $|x|\leq 1/2$, and $\sum_{j\in \NN}\chi(x-j)=1$ for any $x\in\RR$ (the space is independent of the choice of $\chi$ satisfying these assumptions).
 
 We then denote $(X^s_{\rm ul},\big\vert \cdot \big\vert_{X^s_{\rm ul}})$ and $(L^\infty([0,T];X^s_{\rm ul}),\big\Vert \cdot \big\Vert_{L^\infty([0,T];X^s_{\rm ul})})$ similarly to the previously defined Sobolev-based spaces.
 \end{Remark}
 \begin{proof}[Proof of Proposition~\ref{P.ConsCorrector}] 
The well-posedness and estimate of $\Vrl$ for $t\in [0,T]$ has already been stated in Proposition~\ref{P.ConsApp} (here and thereafter, unless otherwise stated, we denote $T=\t T/M$ where $\t T$ is the constant used for the time intervals in the statements of Section~\ref{S.proof}). 
The definition of the corrector and remainder terms, as well as the desired estimates, is obtained in three steps. First we construct a high-order approximate solution corresponding to the initial data $\zeta_2^0,u_s^0$, using the corresponding solution to the rigid-lid system, and that we will refer to as {\em slow mode approximate solution}. Then we see how to construct the {\em fast mode approximate solution} in order to deal with the inadequacy of the slow mode approximate solution with regards to the initial data. Finally we show that, thanks to the localization in space of the initial data, the coupling effects between the two modes are weak, so that the superposition of the two contributions produces the desired approximate solution.
 \medskip
 
\noindent{\em Construction of the slow mode approximate solution.}
 We proceed as in the proof of Proposition~\ref{P.ConsApp}, but we propose a higher order definition for the corrector term, in order to reach the improved precision. More precisely, we seek $\Vapp^s\equiv \Vrl+ \Vs+\r^2\Vr$, with $\Vrl+\Vs\equiv (\r\breve\zeta_1,\eta,v,0)^\top$ as in the proof of Proposition~\ref{P.ConsApp}, and $\Vr\equiv (0,0,0,\breve m)^\top$ to be determined. Following the exact same steps as in the proof of Proposition~\ref{P.ConsApp}, we see that the only difficulty we face lies in the estimate of
 \[ r_1 \ = \ \r\left( \partial_t \breve \zeta_1 + \partial_x \breve m +\frac{\gamma+\delta}{\gamma} \partial_x\left( h_1\frac{\r^2\breve m-h_2 v}{h_1+h_2}\right)\right),\]
 where $\Vrl\equiv (0,\eta,v,0)^\top$ is the rigid-lid solution defined in Definition~\ref{D.VRL}, and $\breve \zeta_1 $ is defined in~\eqref{defbz1}.
 It is therefore natural to set
 \begin{equation}\label{defbmcor}
 \breve m(t,x) \ \equiv \ -\int_0^x \partial_t \breve \zeta_1(t,x')\ dx'+\delta \u h_1(t,x)h_2(t,x) v(t,x),\end{equation}
 where we denote $\u h_1\equiv 1-\eta$ and $h_2\equiv \delta^{-1}+\eta$.

Note that $\breve m$ may not have finite energy, since it does not necessarily decay when $x\to\pm\infty$. 
However, recall the estimates of Proposition~\ref{P.ConsApp}:
 \begin{align}\label{estVsWP} \forall t\in [0,T],\qquad \big\vert \Vrl \big\vert_{X^{s+1}} +\big\vert \partial_t \Vrl \big\vert_{X^{s}} \ &\lesssim\ C_0\ M,\\
\label{estz1WP} \forall t\in [0,T],\qquad \big\vert \breve \zeta_1 \big\vert_{H^{s+1}} +\big\vert \partial_t \breve \zeta_1 \big\vert_{H^{s}} \ &\lesssim\ C_0\ M.
\end{align}
(here and below, we denote $C_0=C(M,h_0^{-1},\delta_{\min}^{-1},\delta_{\max},\gamma_{\min}^{-1})$). One deduces
\begin{equation} \label{estVrsWP}\forall t\in [0,T],\qquad \big\vert \breve m \big\vert_{H^{s+1}_{\rm ul}} \ + \ \big\vert \partial_x \breve m \big\vert_{H^{s}} \ \lesssim \ C_0\ M,\end{equation}
where we use that $ H^s $ is continuously embedded in $H^{s}_{\rm ul}$ and $H^{s}_{\rm ul}$ is a Banach algebra, for any $s\geq s_0$ (see, {\em e.g.},~\cite[App.~B.4]{Lannes}).
 The estimate on $\Vr$, stated in the Proposition, is given by~\eqref{estVsWP},\eqref{estz1WP},\eqref{estVrsWP}.
 
 Note that~\eqref{estVrsWP} yields in particular, for any $f\in H^{s}$, $s\geq s_0$, that
\begin{multline} \label{timesm}
 \big\vert \breve m f \big\vert_{H^{s}} \ \leq \ \big\vert \breve m \Lambda^{s} f \big\vert_{L^2} + \big\vert \big[\Lambda^{s},\breve m \big] f \big\vert_{L^2} \lesssim \big\vert \breve{m}\big\vert_{L^\infty} \big\vert f\big\vert_{H^{s}}+\big\vert \partial_x \breve{m}\big\vert_{H^{\max\{s-1,s_0\}}} \big\vert f\big\vert_{H^{\max\{s-1,s_0\}}}\\ \lesssim \ C_0 \ M\ \big\vert f \big\vert_{H^{s}} ,
\end{multline}
where we used the commutator estimate recalled in Lemma~\ref{L.KatoPonce}. Using the above, it is now straightforward to check that $\Vapp^s\equiv \Vrl+\Vs+\r^2\Vr \equiv (\r\breve\zeta_1,\eta,v,\r^2 \breve m)^\top$ satisfies~\eqref{FS2}, up to a remainder term, $R^s$, with
 \begin{equation} \label{estRsWP}\big\Vert R^s\big\Vert_{L^\infty([0,T];H^s)} \ \lesssim \ C_0\ M\ \r^2.
 \end{equation}
 Here, we used the fact that the occurrences of $\breve m$ in~\eqref{FS2} are either of the form $\partial_x \breve m$, or $\breve m \times f$ with $f\in H^s$, and both of these contributions are bounded in $H^s$, thanks to~\eqref{estVrsWP} and~\eqref{timesm}.
\bigskip
 
 \noindent {\em Construction of the fast mode approximate solution.}
 The corrector $\Vf$ has been defined as the unique solution to
 \[ \partial_t \Vf \ + \ \frac1\r L_{(0)} \partial_x\Vf \ = \ 0 ,\quad \text{ where we recall } L_{(0)} \ \equiv \ \begin{pmatrix}
 0&0 & 0& 1 \\
 0 & 0 &0&0 \\
 0 &0 &0 & 0 \\
 1+\delta^{-1}& 0& 0 &0
 \end{pmatrix}, \]
 with initial data $\Vf\id{t=0}\equiv (\zeta_1^0-\r \breve{\zeta_1}\id{t=0},0,0,m^0)^\top$.
 
 Our aim is to prove that $\Vf$ is an approximate solution to~\eqref{FS2}. We recall that the system reads 
 \[ \partial_t V+ \frac1\r \left( L_\r \ + \ \r B[V]\right) \partial_x V=0,\quad \text{ with }
 L_\r \ \equiv \ \begin{pmatrix}
 0&0 & \frac{\gamma-1}{\gamma(\delta+1)}& \frac{\gamma+\delta}{\gamma(\delta+1)} \\
 0 & 0 & \frac{\r}{1+\delta}& \frac{\r}{1+\delta} \\
 0 &\r (\gamma+\delta) &0 & 0 \\
 \gamma(1+\delta^{-1})& \r\frac{\delta+\gamma}{\delta} & 0 &0
 \end{pmatrix} .\]
 
 Thus $\Vf $ satisfies
 \[ \partial_t \Vf+ \frac1\r \left( L_\r \ + \ \r B[\Vf]\right) \partial_x \Vf=R^f,\]
 with
 \[ R^f\equiv \frac1\r (L_\r-L_{(0)}) \partial_x \Vf \ + \ B[\Vf]\partial_x \Vf .\]

It is obvious that for any $t\in \RR$, $\Vf$ satisfies
 \begin{equation} \label{estVfWP}
 \big\vert \Vf(t,\cdot) \big\vert_{X^{s+1}} \ \lesssim\ \big\vert {\Vf}\id{t=0}\big\vert_{X^{s+1}} \ \leq \ C_0\ M\ \r,
 \end{equation}
where we used~\eqref{estz1WP} and the hypothesis on the initial data of the Proposition.
 
 In particular, Lemma~\ref{C.B} and Lemma~\ref{L.Moser} yield
 \begin{equation}\label{est:BV}
 \big\vert B[\Vf] \partial_x \Vf\ \big\vert_{X^s} \ \lesssim \ \big\vert \Vf\big\vert_{L^\infty(\RR)^4}\big\vert \Vf\big\vert_{X^{s+1}} \ \leq\ C_0\ M^2\ \r^2.
 \end{equation}
 
 Now, we use the fact that
 $(\Id-\Pi)\Vf=\Vf$ where we recall that $\Pi$ represents the orthogonal projection onto $\ker(L_{(0)})$:
 $\Id-\Pi\equiv\left( \begin{smallmatrix}1 & & & \\ & 0 & & \\ & & 0 & \\ & & & 1\end{smallmatrix}\right)$.
 
 It is straightforward to check that
 \[
 \big\Vert (L-L_{(0)})(\Id-\Pi) \big\Vert \ \lesssim \ \r^2,
 \]
 so that
 \begin{equation}\label{est:LV}
 \big\vert \frac1\r (L_\r-L_{(0)}) \partial_x \Vf \ \big\vert_{X^s} \ = \ \big\vert \frac1\r (L_\r-L_{(0)})(\Id-\Pi) \partial_x \Vf \ \big\vert_{X^s} \ \lesssim C_0\ M\ \r^2.
 \end{equation}
 
Estimates~\eqref{est:BV},\eqref{est:LV}, immediately yield the desired result: $\Vf$ satisfies~\eqref{FS2}, up to a remainder term, $R^f$, satisfying
 \begin{equation} \label{estRfWP} \big\Vert R^f\big\Vert_{L^\infty([0,T];H^s)} \ \lesssim \ C_0\ M\ \r^2.
 \end{equation}
 \bigskip
 
 \noindent {\em Completion of the proof.}
 One easily checks that $\Vapp \equiv \Vapp^s+\Vf$ satisfies
 \[\partial_t \Vapp+ \frac1\r \left( L_\r \ + \ \r B[\Vapp]\right) \partial_x \Vapp
 =R^f+R^s+R^c,\]
where
\[ R^c \ \equiv \ (B[\Vapp]-B[\Vf])\partial_x\Vf+ (B[\Vapp]-B[\Vapp^s])\partial_x \Vapp^s.\]
The contribution of $R^f+R^s$ is controlled as a result of the above calculations; see~\eqref{estRsWP} and~\eqref{estRfWP}. Thus the only remaining term to control is $R^c$, which contains the coupling effects between $\Vf$ and $\Vapp^s$. 

Note that similarly to~\eqref{bound:B} in Lemma~\ref{C.B}, one can check that estimates~\eqref{estVsWP}~\eqref{estz1WP},~\eqref{estVrsWP},~\eqref{timesm} and~\eqref{estVfWP} yield
\begin{equation}\label{estRc}
 \big\vert R^c\big\vert_{X^s} \ \leq \ C_0\ \times\ \left( \big\Vert \Vrl \otimes \partial_x\Vf \big\Vert_{X^s}+\big\Vert \Vf \otimes \partial_x\Vrl \big\Vert_{X^s} \ +\ M\ \r^2\right)
\end{equation}
with $C_0=C(M,h_0^{-1},\delta_{\min}^{-1},\delta_{\max},\gamma_{\min}^{-1})$, and where $U\otimes V$ denotes the outer product of $U$ and $V$.

In order to control the latter contribution, we make use of the fact that the initial data is assumed to be spatially localized. Thus $\Vf$ is the superposition of two spatially localized waves, with center of mass $x\approx \pm c/\r t$. It follows that the contribution of the outer products will decay after some time, provided one can prove that $\Vrl$ remains spatially localized around $x=0$ on the time interval $[0,T]$. This is where it is convenient, although certainly not necessary, to restrict ourselves to the time domain $t\in[0,T]$, with $T$ bounded, instead of the more stringent $t\in[0,\t T/M]$. Indeed, as it roughly propagates with velocity $\pm 1$, one cannot expect $\Vrl$ to remain spatially localized around $x=0$ during time interval $[0,T]$ with $T\gtrsim 1/M$, uniformly for $M$ small.

We state and prove below the persistence of the spatial decay which holds generically for a quasilinear, hyperbolic system; and complete the proof of Proposition~\ref{P.ConsCorrector} thereafter.

\begin{Lemma}[Persistence of spatial decay]\label{L.persistance}
Let $s\geq s_0+1,\ s_0>1/2$ and $\Vrl \equiv (\eta,v)^\top$ be the solution to~\eqref{RL}, with initial data $\Vrl\id{t=0} \equiv (\eta^0,v^0)^\top$ as above. Assume moreover that there exists $\sigma>0$ such that one has $\langle \cdot \rangle^\sigma \eta^0,\langle \cdot \rangle^\sigma v^0\in H^{s}$ (where we denote $\langle x \rangle\equiv (1+|x|^2)^{1/2}$). There exists $M>0$ such that if $\big\vert (\eta^0,v^0)^\top\big\vert_{H^s\times H^s}\leq M$, then one has
\[ \forall t\in [0,T], \qquad \big\vert \langle \cdot \rangle^\sigma \eta\big\vert_{H^s}+\big\vert \langle \cdot\rangle^\sigma v\big\vert_{H^s} \ \leq \ C\big(M,h_0^{-1},\big\vert \langle \cdot \rangle^\sigma \eta^0\big\vert_{H^s}+\big\vert \langle \cdot \rangle^\sigma v^0\big\vert_{H^s},\delta_{\min}^{-1},\delta_{\max}\big) .\]
\end{Lemma}
\begin{proof}[Proof of the Lemma]
Consider $W(t,x)=\langle x\rangle^\sigma \Vrl(t,x)$ (here and thereafter, multiplying a vector-valued function by $\langle x\rangle^\sigma$ means that all components are multiplied). One has
\[S[\Vrl]\partial_t \big(\langle \cdot \rangle^{-\sigma }W\big)+ \Sigma[\Vrl]\partial_x \big(\langle \cdot \rangle^{-\sigma} W\big)=0,\]
where $S[\cdot],\Sigma[\cdot]$ are smooth mappings onto the space of $2$-by-$2$ symmetric matrices ($S$ and $\Sigma$ are explicit; see~\cite{GuyenneLannesSaut10} for more details).

It follows, since the multiplication with $\langle \cdot \rangle^\sigma$ obviously commutes with $S[\cdot],\Sigma[\cdot],\partial_t$,
\[S[\Vrl ]\partial_t W + \Sigma[\Vrl]\partial_x W+\langle x\rangle^\sigma \partial_x \big(\langle x\rangle^{-\sigma} \big) \Sigma[\Vrl] W =0.\]

$S[\Vrl]$ is positive definite, so that there exists $0<c_0<\infty$ such that
\[ \frac1{c_0} \big\vert W \big\vert_{H^s(\RR)^2}^2 \ \leq \ E^s(W)\equiv \big( S[\Vrl]\Lambda^s W,\Lambda^s W\big) \ \leq \ c_0 \big\vert W \big\vert_{H^s(\RR)^2}^2 .\]

Using the usual technique for {\em a priori} $H^s$ estimates (see Lemma~\ref{L.Hs} for example), one obtains
\[ \frac{d}{dt}E^s(W)\leq C\big(\big\vert \Vrl\big\vert_{X^s} ,\big\vert \partial_t \Vrl\big\vert_{X^{s-1}} \big)\ E^s(W) + C\big(\big\vert \langle x\rangle^\sigma \partial_x \big(\langle x\rangle^{-\sigma}\big)\big\vert_{H^s},\big\vert \Vrl\big\vert_{X^s} \big)\ E^s(W)^{1/2} .\]

Now, using the control of $\Vrl\in X^s$ in~\eqref{estVsWP}, and since one has
\[ \big\vert \langle x\rangle^\sigma \partial_x \big(\langle x\rangle^{-\sigma} \big)\big\vert_{H^s}= \big\vert \sigma x \langle x\rangle^{-2} \big\vert_{H^s}\lesssim \sigma,\]
it follows from Gronwall-Bihari's inequality:
\[ E^s(W) \ \leq \ E^s(W\id{t=0}) \exp (C_0 t) \ + \ \int_0^t C_1 \exp (C_0 (t-t'))\ dt' ,\]
with $C_0,C_1=C\big(M,h_0^{-1},\big\vert \langle \cdot \rangle^\sigma \eta^0\big\vert_{H^s}+\big\vert \langle \cdot \rangle^\sigma v^0\big\vert_{H^s},\delta_{\min}^{-1},\delta_{\max}\big)$, and the Lemma is proved.
\end{proof}
\bigskip

Let us now complete the proof of Proposition~\ref{P.ConsCorrector}. We use the following calculation to estimate $R^c$ in~\eqref{estRc}. Set $s>1/2$, $\sigma> 0$, and $c\neq 0$. Let $u,v$ satisfy $\langle \cdot \rangle^\sigma v(t,\cdot)\in H^s$, and $\langle \cdot \rangle^\sigma u(\cdot)\in H^s$. Then one has
\[ \big\vert v(\cdot) u_\pm(\cdot- c/\r t) \big\vert_{H^s} \ \\ \lesssim \big\vert (1+|\cdot|^2)^\sigma v \big\vert_{H^s} \big\vert (1+|\cdot|^2)^\sigma u \big\vert_{H^s} \big\vert (1+|\cdot|^2)^{-\sigma} (1+|\cdot - c/\r t |^2)^{-\sigma} \big\vert_{H^s} 
 ,
\]
 and one can check (see~\cite{Lannes03} for example) that for any $\sigma>1/2$ and $T>0$, one has
 \[ \int_0^T \big\vert (1+|\cdot|^2)^{-\sigma} (1+|\cdot - c/\r t' |^2)^{-\sigma} \big\vert_{H^s} \ dt'\ \leq \ C(\frac1{2\sigma-1},\frac1c)\ \r,\]
thus uniformly bounded with respect to $1/\r$ and $T$.

It is now straightforward, applying Lemma~\ref{L.persistance}, the definition of $\Vf$,~\eqref{estVfWP} and the above calculations to~\eqref{estRc}, that the following estimate holds:
\begin{equation}\label{estRcWP}
\int_0^T\big\vert R^c(t',\cdot)\big\vert_{X^s}\ dt' \ \leq \ C_0\ M\ \r^2,
\end{equation}
with $C_0=C(M,h_0^{-1},\frac1{2\sigma-1},\delta_{\min}^{-1},\delta_{\max},\gamma_{\min}^{-1})$. 

Estimates~\eqref{estRsWP},~\eqref{estRfWP}, and~\eqref{estRcWP} conclude the proof of Proposition~\ref{P.ConsCorrector}.
 \end{proof}

Let us conclude this section with the following result, which corresponds to Theorem~\ref{T.mr}, when Proposition~\ref{P.ConsCorrector} is used instead of Proposition~\ref{P.ConsApp}.

\begin{Theorem} \label{T.mr2}
Let $s\geq s_0+1$, $s_0>1/2$. Let $\zeta_1^0,\zeta_2^0,u_1^0,u_2^0\in H^{s+1}(\RR)$ be such that~\eqref{condH0} holds with $h_0>0$ and there exists $0<M<\infty$ and $\sigma>1/2$ such that
\begin{equation}\label{condE2}
 \big\vert (1+|\cdot|^2)^\sigma \zeta_2^0 \big\vert_{H^{s+1}}\ + \ \big\vert (1+|\cdot|^2)^\sigma (u_2^0 - \gamma u_1^0) \big\vert_{H^{s+1}}\ \ \leq \ M,
 \end{equation}
and
\begin{equation}\label{condWP2}
\big\vert (1+|\cdot|^2)^\sigma\zeta_1^0 \big\vert_{H^{s+1}} \ + \ \big\vert (1+|\cdot|^2)^\sigma\big(\gamma h_1^0 u_1^0\ + \ h_2^0 u_2^0\big) \big\vert_{H^{s+1}} \ \leq \ M\r.
 \end{equation}
Then there exists $T^{-1},C$, depending non-decreasingly on $M,h_0^{-1},\frac1{s_0-1/2},\frac1{2\sigma-1},\delta_{\min}^{-1},\delta_{\max},\gamma_{\min}^{-1}$, such that
one can uniquely define $U\in C([0,T];X^{s+1})\cap C^1([0,T];X^{s}) $, the solution to~\eqref{FS} with initial data $U\id{t=0}= (\zeta_1^0,\zeta_2^0,u_1^0,u_2^0)^\top$; and $\Vrl,\Vs,\Vf$ as in Proposition~\ref{P.ConsCorrector}. Denote $\Uapp$ the approximate solution corresponding to $\Vrl+\Vs+\Vf$, after the change of variables~\eqref{VtoU}. Then one has
\[\big\Vert U-\Uapp \big\Vert_{L^\infty([0,T];X^s_{\rm ul})} \ \leq \ C\ M\ \r^2\ .\]
\end{Theorem}
\begin{proof}[Sketch of the proof.]
The existence and uniqueness of $U$ has already been stated in Theorem~\ref{T.mr}. The existence and uniqueness of $\Vrl,\Vs,\Vf$ is guaranteed by Proposition~\ref{P.ConsCorrector}. Now, one can follow the exact same procedure as described in Section~\ref{S.proof} (and especially Section~\ref{S.completion}), using the result of Proposition~\ref{P.ConsCorrector} instead of the corresponding Proposition~\ref{P.ConsApp}. Note however that the remainder term constructed in Proposition~\ref{P.ConsCorrector}, $\Vr$, may not have finite $H^s$ norm; thus we need to work with uniformly local Sobolev spaces, defined in Remark~\ref{R.ul}.

However, as initially remarked by Kato~\cite{Kato75}, the energy method for hyperbolic quasilinear systems in Sobolev spaces extends naturally to uniformly local Sobolev spaces, without significant change in the proof (in particular, similar product and commutator estimates hold; see~\cite[App.~B]{Lannes}); thus we do not detail further on.

We simply remark that $\Vapp$ has been constructed so that $W\equiv V-\Vapp$ satisfies
\[ \big\vert W\id{t=0}\big\vert_{X^s_{\rm ul}} \ \lesssim \ C_0 M \r^2,\]
where we denote $V\equiv (\zeta_1,\zeta_2,u_s,m)^\top$ the solution to~\eqref{FS2} corresponding to $U$, in terms of the variables defined by~\eqref{UtoV}.
Consequently, the energy estimate~\eqref{energyestimateHsW} in Lemma~\ref{L.HsW} implies
\[ \forall t\in [0,T],\quad \big\vert W\big\vert_{X^s_{\rm ul}} \ \lesssim \ C_0M \r^2 \ +\ \int_0^t \big\vert R(t',\cdot) \big\vert_{X^s_{\rm ul}}dt',\]
and Proposition~\ref{P.ConsCorrector} immediately yields the desired estimate.
\end{proof}

\subsection{The case of ill-prepared initial data}\label{S.IP}
In this section, we are concerned with the case of ill-prepared initial data, that is initial data which fail to meet the smallness assumption in~\eqref{condE0}, or in other words {\em admitting a non-small fast mode}. Once again, we construct an approximate solution as the superposition of a slow mode approximate solution, obtained from the corresponding solution to the rigid-lid system~\eqref{RL}, and a fast mode approximate solution, that we shall exhibit below. There are two main differences with the previous results, due to the fact that the slow mode approximate solution is no longer of size $\O(\r)$:
\begin{enumerate}
\item Nonlinear effects have a non-trivial contribution on the behavior of the fast mode approximate solution, and cannot be neglected.
\item The strategy developed in Section~\ref{S.proof} is not valid anymore, as the hypothesis of Lemma~\ref{L.HsW} is no longer satisfied.
\end{enumerate}
As a consequence of the latter point, we restrict our statement to a consistency result, namely Proposition~\ref{P.ConsIP}, below; we cannot deduce an estimate on the difference between the exact and the approximate solution, as in Theorems~\ref{T.mr} and~\ref{T.mr2}, or even prove that~\eqref{FS} is well-posed on a time interval independent of~$\r$ small. However, numerical simulations, presented in the subsequent subsection, are in full agreement with the intuitive conjecture that
\[ \big\Vert V- \Vrl - \Vf \big\Vert_{L^\infty([0,T];X^s)} \ = \ \O(\r),\] 
with the notations introduced below.
\begin{Proposition}\label{P.ConsIP}
 Let $s\geq s_0,\ s_0>1/2$, and $\zeta_1^0,\zeta_2^0,u_s^0,m^0\in H^{s+1}(\RR)$, satisfying~\eqref{condH0} (after the change of variable~\eqref{VtoU}) with given $h_0>0$. Assume additionally that there exists $0<M<\infty$ and $\sigma>1/2$ such that
 \[ \big\vert (1+|\cdot|^2)^\sigma \zeta_1^0 \big\vert_{H^{s+2}}+\big\vert (1+|\cdot|^2)^\sigma m^0 \big\vert_{H^{s+2}}+\big\vert (1+|\cdot|^2)^\sigma \zeta_2^0 \big\vert_{H^{s+2}}+ \big\vert (1+|\cdot|^2)^\sigma u_s^0 \big\vert_{H^{s+2}}\ \leq \ M \ .\]
Then there exists $0<T^{-1},C_0\leq C(M,h_0^{-1},\frac1{2\sigma-1},\delta_{\min}^{-1},\delta_{\max},\gamma_{\min}^{-1})$ such that 
 \begin{enumerate}
 \item $\Vrl\equiv(0,\eta,v,0)^\top$ is well-defined by Definition~\ref{D.VRL}, and satisfies
\[\forall t\in [0,T],\qquad \big\vert \Vrl \big\vert_{X^{s+2}} +\big\vert \partial_t \Vrl \big\vert_{X^{s+1}} \ \leq\ C_0\ M .
\]
 \item $\Vf$ is well-defined with
 \[\Vf(t,x)\equiv \begin{pmatrix}u_+(t,x) +u_-(t,x)\\ 0\\ 0\\ c(u_+(t,x)-u_-(t,x)) \end{pmatrix},\]
 where $c\equiv \sqrt{1+\delta^{-1}}$, and $u_\pm$ is the unique solution to
\[ \partial_t u_\pm \ \pm \ \frac{c}\r \partial_x u_\pm \ \pm \ \frac{3}{2c} u_\pm\partial_x u_\pm \ = \ 0,\]
with ${u_\pm}\id{t=0} \ = \ \frac12\big(\zeta_1^0\pm c^{-1} m^0\big)$.
 \item There exists $\Vr$ with
\[ \forall t\in [0,T],\qquad \big\vert \Vr \big\vert_{X^{s+1}} +\r \big\vert \partial_t \Vr \big\vert_{X^{s}} \ \leq\ C_0\ M\ ,
\]
 such that $\Vapp\equiv \Vrl+\Vf+\r\Vr$ satisfies~\eqref{FS2}, up to a remainder term, $R$, with
\[\int_0^T \big\vert R(t,\cdot) \big\vert_{X^s} \ dt \ \leq \ C_0\ M\ \r\ .\]
 \end{enumerate}
 \end{Proposition}
 \begin{Remark}\label{remVf} The fast mode contribution $\Vf$ is different from the one defined in Proposition~\ref{P.ConsCorrector}. Moreover, it is not a corrector term {\em per se}, since it has the same order of magnitude as $\Vrl$. We decide to use the same notation in order to acknowledge the following fact: one can replace $\Vf$ in Proposition~\ref{P.ConsCorrector} by the one defined above, without modifying the rest of the statement; nonlinear effects on the fast mode component are negligible in the case of well-prepared initial data.
 \end{Remark}
\begin{proof} [Proof of Proposition~\ref{P.ConsIP}]
We follow the same three steps as in the proof of Proposition~\ref{P.ConsCorrector}. We first construct an approximate solution corresponding to the slow mode and fast mode, respectively. Finally, we prove that the coupling effects between the two modes are weak, thanks to the appropriate spatial localization of the initial data, and therefore the superposition of the two modes yields an approximate solution.
\medskip

 \noindent {\em Construction of the slow mode approximate solution.} Proposition~\ref{P.ConsApp} directly gives the desired result: denoting $\Vr^s\equiv (\breve\zeta_1,0,0,\r \breve m)$, with $\breve\zeta_1,\breve m$ defined in~\eqref{defbz1},\eqref{defbm}, one has
 \begin{align}\label{estVsIP} \forall t\in [0,T],\qquad \big\vert \Vrl \big\vert_{X^{s+2}} +\big\vert \partial_t \Vrl \big\vert_{X^{s+1}} \ &\lesssim\ C_0\ M,\\
\label{estVrsIP} \forall t\in [0,T],\qquad \big\vert \Vr^s \big\vert_{X^{s+2}} +\big\vert \partial_t \Vr^s \big\vert_{X^{s+1}} \ &\lesssim\ C_0\ M,
\end{align}
and $\Vapp^s\equiv\Vrl+\r\Vr^s$ satisfies~\eqref{FS2} up to a remainder term, $R^s$, with
\begin{equation}\label{estRsIP} \big\Vert R^s \big\Vert_{L^\infty([0,T];X^{s+1})}\ \lesssim \ C_0 \ M (M\ \r+\r^2) \lesssim \ C_0 \ M\ \r\ ,\end{equation}
with $C_0=C(M,h_0^{-1},\delta_{\min}^{-1},\delta_{\max},\gamma_{\min}^{-1})$. As previously, the first steps of the proof are valid with $T=\t T/M$, but the last step ---as it uses the localization in space of the two modes--- asks for $T$ to be uniformly bounded.
\medskip

 \noindent {\em Construction of the fast mode approximate solution.}
 We recall that~\eqref{FS2} reads 
 \[
 \partial_t V+ \frac1\r \left( L_\r \ + \ \r B[V]\right) \partial_x V=0,\]
with $V\equiv (\zeta_1,\zeta_2,u_s,m)^\top$.
We denote $ L_\r\ \equiv \ L_{(0)}+\r L_{(1)}+\O(\r^2)$, with
\[ L_{(0)} \ \equiv \ \begin{pmatrix}
0&0 & 0& 1 \\
0 & 0 & 0&0\\
0 &0&0 & 0 \\
1+\delta^{-1}& 0 & 0 &0
\end{pmatrix}, \quad L_{(1)} \ \equiv \ \begin{pmatrix}
0&0 & 0& 0 \\
0 & 0 & \frac{1}{1+\delta}& \frac{1}{1+\delta} \\
0 &\gamma+\delta&0 & 0 \\
0& \frac{\delta+1}{\delta} & 0 &0
\end{pmatrix}.\]

One can also check that $ B[(\zeta_1,0,0,m)^\top]\equiv B_{(1)}[(\Id-\Pi)V]+\O(\r)$, with
\[ B_{(1)}[(\Id-\Pi)V] \ \equiv \ \begin{pmatrix}
0&0 & 0& 0\\
0 & \frac{\delta}{\delta+1}m & 0&0 \\
0 &0 &\frac{\delta}{\delta+1}m & 0 \\
\zeta_1 & 0 & 0 &2\frac{\delta}{\delta+1}m 
\end{pmatrix}.\]

In the following, we seek an approximate solution to
\begin{equation}\label{FS2fm}
\partial_t V+ \left( \frac1\r L_{(0)}\ + \ L_{(1)} \ + \ B_{(1)}[(\Id-\Pi)V]\right) \partial_x V=0,
\end{equation}
with initial data satisfying $(\Id-\Pi)V\id{t=0}=V\id{t=0}$.

Our strategy is based on a WKB-type expansion, namely we seek an approximate solution to~\eqref{FS2fm} under the form
\[ \Vapp^f(t,x) \ = \ \Vf(t,t/\r,x) \ + \ \r \Vr^f(t,t/\r,x),\]
where (with a straightforward abuse of notation) $\Vapp^f(t,\tau,x)$ is an approximate solution to
\begin{equation}\label{BKW}
\frac1\r \partial_\tau \Vapp^f+ \partial_t \Vapp^f+ \left( \frac1\r L_{(0)}\ + \ L_{(1)} \ + \ B_{(1)}[(\Id-\Pi)\Vapp^f]\right) \partial_x \Vapp^f=0.
\end{equation}

Based on the fact that at first order (in terms of $\r$), the system~\eqref{BKW} is a simple linear equation, $\partial_\tau V \ +\ L_{(0)}\partial_x V \ = \ 0$, and from the assumption on the initial data, we set $\Vf$ as the superposition of decoupled waves, supported on the eigenvectors of $L_{(0)}$ corresponding to non-zero eigenvalues. 

The analysis of higher-order terms yields
\begin{itemize}
\item the behavior of $\Vf$ with respect to the large-time variable, $t$, which takes into account the nonlinear effects on the propagation of each decoupled waves;
\item a remainder term, $\Vr^f(t,\tau,x)$, which mimics the coupling effects between the two counter-propagating waves of $\Vf$, as well as the ``slow mode component'', $\Pi\Vf$. 
\end{itemize}
The key ingredient in the proof is to show that one can set $\Vf$ such that $\Vr^f$ remains small for large time.
This strategy has been applied notably to the rigorous justification of the Korteweg-de Vries equation as a model for the propagation of surface waves in the long wave regime~\cite{SchneiderWayne00,BonaColinLannes05}, and later on to similar problems in the bi-fluidic setting~\cite{Duchene11a,Duchene13}. The strategy is described comprehensively for example in~\cite[Chap. 7]{Lannes}, thus we do not detail the calculations, and simply state the outcome.
\medskip

It is convenient to introduce here the following eigenvectors of $L_{(0)}$:\footnote{Of course a fourth vector ---second linearly independent element of $\ker(L_{(0)})$--- could be defined, but this is not necessary in our analysis.}
\[ \e_+ \ = \ \begin{pmatrix}
1\\ 0 \\ 0 \\ c
\end{pmatrix}, \quad \e_- \ = \ \begin{pmatrix}
1\\ 0 \\ 0 \\ -c
\end{pmatrix}, \quad \e_0 \ = \ \begin{pmatrix}
0\\ 1 \\ 0 \\ 0
\end{pmatrix}\]

We set
 \[\Vf(\cdot,\tau,x)\equiv u_+(\cdot,x-c\tau )\e_+ \ +\ u_-(\cdot,x+c\tau)\e_-,\]
 where $u_\pm(t,y)$ is uniquely defined by
 \[ \partial_t u_\pm \ \pm\ \frac{3}{2c}u_\pm \partial_y u_\pm \ = \ 0,\]
 with $u_\pm\id{t=0}= \frac12\big(\zeta_1^0\pm c^{-1} m^0\big)$. One checks immediately that $\Vf:(t,x)\mapsto \Vf(t,t/\r,x)$ is as in the Proposition, explaining our (slightly misused) notation.

In the same way, we write
\[ \Vr^f(\cdot,\tau,x)\equiv r_+(t,\tau,x ) \e_+ +r_-(t,\tau,x)\e_- + r_0(t,\tau,x)\e_0 ,\]
with functions $r_+,r_-,r_0$ determined by
\begin{align*}
\partial_\tau r_+(\cdot,\tau,x)+c\partial_x r_+(\cdot,\tau,x) + \frac{3}{4c}\partial_x \big(u_-(\cdot,x-c\tau)^2\big)-\frac1{2c}\partial_x\big(u_-(\cdot,x-c\tau)u_+(\cdot,x+c\tau)\big) \ &=\ 0 ,\\
\partial_\tau r_-(\cdot,\tau,x)-c\partial_x r_+(\cdot,\tau,x) - \frac{3}{4c}\partial_x \big(u_+(\cdot,x-c\tau)^2\big)+\frac1{2c}\partial_x\big(u_-(\cdot,x-c\tau)u_+(\cdot,x+c\tau)\big) \ &=\ 0 ,\\
\partial_\tau r_0(\cdot,\tau,x)\ + \ \frac1{\delta c}\partial_x\big(u_+(\cdot,x+c\tau)-u_-(\cdot,x-c\tau)\big) \ &=\ 0 ,
\end{align*}
and $\Vr^f(\cdot,0,\cdot)\equiv 0$.

 One can check that $\Vapp^f(t,\tau,x) \ = \ \Vf(t,\tau,x) \ + \ \r \Vr^f(t,\tau,x)$, as defined above, satisfies
\[\frac1\r \partial_\tau \Vapp^f+ \partial_t \Vapp^f+ \left( \frac1\r L_{(0)}\ + \ L_{(1)} \ + \ B_{(1)}[\Vapp^f]\right) \partial_x \Vapp^f=R^f,\]
with $R^f\equiv \r \partial_t\Vr^f + \r L_{(1)}\partial_x \Vr^f + B_{(1)}[\Vapp^f] \partial_x \Vapp^f -B_{(1)}[\Vf]\partial_x \Vf$.

It follows (using~\eqref{bound:B} in Lemma~\ref{C.B}) that
\begin{equation}\label{estRfIP0}
\big\vert R^f \big\vert_{X^s} \ \leq\ \r\ C\big( \big\vert \partial_t \Vr^f \big\vert_{X^s},\big\vert  \Vr^f \big\vert_{X^{s+1}},\big\vert \Vf \big\vert_{X^{s+1}}\big).\end{equation}

In order to estimate the above, one needs to control $\Vr^f$, using the following two Lemmata.
\begin{Lemma}
\label{L.easy}
 Let $s\geq 0$, and $f^0 \in H^{s}(\RR)$. Then there exists a unique global strong solution, $u(\tau,x)\in C^0(\RR; H^{s})\cap C^1(\RR; H^{s-1})$, of 
\[
 \left\{\begin{array}{l}
 (\partial_\tau+c_1\partial_x) u=\partial_x f \\
 u\id{t=0}=0
 \end{array}\right. \ \text{ with }\ \ \quad
 \left\{\begin{array}{l}
 (\partial_\tau+c_2\partial_x) f =0\\
 {f_i}\id{t=0}=f^0
 \end{array}\right.
\]
where $c_1\neq c_2$. Moreover, one has the following estimates for any $\tau\in\RR$:
\[ \big\vert u(\tau,\cdot) \big\vert_{H^{s}(\RR)} \ \leq \ \frac{2}{|c_1-c_2|}\big\vert f^0 \big\vert_{H^{s}(\RR)}.\]
\end{Lemma}
\begin{Lemma} \label{L.Lannes} Let $s\geq s_0>1/2$, and $v^0_1$, $v^0_2 \in H^{s}(\RR)$. Then there exists a unique global strong solution, $u\in C^0(\RR; H^{s})$, of 
\[
 \left\{\begin{array}{l}
 (\partial_\tau+c\partial_x) u=g(v_1,v_2)\\
 u\id{t=0}=0
 \end{array}\right. \ \text{ with }\ \quad\forall i\in\{1,2\}, \quad
 \left\{\begin{array}{l}
 (\partial_\tau+c_i\partial_x) v_i=0\\
 {v_i}\id{t=0}=v^0_i
 \end{array}\right.
\]
where $c_1\neq c_2$, and $g$ is a bilinear mapping defined on $\RR^2$ and with values in $\RR$. 
Assume moreover that there exists $\sigma>1/2$ such that $v^0_1 (1+|\cdot|^2)^\sigma$, and $v^0_2 (1+|\cdot|^2)^\sigma\in H^{s}(\RR)$, then one has the (uniform in time) estimate 
\[ \big\Vert u \big\Vert_{L^\infty(\RR; H^{s}(\RR))}\ \leq\ C(\frac1{c_1-c_2},\frac1{\sigma-1/2}) \ \big\vert v^0_1 (1+|\cdot|^2)^\sigma \big\vert_{H^{s}(\RR)}\big\vert v^0_2 (1+|\cdot|^2)^\sigma \big\vert_{H^{s}(\RR)}.\]
 \end{Lemma}
 Lemma~\ref{L.easy} is straightforward, and Lemma~\ref{L.Lannes} follows from Proposition~3.5 in~\cite{Lannes03}.
 \bigskip

Lemmata~\ref{L.easy} and~\ref{L.Lannes} applied to $\Vr^f$ immediately yield
\[
\big\vert \Vr^f(t,\tau,\cdot) \big\vert_{X^{s+1}}\ \leq \ C \big\vert u_\pm(t,\cdot) \big\vert_{H^{s+1}(\RR)}+C \big\vert u_+(t,\cdot) (1+|\cdot|^2)^\sigma \big\vert_{H^{s+2}(\RR)}\big\vert u_-(t,\cdot) (1+|\cdot|^2)^\sigma \big\vert_{H^{s+2}(\RR)}.
\]
One can apply the same arguments to $\partial_t\Vr^f$ (differentiating the equations satisfied by $r_\pm$ and $r_0$ with respect to the parameter $t$, and using $\partial_t u_\pm(t,y) = \mp \frac{3}{2c}u_\pm(t,y) \partial_y u_\pm(t,y)$), and one obtains
\[
\big\vert \partial_t \Vr^f(t,\tau,\cdot) \big\vert_{X^{s}} \ \leq \ C \big\vert u_\pm(t,\cdot) \big\vert_{H^{s+1}(\RR)}+C \big\vert u_+(t,\cdot) (1+|\cdot|^2)^\sigma \big\vert_{H^{s+2}(\RR)}\big\vert u_-(t,\cdot) (1+|\cdot|^2)^\sigma \big\vert_{H^{s+2}(\RR)}.
\]

It is not difficult to show that the inviscid Burgers' equation propagates locally in time the localization in space of its solutions (see Lemma~\ref{L.persistance}), so that one has
\begin{equation}\label{eq:persistanceBurg} \forall t\in [0,T], \quad \big\vert u_\pm(t,\cdot) (1+|\cdot|^2)^\sigma \big\vert_{H^{s+2}(\RR)} \ \lesssim \ \big\vert u_\pm(0,\cdot) (1+|\cdot|^2)^\sigma \big\vert_{H^{s+2}(\RR)} \ \leq \ M,\end{equation}
thus we proved
\[ \forall (t,\tau)\in [0,T]\times\RR , \quad \big\vert \Vr^f(t,\tau,\cdot) \big\vert_{X^{s+1}} +\big\vert \partial_t \Vr^f(t,\tau,\cdot) \big\vert_{X^{s}} \ \leq \ C_0\ M ,\]
with $C_0=C(M,\frac1{2\sigma-1},\delta_{\min}^{-1},\delta_{\max},\gamma_{\min}^{-1})$.

Finally, we recall that $\Vf\equiv u_+(t,x-ct/\r)\e_++u_-(t,x+ct/\r)\e_-$ and $\Vr^f\equiv \Vr^f(t,t/\r ,x)$, and one deduces 
 \begin{align}\label{estVfIP} \forall t\in [0,T],\qquad \big\vert \Vf \big\vert_{X^{s+2}} +\r \big\vert \partial_t \Vf \big\vert_{X^{s+1}} \ &\leq\ C_0\ M,\\
\label{estVrfIP} \forall t\in [0,T],\qquad \big\vert \Vr^f \big\vert_{X^{s+1}} +\r\big\vert \partial_t \Vr^f \big\vert_{X^{s}} \ &\leq\ C_0\ M,
\end{align}
with $C_0=C(M,\frac1{2\sigma-1},\delta_{\min}^{-1},\delta_{\max},\gamma_{\min}^{-1})$.

Therefore~\eqref{estRfIP0} simply becomes
\begin{equation}\label{estRfIP}
\big\Vert R^f \big\Vert_{L^\infty([0,T];X^s)}\ \leq\ C_0\ M \ \r,\end{equation}
with $C_0=C(M,\frac1{2\sigma-1},\delta_{\min}^{-1},\delta_{\max},\gamma_{\min}^{-1})$.
\bigskip

 \noindent {\em Completion of the proof.}
 One easily checks that $\Vapp \equiv \Vapp^s+\Vapp^f\equiv \Vrl+\Vf+\r\Vr^s+\r\Vr^f$ satisfies
 \[ \partial_t \Vapp+ \frac1\r \left( L_\r \ + \ \r B[\Vapp]\right) \partial_x \Vapp\\
 =R^s+R^s+R^c,\]
 where $R^s$ and $R^f$ have been defined and estimated above, and with 
 \[R^c\ \equiv\ (B[\Vapp]-B[\Vapp^f])\partial_x\Vapp^f+(B[\Vapp]-B[\Vapp^s])\partial_x \Vapp^s.\]

The contribution of $R^f+R^s$ is controlled as a result of the above calculations; see~\eqref{estRsIP} and~\eqref{estRfIP}. Thus the only component to control comes from the coupling effects between $\Vapp^s$ and $\Vapp^f$, displayed in $R^c$. Recalling the construction of $\Vapp^s\equiv \Vrl+\r \Vr^s$ and $\Vapp^f\equiv \Vf+\r \Vr^f$, using estimates~\eqref{estVsIP},~\eqref{estVrsIP},~\eqref{estVfIP},\eqref{estVrfIP}, one can check that 
\[ \big\vert R^c\big\vert_{X^s} \ \leq \ C_0\times \left( \big\Vert \Vrl\otimes \partial_x\Vf \big\Vert_{X^s}+\big\Vert \Vf \otimes \partial_x\Vrl \big\Vert_{X^s} \ + \ M \r \right),\]
with $C_0=C(M,h_0^{-1},\frac1{2\sigma-1},\delta_{\min}^{-1},\delta_{\max},\gamma_{\min}^{-1})$, and again $U\otimes V$ is the outer product of $U$ and $V$.

We estimate the above as in the proof of Proposition~\ref{P.ConsCorrector}, using spatial localization. For any function $v$ satisfying $(1+|\cdot|^2)^\sigma v(t,\cdot)\in H^s$, one has
\begin{multline*} \big\vert v(t,\cdot) u_\pm(t,\cdot\mp c/\r t) \big\vert_{H^s} \ \\ \lesssim \big\vert (1+|\cdot|^2)^\sigma v(t,\cdot) \big\vert_{H^s} \big\vert (1+|\cdot|^2)^\sigma u_\pm(t,\cdot) \big\vert_{H^s} \big\vert (1+|\cdot|^2)^{-\sigma}(1+|\cdot \mp c/\r t|^2)^{-\sigma} \big\vert_{H^s} 
,
\end{multline*}
and recall that for any $\sigma>1/2$ and $t>0$, one has
\[ \int_0^t \big\vert (1+|\cdot|^2)^{-\sigma}(1+|\cdot \mp c/\r t'|^2)^{-\sigma} \big\vert_{H^s} \ dt'\ \leq \ C\big(\frac1{2\sigma-1},\frac1c\big) \ M\ \r,\]
 thus uniformly bounded with respect to $1/\r$ and $T$.

Thus it follows from Lemma~\ref{L.persistance} and~\eqref{eq:persistanceBurg} that one can restrict $T>0$ such that
\[ \int_0^{T} \big\vert R^c(t,\cdot)\big\vert_{X^s}\ dt \ \leq \ \ C_0\ M\ \r,\]
with $C_0=C(M,h_0^{-1},\frac1{2\sigma-1},\delta_{\min}^{-1},\delta_{\max},\gamma_{\min}^{-1})$.
Proposition~\ref{P.ConsIP} is proved.
\end{proof}

\begin{Remark}
As mentioned previously, we are unable to deduce from Proposition~\ref{P.ConsIP} a rigorous estimate on the difference between the exact solution and the constructed approximate solution as in Theorems~\ref{T.mr} or~\ref{T.mr2}. Indeed, the strategy developed in Section~\ref{S.completion} fails, as the solution does not satisfy the assumption of Lemma~\ref{L.HsW}, and more precisely the estimate on the time derivative, $\partial_t V$. A closer look at the proof shows that the only problematic term to estimate is $\big|\big[\partial_t, T[\W]\big] \Lambda^s W\big|_{L^2}$; and even more precisely $\big|\big[\partial_t, T[\W]\big] \Pi \Lambda^s W\big|_{L^2}$, as the supplementary is estimated through~\eqref{bound:PfdS} in Lemma~\ref{C.S}. We expect that the following strategy would imply the desired result: decompose
\[ \big|\big[\partial_t, T[\W]\big] \Pi \Lambda^s W\big|_{L^2} \lesssim \big\Vert (\Pi \Lambda^s W)\otimes\Pi \partial_t \W \big\Vert_{L^2}+\big\Vert (\Pi \Lambda^s W)\otimes(\Id-\Pi) \partial_t \W \big\Vert_{L^2}.\]
The first term is uniformly bounded as $ \Pi \partial_t \W$ roughly corresponds to the slow mode of the flow; the second term can be estimated using the different spatial localization of $\Pi W$ and $(\Id-\Pi)\W$. 

Following this strategy would require a few technical results and lengthy calculations, thus we do not pursue. Let us simply remark that the numerical simulations presented in the following section show perfect agreement with the desired result, namely 
\[ \big\Vert V- \Vrl - \Vf \big\Vert_{L^\infty([0,T];X^s)} \ = \ \O(\r).\] 
\end{Remark}

\subsection{Discussion and numerical simulations}\label{S.numerics}
In this section, we illustrate and discuss the results displayed in Theorem~\ref{T.mr} and Proposition~\ref{P.ConsApp} (validity of the rigid-lid approximation), Proposition~\ref{P.ConsCorrector} and Theorem~\ref{T.mr2} (improved approximate solution) and Proposition~\ref{P.ConsIP} (case of ill-prepared initial data).

In each case, we construct the appropriate approximate solution ($\Vrl,\Vf,\Vs$) and compare with the exact solution of the free-surface system~\eqref{FS2} (which is equivalent to~\eqref{FS} with the corresponding variables); for different values of $\r$ (and $\alpha=\r$), while the other parameters are fixed.

More precisely, we set:
\[ \delta=1/2 \quad ;\quad  \epsilon=1/2 \quad;\quad \gamma\in\{ 0.75,0.9,0.93,0.95,0.965,0.0975,0.09825,0.09875,0.099\}.\]
The initial data is set as follows:
\[ \zeta_2\id{t=0}=\exp\big(-(x/2)^2\big) \quad ; \quad u_s\id{t=0}=\frac{-1}3\exp\big(-(x/2)^2\big),\] 
and
\[ \zeta_1\id{t=0}=0\quad ; \quad u_s\id{t=0}=\begin{cases}0 \quad &\text{in the well-prepared case;}\\
2\exp\big(-(x/2)^2\big)\quad &\text{in the ill-prepared case.}
\end{cases}\]
We compute for times $t\in [0,T]$ with $T=4$. 
\medskip

Each figure contains three panels. The upper-left panel represents the initial data. For the sake of readability, we plot respectively $1+\delta^{-1}+\epsilon \zeta_1\id{t=0}$, $\delta^{-1}+\epsilon\zeta_2\id{t=0}$, $1+u_s\id{t=0}$ and $m\id{t=0}$. The lower panel represents the solution of the free-surface system~\eqref{FS2} as well as the corresponding approximate solution at stake (the latter with dotted lines), at final time $T=4$, for $\gamma=0.9$, thus $\r\approx 0.2673$. Finally, in the upper-right panel, we plot the normalized discrete $l^2$-norm of the difference between the aforementioned data in a log–log scale, for several values of $\r$ (the markers reveal the positions which have been computed), at final time $T=4$.
\medskip

The numerical scheme we use is based on spectral methods as for the space discretization (see~\cite{Trefethen}), thus yields an exponential accuracy with respect to the size of the grid $\Delta x$, as long as the signal is smooth (note that the major drawback is that the discrete differentiation matrices are not sparse). We set $\Delta x=0.1$ (for $x\in [-100,100]$), which is sufficient for the numerical errors to be undetectable. We then use the \textsf{Matlab} solver \verb+ode45+, which is based on the fourth and fifth order Runge-Kutta-Merson method~\cite{ShampineReichelt97}, with a tolerance of $10^{-8}$, in order to solve the time-dependent problem.
\medskip

\paragraph{Well-prepared initial data.} In Figure~\ref{F.0}, we present a numerical simulation corresponding to the setting of Theorem~\ref{T.mr}, thus we compare the solution of the free-surface system with the corresponding solution of the rigid-lid system (or more precisely, the rigid-lid approximate solution defined in Definition~\ref{D.VRL}). One straightforwardly sees that the free-surface solution closely follows the deformation of the interface and shear velocity predicted by the rigid-lid approximation, even for a relatively large value of $\r$ (we recall $\gamma=0.9$ in the panel~\ref{F.final0}). As a matter of fact, the precision of the approximation is not foreseen from Theorem~\ref{T.mr}: as we can see from the panel~\ref{F.rate0}, the convergence rate for $\zeta_2$ and $u_s$ is $\O(\r^2)$, as Theorem~\ref{T.mr} predicts only $\O(\r)$. One can see that the main error in the rigid-lid approximation is supported on the deformation of the surface, $\zeta_1$, as well as on the horizontal momentum, $m$ (and more precisely the fast mode of the horizontal momentum).

 \begin{figure}[!ht] \vspace{-.5cm}
 \subfigure[Initial data]{
\includegraphics[width=0.45\textwidth]{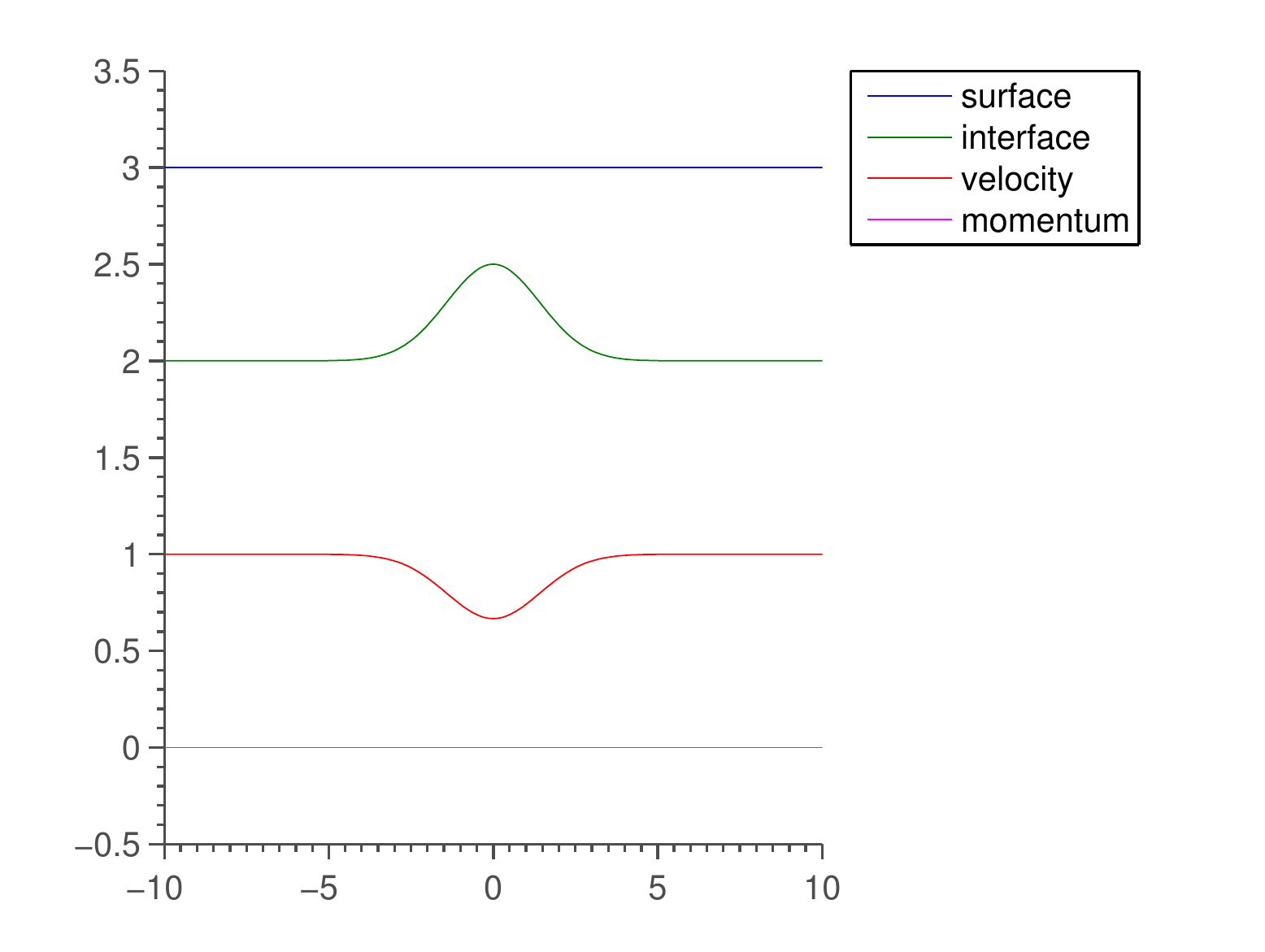}
\label{F.init0}
}
 \subfigure[Error with respect to $\r$ (log-log scale)]{
\includegraphics[width=0.45\textwidth]{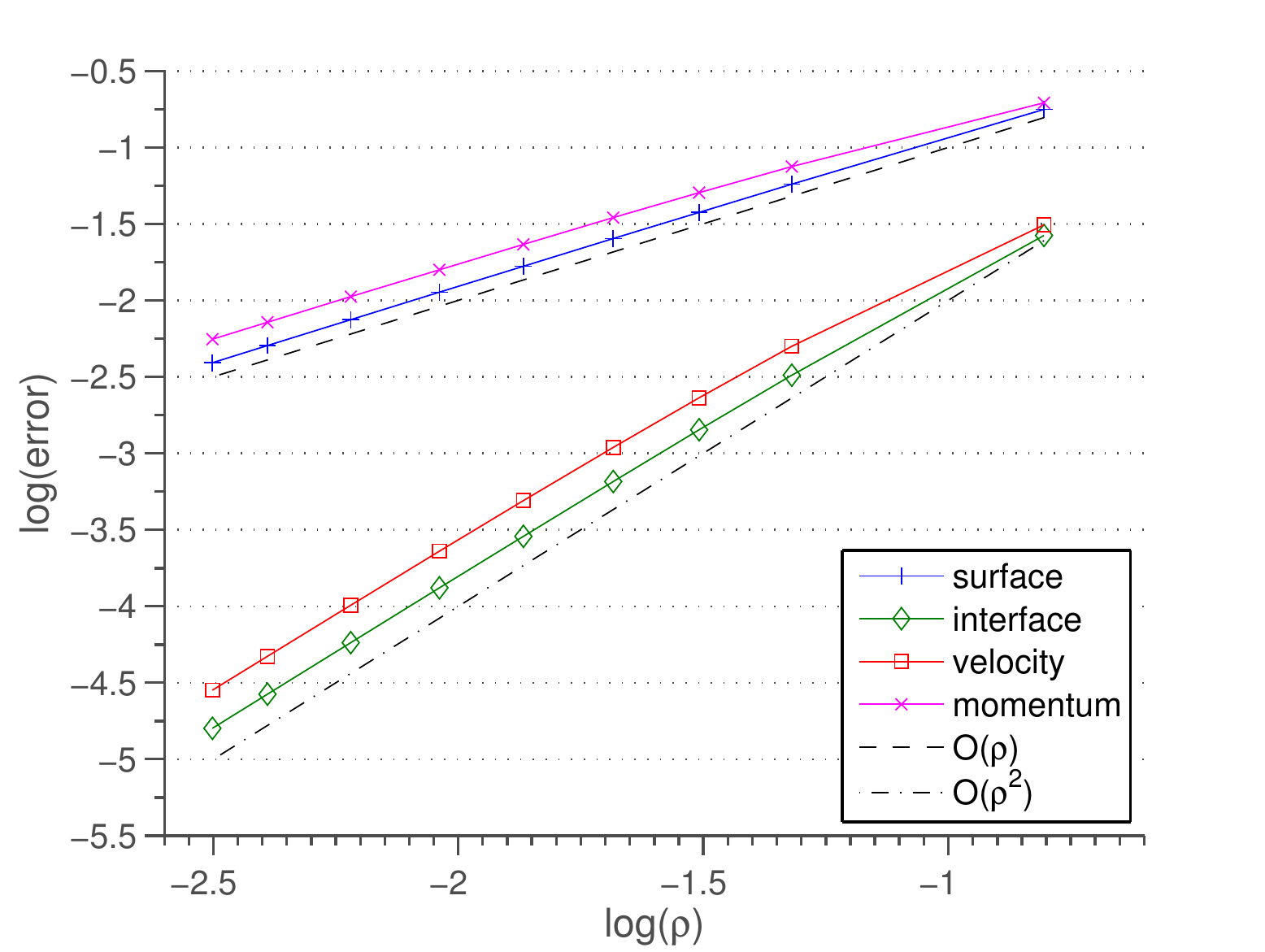}
\label{F.rate0}
}\vspace{-.4cm}
 \subfigure[Flow at finite time $T=4$]{
\includegraphics[width=.9\textwidth]{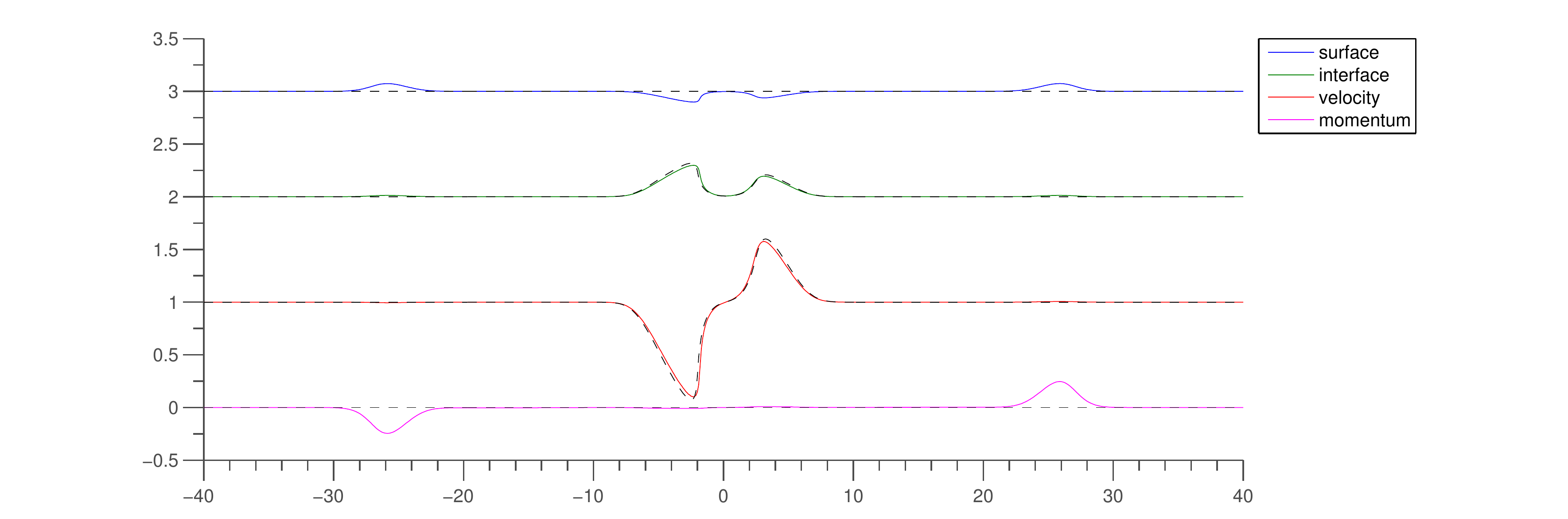}
\label{F.final0}
}
\caption{Solution of the free-surface system compared with the rigid-lid approximate solution}
\label{F.0}
\end{figure} 
\begin{figure}[!ht] \vspace{-.5cm}
 \subfigure[Initial data]{
\includegraphics[width=0.45\textwidth]{./initWP.pdf}
\label{F.init1}
}
 \subfigure[Error with respect to $\r$ (log-log scale)]{
\includegraphics[width=0.45\textwidth]{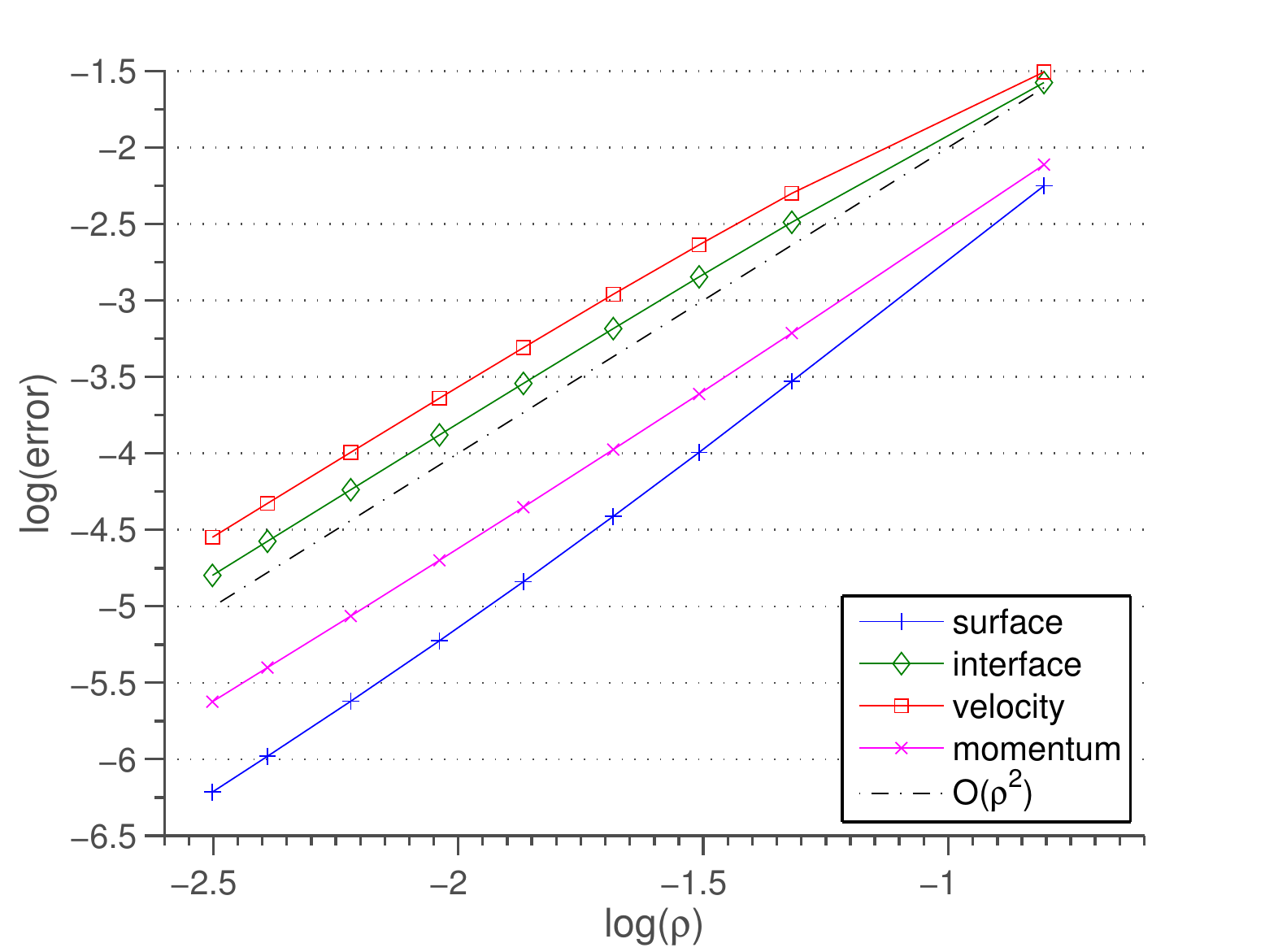}
\label{F.rate1}
}\vspace{-.4cm}
 \subfigure[Flow at finite time $T=4$]{
\includegraphics[width=.9\textwidth]{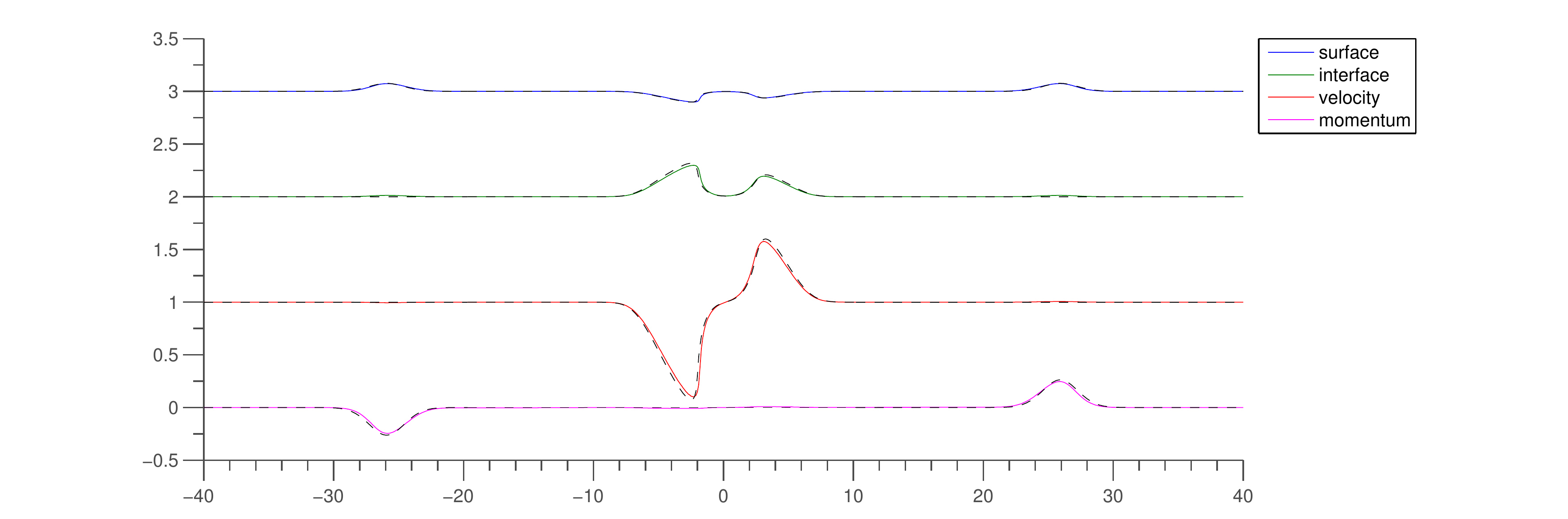}
\label{F.final1}
}
\caption{Solution of the free-surface system compared with the improved approximate solution}
\label{F.1}
\end{figure}

Of course, such result is predicted by Theorem~\ref{T.mr2}, since the first-order corrector constructed in Proposition~\ref{P.ConsCorrector} follows precisely the above description. We show in Figure~\ref{F.1} the precision of the improved rigid-lid approximation. One sees that the main differences between the free-surface solution and the rigid-lid approximate solution have been recovered. The rate of convergence is now $\O(\r^2)$ for each variable $\zeta_1,\zeta_2,u_s,m$, in full accordance with Theorem~\ref{T.mr2}.

\paragraph{Ill-prepared initial data.} We discuss now the case of ill-prepared initial data, that is when $\zeta_1\id{t=0},m\id{t=0}$ are not assumed to be small. We chose to set a non-trivial initial value only to the horizontal momentum variable $m$, so that the hypothesis $\alpha=\r$ cannot artificially modify the convergence rate (recall the surface deviation from the flat equilibrium value is represented by $\epsilon\alpha\zeta_1$).
\medskip

We plot in Figure~\ref{F.2} the difference between the exact solution of the free-surface system and the approximate solution constructed in Proposition~\ref{P.ConsIP}. As one can see, there is a noticeable difference between the two solution. Moreover, this discrepancy seems to be mainly located on the fast mode, and on the variables $\zeta_1,m$. As a matter of fact, the variables $\zeta_2,u_s$ present a slightly better convergence rate in panel~\ref{F.rate2} (around $\O(\r^{1.2})$ and $\O(\r^{1.5})$, respectively) than predicted by Proposition~\ref{P.ConsIP}, namely $\O(\r)$.
\medskip
 
Such a result advocates for the construction of a higher-order approximation, similarly to the case of well-prepared initial data. Indeed, we know from Proposition~\ref{P.ConsCorrector} that one can construct a first-order slow mode corrector term $(\r \breve\zeta_1,0,0,0)^\top$ and that its initial value plays a role in the construction of the fast mode corrector. More precisely, one has to modify the initial data of the fast mode corrector in order to ensure that the full approximate solution enjoys the appropriate initial data. Using both statements of Proposition~\ref{P.ConsCorrector} and Proposition~\ref{P.ConsIP}, we define the improved approximation for ill-prepared initial data as 
\[ \Vapp \ = \ \Vrl \ + \ \Vs \ + \ \Vf,\]
where 
\begin{itemize}
\item $\Vrl\equiv (0,\eta,v,0)^\top$ is defined by Definition~\ref{D.VRL};
 \item $\Vs\equiv (\r\breve{\zeta_1},0,0,0)^\top$ is defined by $ \breve \zeta_1\ \equiv \ -\big(\eta+\frac{\delta}2\eta^2\big) \frac{( 1-\eta)(\delta^{-1}+\eta)v^2}{(1+\delta^{-1})^2}$.
 \item $\Vf$ is defined with \[\Vf(t,x)\equiv \begin{pmatrix}u_+(t,x) +u_-(t,x)\\ 0\\ 0\\ c(u_+(t,x)-u_-(t,x)) \end{pmatrix},\]
where $c\equiv \sqrt{1+\delta^{-1}}$, and $u_\pm$ is the unique solution to $\partial_t u_\pm \pm \frac{c}\r \partial_x u_\pm \pm \frac{3}{2c} u_\pm\partial_x u_\pm = 0$,
 with ${u_\pm}\id{t=0} = \frac12\big(\zeta_1^0-\r \breve{\zeta_1}\id{t=0}\pm c^{-1} m^0\big)$.
\end{itemize}
Let us notice that, as previously mentioned in Remark~\ref{remVf}, this improved approximation is equivalent to the one already defined in Proposition~\ref{P.ConsCorrector} for well-prepared initial data. Thus this approximate solution is quite general and robust: it offers the same precision as our previously constructed approximate solutions in the well-prepared case (Proposition~\ref{P.ConsCorrector}) as well as in the ill-prepared case (Proposition~\ref{P.ConsIP}). 
\medskip

We investigate in Figure~\ref{F.3} the accuracy of this improved approximate solution. Comparing panels~\ref{F.final2} and~\ref{F.final3}, one clearly sees that the new approximate solution shows a better resemblance than the original approximate solution; the main discrepancy seems to be recovered. However, as one can see from panel~\ref{F.rate3}, this apparent improvement is not reflected in the convergence rate. Although the produced error is clearly smaller, the rate is not better than $\O(\r)$ where $\zeta_1$ and $m$ are involved ($\zeta_2$ and $u_s$ are unchanged). It is not clear to us whether a better approximate solution can be constructed, nor what explains the slightly better convergence rate on $\zeta_2$ and $u_s$. Our numerical simulations indicate that there is a non-trivial coupling between the fast and slow modes during early times (when both are localized at the same place), and that the contribution of these coupling effects is of size $\approx \r$. Thus in order to take into account these coupling effects, one may have no other choice than solving a fully coupled system, at least for small time, $t=\O(\r)$.

 \begin{figure}[!ht] \vspace{-1cm}
 \subfigure[Initial data]{
 \includegraphics[width=0.45\textwidth]{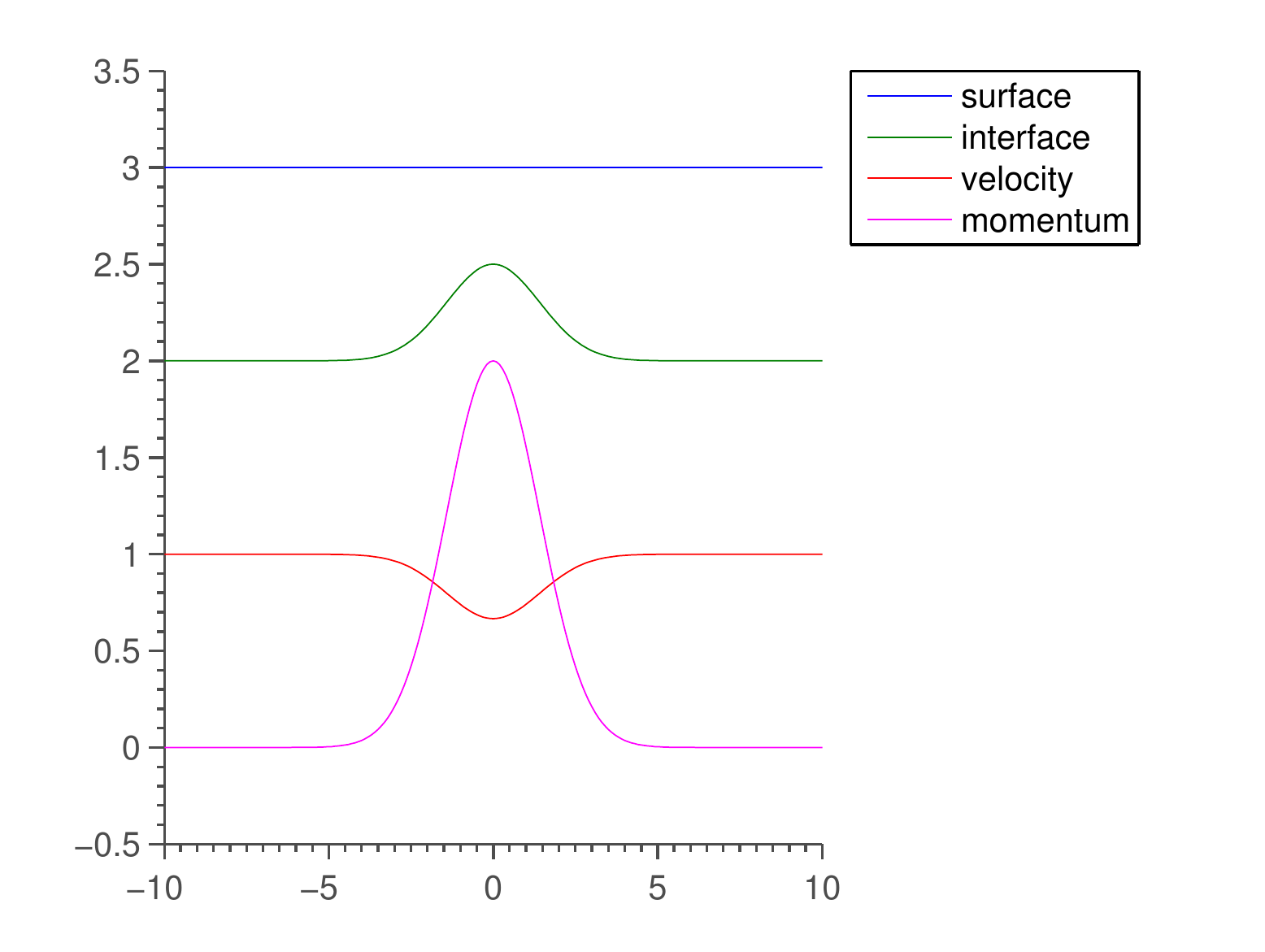}
 \label{F.init2}
 }
 \subfigure[Error with respect to $\r$ (log-log scale)]{
 \includegraphics[width=0.45\textwidth]{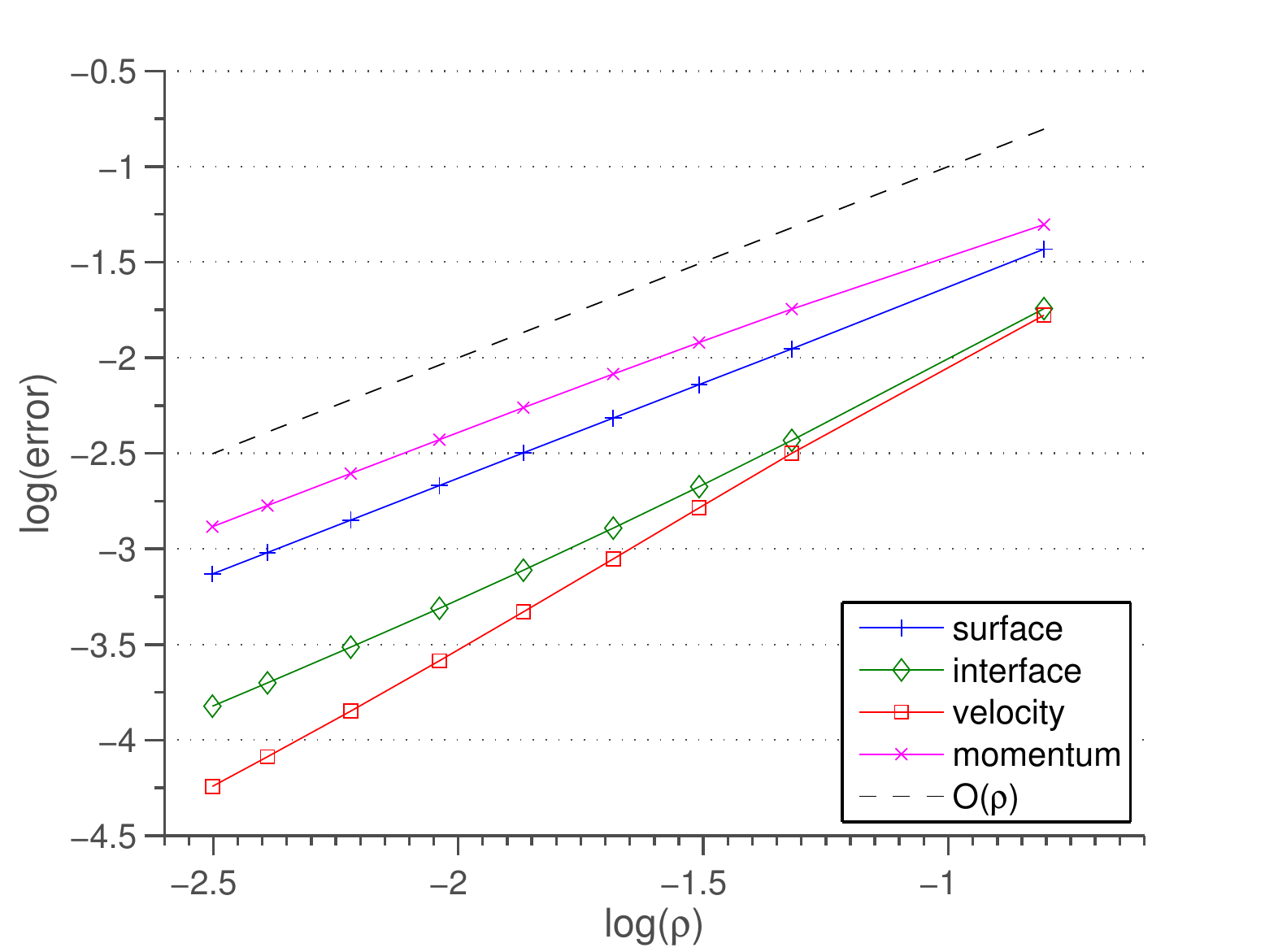}
 \label{F.rate2}
 }\vspace{-.4cm}
 \subfigure[Flow at finite time $T=4$]{
 \includegraphics[width=.9\textwidth]{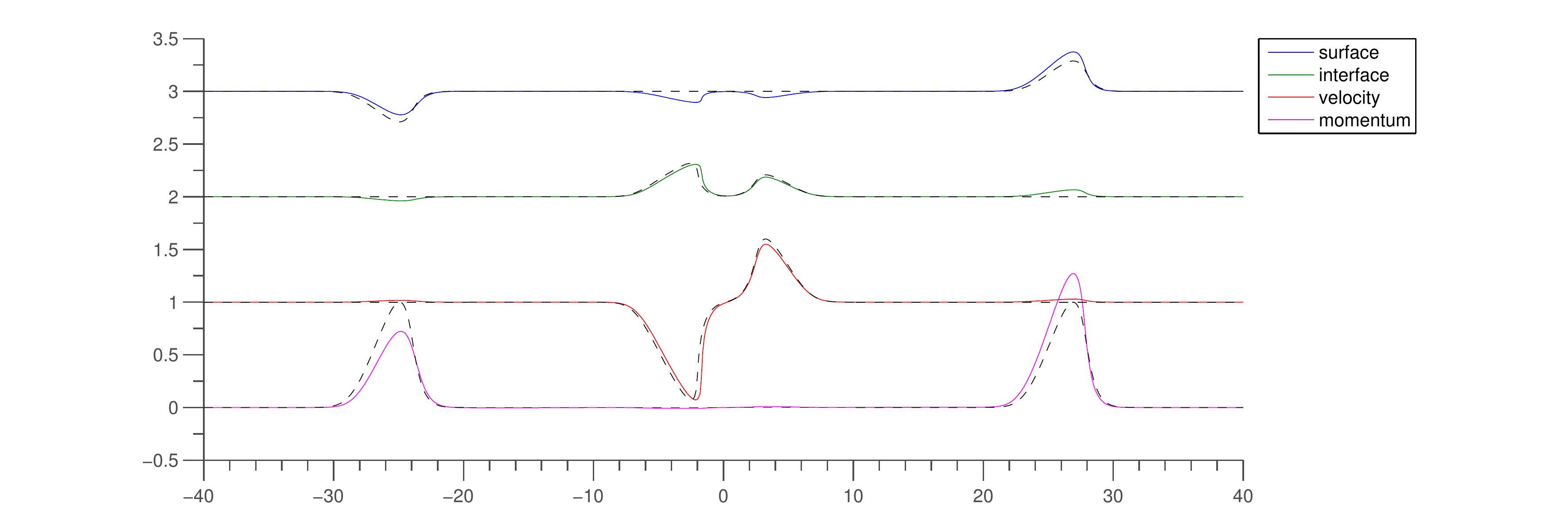}
 \label{F.final2}
 }\vspace{-.2cm}
 \caption{Solution of the free-surface system compared with the approximate solution, for ill-prepared initial data}
 \label{F.2}
 \end{figure}
 \begin{figure}[!ht] \vspace{-.5cm}
 \subfigure[Initial data]{
\includegraphics[width=0.45\textwidth]{./initWP.pdf}
\label{F.init3}
}
 \subfigure[Error with respect to $\r$ (log-log scale)]{
\includegraphics[width=0.45\textwidth]{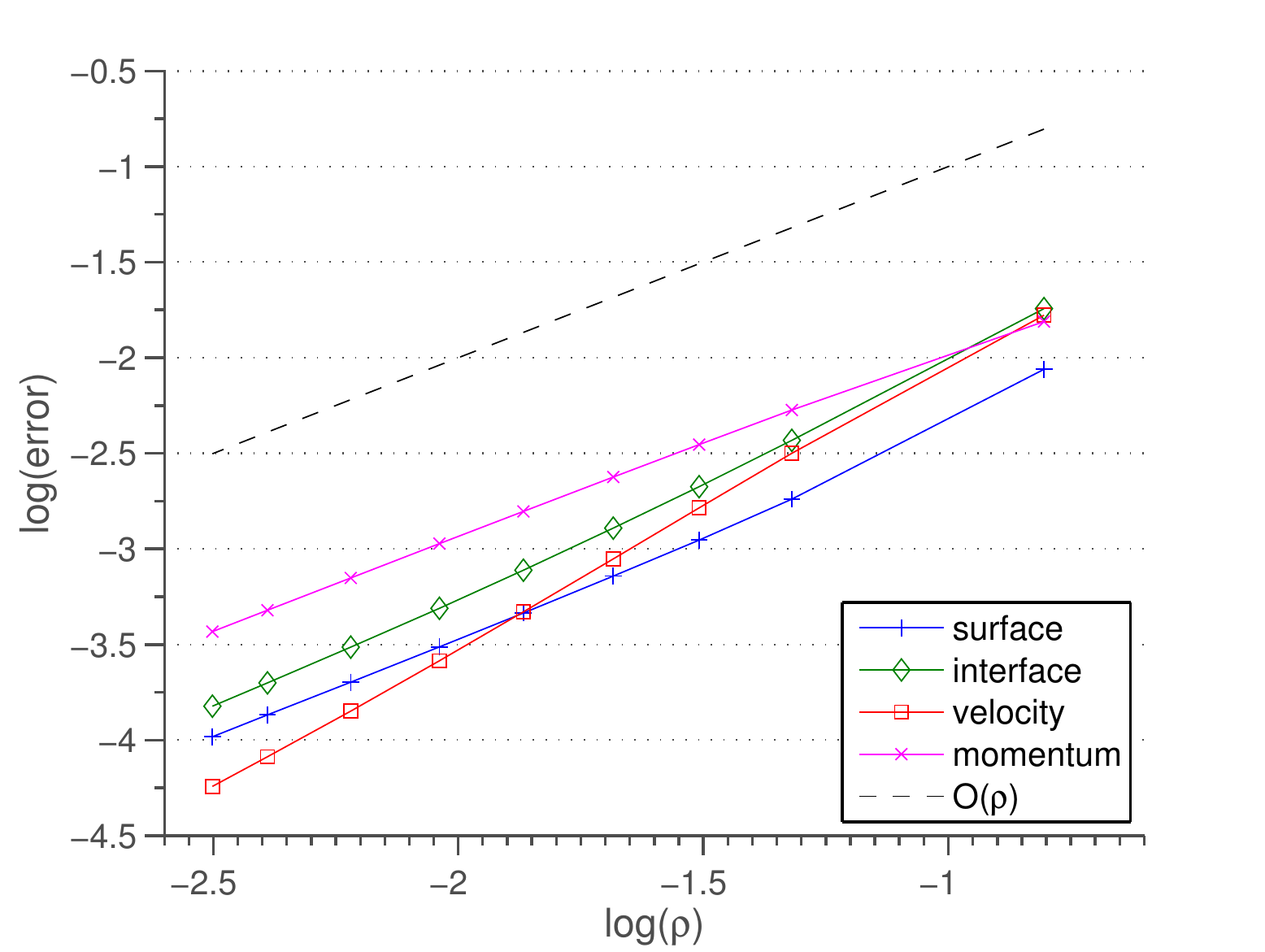}
\label{F.rate3}
}\vspace{-.4cm}
 \subfigure[Flow at finite time $T=4$]{
\includegraphics[width=.9\textwidth]{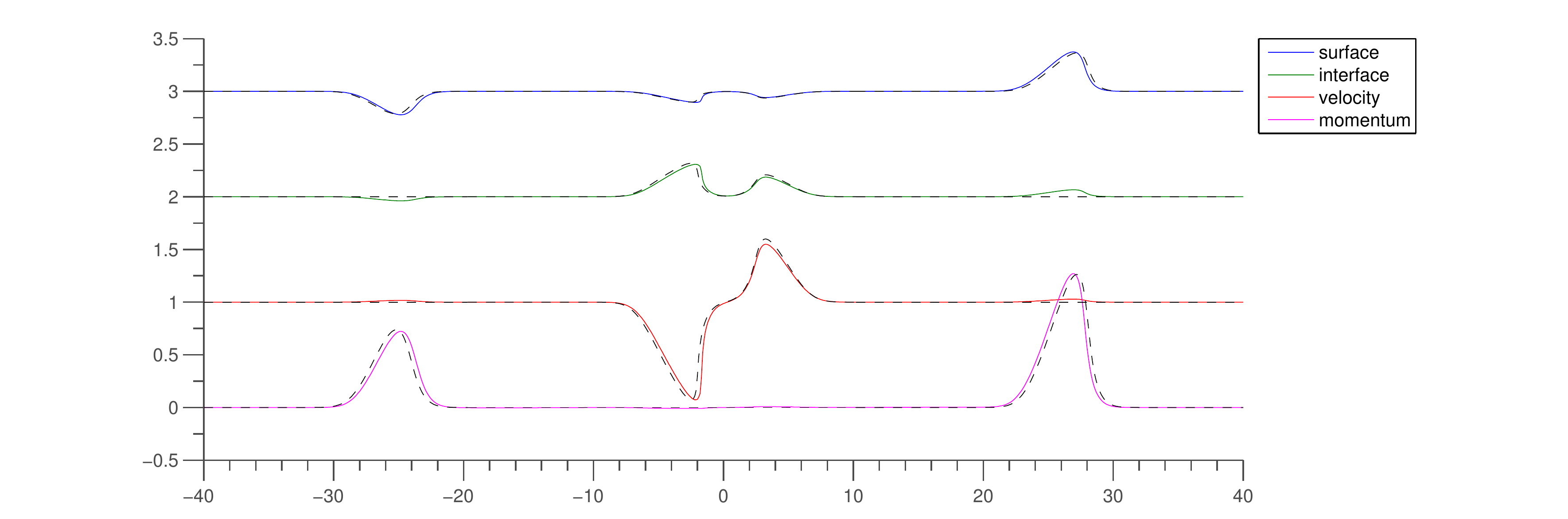}
\label{F.final3}
 }\vspace{-.2cm}
\caption{Solution of the free-surface system compared with the improved approximate solution, for ill-prepared initial data}
\label{F.3}
\end{figure}

\appendix
\section{Proof of Proposition~\ref{P.WPFS}}\label{S.app}
In this section, we detail the proof of Proposition~\ref{P.WPFS}, which follows the classical theory concerning Friedrichs-symmetrizable quasilinear systems. The proof is based on {\em a priori} energy estimates, for which the key ingredients are product and commutator estimates in Sobolev spaces. We first recall such results, and let the reader refer to, {\em e.g.},~\cite{AlinhacGerard,Lannes} for the proof of Lemmata~\ref{L.Moser} and~\ref{L.KatoPonce}.
\begin{Lemma}[Product estimates]\label{L.Moser}~\\
Let $s\geq 0$. For any $f,g\in H^s(\RR)\bigcap L^\infty(\RR)$, one has:
\[ \big\vert \ f \ g\ \big\vert_{H^s} \ \lesssim \ \big\vert \ f \ \big\vert_{L^\infty} \big\vert \ g\ \big\vert_{H^s}+\big\vert \ f \ \big\vert_{H^s} \big\vert \ g\ \big\vert_{L^\infty}.
\]
If $s\geq s_0>1/2$, one deduces thanks to continuous embedding of Sobolev spaces,
\[ \big\vert \ f \ g\ \big\vert_{H^s} \ \lesssim\ \big\vert \ f \ \big\vert_{H^s} \big\vert \ g\ \big\vert_{H^s} .\]
Let $F\in C^\infty(\RR)$ such that $F(0)=0$. If $g\in H^s(\RR)\bigcap L^\infty(\RR)$ with $s\geq 0$, one has $F(g)\in H^s(\RR)$ and
\[ \big\vert\ F(g) \ \big\vert_{H^s} \ \leq \ C(\big\vert\ g\ \big\vert_{L^\infty},\big\vert\ F\ \big\vert_{C^\infty})\big\vert \ g \ \big\vert_{H^s}.\]
\end{Lemma}
Throughout the paper, we repeatedly make use of the following Corollary.
\begin{Corollary}\label{C.depth}
Let $f,\zeta\in L^\infty\bigcap H^{ s}$, with $s\geq 0$ and $h(\zeta)\equiv 1-\zeta$, with $h(\zeta)\geq h_0>0$ for any $x\in\RR$. Then one has
\begin{align*} \big\vert \frac1{h(\zeta)} f \big\vert_{H^{s}} \ &\leq \ C(h_0^{-1},\big\vert\zeta\big\vert_{L^{\infty}})\big( \big\vert f\big\vert_{H^{s}}+ \big\vert \zeta\big\vert_{H^{s}}\big\vert f\big\vert_{L^\infty}\big)\\
 \big\vert f-\frac1{h(\zeta)} f \big\vert_{H^{s}} \ &\leq \ \ C(h_0^{-1},\big\vert\zeta\big\vert_{L^{\infty}})\big( \big\vert\zeta\big\vert_{L^{\infty}}\big\vert f\big\vert_{H^{s}}+ \big\vert \zeta\big\vert_{H^{s}}\big\vert f\big\vert_{L^\infty}\big).\end{align*}
\end{Corollary}
\begin{proof}
We will use the identity
\[ \frac1{h(\zeta)} f \ = \ \frac1{1-\zeta}f \ = \ f \ + \ \frac{\zeta}{1-\zeta}f.\]
By Lemma~\ref{L.Moser}, one deduces
 \begin{align*}\big\vert \frac1{h(\zeta)} f \big\vert_{H^{s}} \ &\leq \ \big\vert f \big\vert_{H^{s}} \ + \ \big\vert \frac{\zeta}{1-\zeta} f \big\vert_{H^{s}}\\
 &\lesssim \ \big\vert f \big\vert_{H^{s}} \ + \ \big\vert \frac{\zeta}{1-\zeta} \big\vert_{L^\infty}\big\vert f \big\vert_{H^{s}} \ + \ \big\vert \frac{\zeta}{1-\zeta}\big\vert_{H^{s}}\big\vert f \big\vert_{L^\infty}.
 \end{align*}
 The only non-trivial term to estimate is now $\big\vert \frac{\zeta}{1-\zeta}\big\vert_{H^{s}}$. Using that $h(\zeta)=1-\zeta \geq h_0>0$, we introduce a function $F\in C^\infty(\RR)$ such that 
 \[F(X) \ = \ \begin{cases}
 \frac{X}{1-X} & \text{ if } 1-X\geq h>0,\\ 0 & \text{ if } 1-X\leq 0.
 \end{cases}
\]
The function $F$ satisfies the hypotheses of Lemma~\ref{L.Moser}, and one has
\[ \big\vert \frac{\zeta}{1-\zeta}\big\vert_{H^{s}} \ = \ \big\vert F(\zeta)\big\vert_{H^{s}} \ \leq \ C(\big\vert \zeta\big\vert_{L^\infty},h_0^{-1})\big\vert \zeta \big\vert_{H^{s}}.\]
The first estimate of the Lemma is proved. The second estimate is obtained in the same way, using
\[ f-\frac1{h(\zeta)} f \ =\ - \frac{\zeta}{1-\zeta}f.\]
The Corollary is proved.\end{proof}

The following Lemma presents a generalization of the Kato-Ponce~\cite{KatoPonce88} commutator estimates due to Lannes~\cite{Lannes06} (one has $\big\vert f\big\vert_{H^{s}}$ instead of $\big\vert\partial_x f\big\vert_{H^{s-1}}$ in the standard Kato-Ponce estimate).
\begin{Lemma}[Commutator estimates]\label{L.KatoPonce}~\\
For any $s\geq0$, and $\partial_x f, g\in L^\infty(\RR)\bigcap H^{s-1}(\RR),$ one has
\[ \big\vert\ [\Lambda^s,f]g\ \big\vert_{L^2}\ \lesssim\ \big\vert \ \partial_x f\ \big\vert_{H^{s-1}}\big\vert \ g\ \big\vert_{L^\infty}+\big\vert \ \partial_x f\ \big\vert_{L^\infty}\big\vert \ g\ \big\vert_{H^{s-1}} .
\]
Thanks to continuous embedding of Sobolev spaces, one has for $s\geq s_0+1, \ s_0>\frac{1}{2},$ 
\[ \big\vert\ [\Lambda^s,f]g\ \big\vert_{L^2}\ \lesssim\ \big\vert \ \partial_x f\ \big\vert_{H^{s-1}}\big\vert\ g\ \big\vert_{H^{s-1}} .
\]
\end{Lemma}
\bigskip

Let us now continue with the proof of Proposition~\ref{P.WPFS}.
The system~\eqref{FS} is quasilinear. We prove below that it is Friedrichs-symmetrizable, under conditions~\eqref{condH}. We display below the symmetrizer of the system, and compute the necessary energy estimates in Lemmata~\ref{L.L2} and~\ref{L.Hs}. 
\bigskip

\noindent {\em Symmetrizer of the system.} Recall that~\eqref{FS} reads $\partial_t U \ + \ A[U]\partial_x U \ = \ 0$, with
 \begin{equation}\label{defA0A1} A[U] \ \equiv \ \begin{pmatrix}
 u_1 &\frac{ u_2-u_1}{\r} & \frac{h_1}{\r} & \frac{h_2}{\r} \\
0 & u_2 & 0 & h_2 \\
\frac{1}{\r} & 0 & u_1 & 0 \\
 \frac\gamma{\r} & \delta+\gamma & 0 & u_2
\end{pmatrix},\end{equation}
where we denote $h_1\equiv 1+\r \zeta_1-\zeta_2$ and $h_2\equiv \delta^{-1}+\zeta_2$.
Define
\begin{equation}\label{defS} S[U]\equiv\begin{pmatrix}
\gamma &0& 0 & 0 \\
0& \gamma+\delta&  0 & u_2-u_1\\
0 & 0 &\gamma h_1 & 0 \\
0& u_2-u_1 & 0 &h_2
\end{pmatrix} .\end{equation}
One can easily check that $S[U]A[U]\equiv\Sigma[U]$ and $S[U]$ are symmetric. More precisely, one has
\begin{equation}\label{defSigma} \Sigma[U] \ \equiv \ \begin{pmatrix}
 \gamma u_1 &\frac{\gamma(u_2-u_1)}{\r}& \frac{\gamma h_1}{ \r}& \frac{\gamma h_2}{ \r} \\
\frac{\gamma (u_2-u_1)}{\r} &2 (\gamma+\delta) (2u_2-u_1)& 0 &(\gamma+\delta)h_2+ u_2(u_2-u_1)\\
 \frac{\gamma h_1}{ \r}& 0 & \gamma h_1u_1 & 0 \\
 \frac{\gamma h_2}{\r} &(\gamma+\delta)h_2+ u_2(u_2-u_1)& 0 &h_2 (2u_2-u_1)
\end{pmatrix}.\end{equation}
One easily checks that $S[U]$ is positive definite provided that the following holds: 
\[ \gamma>0 \quad ; \quad \gamma+\delta>0\quad ; \quad h_1 > 0 \quad ; \quad h_2-\frac{|u_2-u_1|^2}{\gamma+\delta}> 0,\]
which is guaranteed by condition~\eqref{condH}.
\medskip

\noindent {\em Energy of the system.} The natural energy of our system is 
\begin{align}\label{eqn:def-energy}
E^s(U)  &\equiv  \big( S[\underline{U}] \Lambda^s U,\Lambda^s U\big) \\
 &=  \gamma \big\vert \zeta_1\big\vert_{H^s}^2 + (\gamma+\delta )\big\vert \zeta_2\big\vert_{H^s}^2+\gamma \int_\RR \u h_1 \big\vert \Lambda^s u_1\big\vert^2+\int_\RR \u h_2 \big\vert\Lambda^s u_2\big\vert^2 +2\int_\RR (\u u_2- \u u_1) \big\{\Lambda^su_2\big\}\big\{\Lambda^s\zeta_2\big\},\nn \end{align}
with $\u h_1\equiv 1+ \r \underline{\zeta}_1-\underline{\zeta}_2$ and $\u h_2\equiv \delta^{-1}+\u \zeta_2 $. 

We precise below the equivalence between our energy and the norm $X^s$ offered by the well-posedness of the symmetrizer. Recall that $X^s$ denotes the space $H^s(\RR)^4$, endowed with the following norm:
\[ \big\vert U\big\vert_{X^s}^2 \ = \ \gamma\big\vert \zeta_1\big\vert_{H^s}^2 \ + \ \big\vert \zeta_2\big\vert_{H^s}^2+\gamma \big\vert u_1\big\vert_{H^s}^2+\big\vert u_2\big\vert_{H^s}^2 .\]

\begin{Lemma}\label{L.energyXs}
 Let $s\geq 0$ and $ \underline{\zeta}\in L^{\infty}(\RR)$, satisfying~\eqref{condH}. Then
$E^s(U)$ is uniformly equivalent to the $\vert \cdot\vert_{X^s}$-norm. More precisely, there exists positive constants $C_2=C(h_{0}^{-1},\delta_{\min}^{-1})>0$ and $C_1=C(\big\vert \u h_1\big\vert_{L^\infty},\big\vert \u h_2\big\vert_{L^\infty},\delta_{\max})>0$ such that
\[
\frac1{C_1}E^s(U) \ \leq \ \big\vert U \big\vert_{X^s}^2 \ \leq \ C_2 E^s(U).
\]
\end{Lemma}
\begin{proof}
The fact that $E^s(U) \ \leq \ C_1 \big\vert U \big\vert_{X^s}$ is a simple consequence of Cauchy-Schwarz inequality, applied to~\eqref{eqn:def-energy}, where we use that~\eqref{condH} yields $\vert \u u_2-\u u_1 \vert^2 < (\gamma+\delta) \u h_2$.
\medskip

The other inequality follows directly from~\eqref{condH}. More precisely, one has
\begin{align*}E^s(U) 
&\geq  \gamma \big\vert \zeta_1\big\vert_{H^s}^2+\gamma h_0 \int_\RR \big\vert \Lambda^s u_1\big\vert^2 \\
&\qquad + (\gamma+\delta )\big\vert \zeta_2\big\vert_{H^s}^2+\int_\RR \u h_2 \big\vert\Lambda^s u_2\big\vert^2 -2\int_\RR\sqrt{(\u h_2-h_0)(\gamma+\delta)} \big\{\Lambda^su_2\big\}\big\{\Lambda^s\zeta_2\big\},
 \end{align*}
and the result is now clear.
Lemma~\ref{L.energyXs} is proved.
\end{proof}

We now highlight energy estimates concerning the linearized system from~\eqref{FS}, namely
\begin{equation}\label{eqn:FSlin}
 \partial_t U \ + \ A[\underline{U}]\partial_x U \ = \ \R \ , 
 \end{equation}
with given $\u U,\R$.

\begin{Lemma}[$L^2$ energy estimate]\label{L.L2}
Set $T,M>0$. Let $U\in L^\infty ( [0,T];X^0)$ satisfy~\eqref{eqn:FSlin}
with given $\R\in L^1([0,T];X^0)$, and $\underline{U}$ satisfying~\eqref{condH} with $h_0>0$ (for any $t\in[0,T]$) as well as
\[ \big\Vert \underline U \big\Vert_{L^\infty([0,T]\times \RR)^4}+\big\Vert \partial_x\underline U \big\Vert_{L^\infty([0,T]\times \RR)^4}+\r \big\Vert \partial_t \underline U \big\Vert_{L^\infty([0,T]\times \RR)^4} \ \leq \ M.\]
 Then there exists $C_0\equiv C(M,h_0^{-1},\delta_{\min}^{-1},\delta_{\max})$ such that
\begin{equation}\label{energyestimateL2}
	\forall t\in [0,T],\qquad
	E^0(U)(t)\leq e^{C_0M\r^{-1} t}E^0(U\id{t=0})+ C_0 \int^{t}_{0} e^{C_0M\r^{-1}( t-t')}\big\vert \R(t',\cdot)\big\vert_{X^s}\ dt'.
\end{equation}
\end{Lemma}
\begin{proof} 
Let us consider the $L^2$-inner product of~\eqref{eqn:FSlin} and $ S[\underline U] U$:
\[ \big(\partial_t U,S[\underline U] U\big) \ + \ \big(A[\underline U]\partial_x U,S[\underline U] U\big)
\ = \ \big( \R ,S[\underline U] U\big) \ .
\]
From the symmetry property of $S[\underline U],\Sigma[\underline U]$, and using the definition of $E^0(U)$, one deduces
\begin{align}
\frac12 \frac{d}{dt}E^0(U) \ &= \frac12\big( U,\big[\partial_t, S[\underline U]\big] U\big)-\big(\Sigma[\underline U]\partial_xU, U\big) +
\big(\R,S[\underline U] U\big) \nn \\
&= \frac12\big( U,\big[\partial_t, S[\underline U]\big] U\big)+\frac12\big(\big[\partial_x,\Sigma[\underline U]\big] U, U\big) +
\big(\R,S[\underline U] U\big) .
\label{eqn:energyequalityX0}
\end{align}
We now estimate each of the terms in the right-hand side of~\eqref{eqn:energyequalityX0}.
\medskip

\noindent {\em Estimate of $\big( U,\big[\partial_t, S[\underline U]\big] U\big)$.} One has 
$ \big( U,\big[\partial_t, S[\underline U]\big] U\big) \ = \ \big( U, {\rm d}S[\partial_t \u U] U\big)$,
with 
\[ {\rm d}S[\partial_t\u U] \equiv \begin{pmatrix}
0 &0& 0 & 0 \\
0& 0& 0 &  \partial_t(\u u_2-\u u_1)\\
0 & 0 &\gamma \partial_t (\r\u \zeta_1-\u \zeta_2) & 0 \\
0& \partial_t(\u u_2-\u u_1) & 0 &\partial_t \u \zeta_2
\end{pmatrix}. \]
Using Cauchy-Schwarz inequality, and Lemma~\ref{L.energyXs}, one has straightforwardly
\begin{equation}\label{eq1}
\big| \big( U,\big[\partial_t, S[\underline U]\big] U\big) \big| \ \leq \ C_0\big\vert \partial_t \u U\big\vert_{L^\infty} C_2^{-1}\ \big\vert U \big\vert_{X^0}^2 \ \leq \ C_0\ M\ \r^{-1}\ E^0(U) ,
\end{equation}
with $C_0=C(h_0^{-1},\delta_{\min}^{-1},\delta_{\max})$.
\medskip

\noindent {\em Estimate of $\big(\big[\partial_x,\Sigma[\underline U]\big] U, U\big) $.} One has $\big(\big[\partial_x,\Sigma[\underline U]\big] U, U\big) =\big( U, {\rm d}\Sigma[\u U] U\big)$ with
\[  {\rm d}\Sigma[\u U]\equiv \begin{pmatrix}
\gamma \partial_x \u u_1 &\frac{\gamma\partial_x(\u u_2-\u u_1)}{\r}& \frac{\gamma \partial_x(\r \u \zeta_1-\u \zeta_2)}{ \r}& \frac{\gamma \partial_x \u \zeta_2}{ \r} \\
\frac{\gamma\partial_x(\u u_2-\u u_1)}{\r} & (\gamma+\delta)\partial_x(2 \u u_2-\u u_1)& 0 &\partial_x \big((\gamma+\delta) \u \zeta_2+\u u_2(\u u_2-\u u_1)\big)\\
 \frac{\gamma \partial_x(\r \u \zeta_1-\u \zeta_2)}{ \r} & 0 &  \gamma \partial_x (\u h_1\u u_1) & 0 \\
\frac{\gamma\partial_x \u \zeta_2}{ \r} &\partial_x \big((\gamma+\delta) \u \zeta_2+\u u_2(\u u_2-\u u_1)\big) & 0 &2\partial_x\big(\u h_2(2 \u u_2-\u u_1)\big)
\end{pmatrix} .\]
As above, Cauchy-Schwarz inequality and Lemmata~\ref{L.Moser} and~\ref{L.energyXs} yield
\begin{equation}\label{eq2}
\big|\big(\Sigma[\underline U]\partial_xU, U\big)\big| \ \leq \ C_0\ M\ \r^{-1}\ E^0(U) ,
\end{equation}
with $C_0=C(M,h_0^{-1},\delta_{\min}^{-1},\delta_{\max})$.
\medskip

\noindent {\em Estimate of $\big(\R,S[\underline U] U\big)$.} By Cauchy-Schwarz inequality and Lemmata~\ref{L.Moser} and~\ref{L.energyXs},
\begin{equation}\label{eq3}
\big| \big(\R,S[\underline U] U\big) \big| \ \leq \ C_0\ \big\vert U \big\vert_{X^s} \big\vert \R \big\vert_{X^s} \ \leq \ C_0' E^s(U)^{1/2} \big\vert \R \big\vert_{X^s} \ ,
\end{equation}
with $C_0,C_0'=C(M,h_0^{-1},\delta_{\min}^{-1},\delta_{\max})$.

Estimate~\eqref{energyestimateL2} is now a consequence of Gronwall-Bihari's inequality applied to the differential inequality obtained when plugging~\eqref{eq1},~\eqref{eq2},~\eqref{eq3} into~\eqref{eqn:energyequalityX0}.
\end{proof}
\begin{Lemma}[$H^s$ energy estimate]\label{L.Hs}
Set $M,T>0$ and $s\geq s_0+1$, $s_0>1/2$. Let $U\in L^\infty ([0,T];X^s)$ 
satisfy~\eqref{eqn:FSlin}
with $\R\in L^1([0,T];X^s)$, and $\underline{U}\in L^\infty ([0,T];X^s)$ satisfying~\eqref{condH} as well as
\[ \big\Vert \underline U \big\Vert_{L^\infty([0,T];X^s)}+\r \big\Vert \partial_t \underline U \big\Vert_{L^\infty([0,T];X^{s-1} )} \ \leq \ M.\]
 Then there exists $C_0\equiv C(M,h_0^{-1},\delta_{\min}^{-1},\delta_{\max})$ such that
\begin{equation}\label{energyestimateHs}
	\forall t\in [0,T],\qquad
	E^s(U)(t)\leq e^{C_0M\r^{-1} t}E^s(U\id{t=0})+ C_0 \int^{t}_{0} e^{C_0M\r^{-1}( t-t')}\big\vert \R(t',\cdot)\big\vert_{X^s}\ dt'.
\end{equation}
\end{Lemma}
\begin{proof}
As previously, we deduce from~\eqref{eqn:FSlin} the identity
\[ \big(\Lambda^s\partial_t U,S[\underline U] \Lambda^s U\big) \ + \ \big(\Lambda^s A[\underline U]\partial_x U,S[\underline U] \Lambda^s U\big)
\ = \ \big( \Lambda^s \R ,S[\underline U] \Lambda^s U\big) \ ,
\]
where we recall the notation $\Lambda\equiv (\Id-\partial_x^2)^{1/2}$. It follows
\begin{align}
\frac12 \frac{d}{dt}E^s(U) \ &= \frac12\big( \Lambda^s U,\big[\partial_t, S[\underline U]\big] \Lambda^s U\big)-\big(S[\underline U]\Lambda^sA[\underline U]\partial_xU, \Lambda^s U\big) +
\big( \Lambda^s \R ,S[\underline U] \Lambda^s U\big) \nn \\
&= \frac12\big( \Lambda^s U,\big[\partial_t, S[\underline U]\big] \Lambda^s U\big)+\frac12\big(\big[\partial_x,\Sigma[\underline U]\big] \Lambda^s U, \Lambda^s U\big) +
\big( \Lambda^s \R ,S[\underline U] \Lambda^s U\big) \nn \\
& \qquad -\big(S[\underline U]\big[\Lambda^s,A[\underline U]\big]\partial_xU, \Lambda^s U\big) .
\label{eqn:energyequalityXs}
\end{align}
The first three terms are bounded exactly as above, when replacing $U$ with $\Lambda^s U$. The only novelty lies in the use of continuous Sobolev embeddings, so that
\[ \big\Vert \underline U \big\Vert_{L^\infty([0,T]\times \RR)^4}+\big\Vert \partial_x\underline U \big\Vert_{L^\infty([0,T]\times \RR)^4}\ \lesssim\ \big\Vert \underline{U}\big\Vert_{L^\infty ([0,T];X^{s})} \ .\]
Similarly, one has
\[ \r \big\Vert \partial_t \underline U \big\Vert_{L^\infty([0,T]\times \RR)^4} \ \lesssim\ \r \big\Vert \partial_t \underline{U}\big\Vert_{L^\infty ([0,T];X^{s-1})}.\]

The remaining term is estimated as follows. Using the commutator estimate in Lemma~\ref{L.KatoPonce}, one has
\[ \big\vert \big[\Lambda^s,A[\underline U]\big]\partial_xU\big\vert_{L^2} \ \leq \ C \big\vert \partial_x U\big\vert_{H^{s-1}} \big\vert \big[\partial_x, A[\underline U]\big]\big\vert_{H^{s-1}}\ \leq \ C_0\ M \ \r^{-1}\ \big\vert U\big\vert_{X^{s}} , \]
with $C_0=C(M,h_0^{-1},\delta_{\min}^{-1},\delta_{\max})$. Altogether, one deduces from~\eqref{eqn:energyequalityXs}
\[\frac12 \frac{d}{dt}E^s(U) \ \leq\ C_0M \r^{-1} E^s(U) \ + \ C_0 E^s(U)^{1/2} \big\vert \R \big\vert_{X^s}.\]
Estimate~\eqref{energyestimateHs} is now a consequence of Gronwall-Bihari's inequality, and the Lemma is proved.
\end{proof}

\begin{proof}[Competion of the proof of Proposition~\ref{P.WPFS}] The well-posedness of system~\eqref{FS} is now a consequence of the energy estimates of Lemmata~\ref{L.L2} and~\ref{L.Hs}, following the standard strategy (we let the reader refer to standard textbooks, {\em e.g.}~\cite{TaylorIII,AlinhacGerard,M'etivier08}, for more details). More precisely, one first show that the linearized problem~\eqref{eqn:FSlin} is well-posed, then the solution of the nonlinear problem~\eqref{FS} is obtained as the limit of an iterative scheme:
\[\partial_t U^{n+1} \ + \ A[U^n]\partial_x U^{n+1} \ = \ 0.\]
The restriction on the timescale $t\in[0, T\r]$ is necessary to guarantee that $(U^n)_{n\in\NN}$ is a Cauchy sequence, and in particular that $U^n$ is uniformly bounded with respect to $n$, over time domain independent of $n$. The desired estimate on $\big\vert U\big\vert_{X^s} $ follows directly from Lemma~\ref{L.Hs} with $\u U=U$ and $R\equiv 0$, and the corresponding estimate on $\big\vert \partial_t U\big\vert_{X^s} $ is then deduced using~\eqref{FS}. The uniqueness comes from a similar estimate on the difference between two solutions, and the blow-up criterion as $t\to T_{\max}$ if $T_{\max}<\infty$ follows from standard continuation arguments. This concludes the proof of Proposition~\ref{P.WPFS}.\end{proof}

\bigskip

\noindent {\bf Acknowledgements.} The author is grateful to
Christophe Cheverry, Jean-Fran\c{c}ois Coulombel and Fr{\'e}d{\'e}ric Rousset for helpful advice and stimulating discussions. 
This work has been partially supported by the project ANR-13-BS01-0003-01 DYFICOLTI.

\end{document}